\documentclass{amsart}
\usepackage[utf8]{inputenc}
\usepackage[english]{babel}
\usepackage{amsmath}
\usepackage{amsfonts}
\usepackage{amssymb}
\usepackage{amsthm}
\usepackage{cmap}
\usepackage{tikz-cd}
\usepackage[mathscr]{euscript}
\usepackage{hyperref}

\numberwithin{equation}{section}

\newtheorem{prop}[equation]{Proposition}
\newtheorem*{prop*}{Proposition}
\newtheorem{lmm}[equation]{Lemma}
\newtheorem{thm}[equation]{Theorem}
\newtheorem{varthm}{Theorem}

\newtheorem*{thm*}{Theorem}

\newtheorem{cor}[equation]{Corollary}
\theoremstyle{definition}
\newtheorem{dfn}[equation]{Definition}
\newtheorem{ntt}[equation]{Notation}

\newtheorem*{ntt*}{Notation}
\theoremstyle{remark}
\newtheorem{rmk}[equation]{Remark}

\DeclareMathOperator{\Hom}{Hom}
\DeclareMathOperator{\Ext}{\mathrm{Ext}}
\DeclareMathOperator{\codim}{codim}
\newcommand{\im}{\mathrm{Im}}
\newcommand{\Q}{\mathbb Q}
\newcommand{\Z}{\mathbb Z}
\DeclareMathOperator{\R}{\mathbb R}
\newcommand{\PP}{\mathbb P}
\newcommand{\Co}{\mathbb C}
\newcommand{\Oh}{\mathcal O}
\newcommand{\Lk}{\mathrm{Lk}}
\DeclareMathOperator{\id}{id}
\DeclareMathOperator{\Sing}{Sing}

\DeclareMathOperator{\rk}{\mathrm{rk}}
\newcommand{\aut}{\mathrm{Aut}}
\newcommand{\Spec}{\mathrm{Spec}}

\newcommand{\Sym}{\mathrm{Sym}}
\newcommand{\Gr}{\mathrm{Gr}}
\newcommand{\LGr}{\mathrm{LGr}}
\newcommand{\OGr}{\mathrm{OGr}}
\newcommand{\Pic}{\mathrm{Pic}}

\newcommand{\reg}[1]{\Gamma(#1)_{\mathrm{reg}}}
\newcommand{\langreg}[1]{\mathrm{L}\Gamma(#1)_{\mathrm{reg}}}

\newcommand{\fr}[1]{H^{#1}_{\mathrm{fr}}}

\DeclareMathOperator{\Tot}{Tot}

\title{On the automorphism groups of smooth Fano threefolds}

\date{}
\author{Nikolay Konovalov}
\address{\parbox{\linewidth}{Faculty of Mathematics, HSE University, \\6 Usacheva ulitsa, Moscow 119048, Russia \\
\\
Max Planck Institute for Mathematics, \\Vivatsgasse 7, 53111 Bonn, Germany}}
\email{nikolay.konovalov.p@gmail.com, konovalov@mpim-bonn.mpg.de}

\begin{document}
\maketitle

\begin{abstract}
Let $\mathcal{X}$ be a smooth Fano threefold over the complex numbers of Picard rank~$1$ with finite automorphism group. We give numerical restrictions on the order of the automorphism group $\mathrm{Aut}(\mathcal{X})$ provided the genus $g(\mathcal{X})\leq 10$ and $\mathcal{X}$ is not an ordinary smooth Gushel--Mukai threefold. More precisely, we show that the order $|\mathrm{Aut}(\mathcal{X})|$ divides a certain explicit number depending on the genus of $\mathcal{X}$. We use a classification of Fano threefolds in terms of complete intersections in homogeneous varieties and the previous paper of A.~Gorinov and the author regarding the topology of spaces of regular sections.
\end{abstract}

\tableofcontents

\section*{Introduction}
We work over the field $\Co$ of complex numbers. By abuse of notation, we use the same symbol for an algebraic variety over $\Co$ and for the associated complex manifold of $\Co$-points.

We fix some notation. Let $\mathcal{X}$ be a smooth Fano threefold with canonical divisor $K_{\mathcal{X}}$. In this case, the number
$$g(\mathcal{X}) = -\frac{1}{2}K_{\mathcal{X}}^3 +1$$
is called the \emph{genus} of $\mathcal{X}$. By the Riemann--Roch theorem and the Kawamata--Viehweg vanishing, we have
$$\dim|-\!\!K_{\mathcal{X}}| = g(\mathcal{X})+1, $$
see~e.g.~\cite[Corollary~2.1.14]{AG5}. In particular, $g(\mathcal{X})$ is an integer and $g(\mathcal{X})\geq 2$. Recall that $\Pic(\mathcal{X})$ is a finitely generated torsion free abelian group, so that
$$\Pic(\mathcal{X})\cong \Z^{\rho(\mathcal{X})}. $$
The integer $\rho(\mathcal{X})$ is called the \emph{Picard rank} of $\mathcal{X}$. The maximal number $\iota(\mathcal{X})$ such that the anticanonical line bundle $\omega^{-1}_{\mathcal{X}}$ is divisible by $\iota(\mathcal{X})$ in $\Pic(\mathcal{X})$ is called the \emph{Fano index}, or simply \emph{index}, of $\mathcal{X}$. Let $\mathcal{L}$ be a line bundle such that 
$$ \omega^{-1}_{\mathcal{X}} \cong \mathcal{L}^{\otimes \iota(\mathcal{X})}.$$
The line bundle $\mathcal{L}$ is unique up to isomorphism since $\Pic(\mathcal{X})$ is torsion free. We define the \emph{degree} of $\mathcal{X}$ as
$$d(\mathcal{X}) = \langle c_1(\mathcal{L})^3, [\mathcal{X}] \rangle, $$
where $\langle - , [\mathcal{X}]\rangle$ is the evaluation of a cohomology class on the fundamental class $[\mathcal{X}] \in H_6(\mathcal{X},\Z)$, i.e. the \emph{Kronecker pairing}.

Our goal is to obtain explicit numerical restrictions on the orders of the automorphism groups $\aut(\mathcal{X})$, where $\mathcal{X}$ is a smooth Fano threefold of Picard rank~$1$. We will use the classification of such Fano threefolds, see e.g.~\cite[\S12.2]{AG5} or~\cite{Mukai89}, as well as the previous paper~\cite{GK17} by A.~Gorinov and the author on the automorphism groups of complete intersections in homogeneous varieties. 

We recall that the classification implies that $\iota(\mathcal{X})\leq 4$, and if $\iota(\mathcal{X})=4$ (resp. $\iota(\mathcal{X}) =3$), then $\aut(\mathcal{X})\cong PGL_4(\Co)$ (resp. $\aut(\mathcal{X})\cong PSO_5(\Co)$) is infinite. Next, if $\iota(\mathcal{X})=2$, then $d(\mathcal{X})\leq 5$; and if $d(\mathcal{X})=5$, then $\aut(\mathcal{X})\cong PSL_2(\Co)$ is also infinite, see e.g. Section~\ref{section: sgm}. Moreover, if $\iota(\mathcal{X})=1$, then $2\leq g(\mathcal{X})\leq 12$ and $g(\mathcal{X})\neq 11$. Finally, we note that for $\iota(\mathcal{X})=1$ and $g(\mathcal{X})=3,4,6$, the classification goes further, i.e. the variety $\mathcal{X}$ belongs to a certain subtype.

Our main results are the following two theorems.

\begin{varthm}\label{theorem:A}
Let $\mathcal{X}$ be a smooth Fano threefold of Picard rank~$1$. Suppose that $\iota(\mathcal{X})=2$.
\begin{enumerate}
\item If $d(\mathcal{X})=1$, then $|\aut(\mathcal{X})|$ divides $2^8 \cdot 3^4 \cdot 5^3 \cdot 7$, see~Corollary~\ref{corollary: bound fano, weighted}.
\item If $d(\mathcal{X})=2$, then $|\aut(\mathcal{X})|$ divides $2^{10}\cdot 3^6 \cdot 5 \cdot 7$, see~Corollary~\ref{corollary: automorphism of fano, double covers}.
\item If $d(\mathcal{X})=3$, then $|\aut(\mathcal{X})|$ divides $2^{10}\cdot 3^5 \cdot 5 \cdot 11$, see~Corollary~\ref{corollary: automorphism of fano, small}.
\item If $d(\mathcal{X})=4$, then $|\aut(\mathcal{X})|$ divides $2^{14}\cdot 3^2 \cdot 5$, see~Corollary~\ref{corollary: automorphism of fano, middle}.
\end{enumerate}
\end{varthm}

Among these, only the case $d(\mathcal{X})=1$ seems to be new. The case $d(\mathcal{X})=2$ can be read off~\cite[Theorem~4.5.1]{GK17}, the case $d(\mathcal{X})=3$ was done in~\cite{WY20}, see Remark~\ref{remark: cubic threefold}, and the case $d(\mathcal{X})=4$ is classical, see Remark~\ref{remark: previous estimates degree 4}.

\begin{varthm}\label{theorem:B}
Let $\mathcal{X}$ be a smooth Fano threefold of Picard rank~$1$. Suppose that $\iota(\mathcal{X})=1$.
\begin{enumerate}
\item If $g(\mathcal{X})=2$, then $|\aut(\mathcal{X})|$ divides $2^9\cdot 3^4 \cdot 5^6 \cdot 7 \cdot 13$, see~Corollary~\ref{corollary: automorphism of fano, double covers}.
\item If $g(\mathcal{X})=3$ and the anticanonical line bundle $\omega^{-1}_{\mathcal{X}}$ is very ample, then $|\aut(\mathcal{X})|$ divides $2^{11}\cdot 3^{10} \cdot 5 \cdot 7 \cdot 61$, see~Corollary~\ref{corollary: automorphism of fano, small}.
\item If $g(\mathcal{X})=3$ and the anticanonical line bundle $\omega^{-1}_{\mathcal{X}}$ is not very ample, then $|\aut(\mathcal{X})|$ divides $2^{11}\cdot 3^{4} \cdot 5 \cdot 71$, see~Corollary~\ref{corollary: automorphism of fano, double covers}.
\item If $g(\mathcal{X})=4$ and $\mathcal{X} \subset \PP^5$ is a complete intersection of a \emph{non-singular} quadric and a cubic hypersurface, then $|\aut(\mathcal{X})|$ divides $2^{10}\cdot 3^5 \cdot 5 \cdot 7 \cdot 43$, see~Corollary~\ref{corollary: bound fano, quadric, simple}.
\item If $g(\mathcal{X})=4$ and $\mathcal{X} \subset \PP^5$ is a complete intersection of a \emph{singular} quadric and a cubic hypersurface, then $|\aut(\mathcal{X})|$ divides $2^{11} \cdot 3^3 \cdot 5 \cdot 7^2$, see~Corollary~\ref{corollary: bound fano, quadric}.
\item If $g(\mathcal{X})=5$, then $|\aut(\mathcal{X})|$ divides $2^{24}\cdot 3^2 \cdot 5^2 \cdot 7^2$, see~Corollary~\ref{corollary: automorphism of fano, middle}.
\item If $g(\mathcal{X})=6$ and $\mathcal{X}$ is a double cover of a quintic del Pezzo threefold branched in an anticanonical divisor, then $|\aut(\mathcal{X})|$ divides $2^{7}\cdot 3^3 \cdot 5 \cdot 7$, see~Corollary~\ref{corollary: bound fano special GM}.
\item If $g(\mathcal{X})=7$, then $|\aut(\mathcal{X})|$ divides $2^{9}\cdot 3^5 \cdot 5 \cdot 7$, see~Corollary~\ref{corollary: bound fano ogr}.
\item If $g(\mathcal{X})=8$, then $|\aut(\mathcal{X})|$ divides $2^{12}\cdot 3^2 \cdot 5 \cdot 11$, see~Corollary~\ref{corollary: automorphism of fano, middle}.
\item If $g(\mathcal{X})=9$, then $|\aut(\mathcal{X})|$ divides $2^{11}\cdot 3^4 \cdot 7$, see~Corollary~\ref{corollary: bound fano lgr}.
\item If $g(\mathcal{X})=10$, then $|\aut(\mathcal{X})|$ divides $2^{7}\cdot 3^5 \cdot 5$, see~Corollary~\ref{corollary: bound fano g2}.
\end{enumerate}
\end{varthm}

Among these, only items~(5),~(7), and~(11) seem to be new. The items (8) and~(10) were done in~\cite[Corollary~4.3.5]{KPS18}, see Remarks~\ref{remark: previous estimates ogr} and~\ref{remark: previous estimates lgr}, respectively. All other items can be read off~\cite[Theorem~4.5.1]{GK17}. Note that many of the previous restrictions on the automorphism groups are actually sharper than the ones in Theorems~\ref{theorem:A} and~\ref{theorem:B}; however, our approach requires much less knowledge about the internal geometry of~$\mathcal{X}$, and we obtain our restrictions in a more uniform (but also in  a more computational) way.

We refer the reader to~\cite[Theorem~1.1.2]{KPS18} for a full classification of smooth Fano threefolds of Picard rank~$1$ with finite automorphism group. Theorem~\ref{theorem:B} does not cover two cases of this classification, namely when $\mathcal{X}$ has genus~$12$ and when $\mathcal{X}$ is an ordinary smooth Gushel--Mukai threefold of genus~$6$. We discuss the case $g(\mathcal{X})=12$ in Remark~\ref{remark: genus 12} and the case $g(\mathcal{X})=6$ in Remarks~\ref{remark: ordinary gm fano}--\ref{remark: lagrangian plane genus 6}. The analysis in these cases will be presented elsewhere. We also note that our calculation is independent of~\cite{KPS18}, and so it might be considered as an alternative partial proof of Theorem~1.1.2,~ibid.

We expect that the numbers in Theorems~\ref{theorem:A} and~\ref{theorem:B} are sharp in the following sense. We conjecture that for any prime $p$ occurring in Theorems~\ref{theorem:A} and~\ref{theorem:B} there exists a smooth Fano threefold $\mathcal{X}$ with the corresponding invariants such that $\mathcal{X}$ admits an automorphism of order $p$, cf.~\cite[Section~5.2]{GK17}. This is true if $\mathcal{X}$ is either a hypersurface in the projective space $\PP^4$ or a double cover of $\PP^3$ branched in a hypersurface, see~\cite{GL13}. The other cases seem to be unknown. 

\smallskip

We briefly explain the idea behind the proof of Theorems~\ref{theorem:A} and~\ref{theorem:B}. By the classification (see~\cite[\S12.2]{AG5} and~\cite{Mukai89}), for each smooth Fano threefold $\mathcal{X}$ of Picard rank~$1$, there exists a \emph{$G$-equivariant} vector bundle $E$ over a (not necessarily smooth) $G$-variety $X$ such that $G$ is a linear algebraic group, and either $\mathcal{X}$ is isomorphic to the \emph{zero locus $Z(s)$ of a regular section} $s\in \reg{X,E}$, see Definition~\ref{definition: regular section}, or $\mathcal{X}$ is isomorphic to a \emph{ramified cover} of $X$ branched in the zero locus $Z(s)$, $s\in \reg{X,E}$. In either of these cases one can relate the automorphism group $\aut(\mathcal{X})$ to the \emph{$G$-stabiliser group} $G_{Z(s)}$ of the associated zero locus.

In~\cite[Sections~2.2 and~4.1]{GK17}, A.~Gorinov and the author developed a computational machinery which restricts the orders $|G_{Z(s)}|$ of the stabiliser groups in terms of the $G$-equivariant characteristic classes of a $G$-equivariant vector bundle~$E$ over a $G$-variety $X$ for every regular section $s\in \reg{X,E}$, see Theorem~\ref{maintheorem} and Corollary~\ref{corollary: main corollary} for a quick recap. Theorems~\ref{theorem:A} and~\ref{theorem:B} are proven by applying this machinery case-by-case to the $G$-equivariant vector bundles which appear in the classification of smooth Fano threefolds. We warn the reader that a large part of the calculations is computer assisted, and uses the system Singular~\cite{Singular}.

Finally, we note that there is a similar classification of smooth Fano threefolds of arbitrary Picard rank in terms of the zero loci of regular sections of equivariant vector bundles, see~\cite{BFT22}, as well as a classification of smooth Fano threefolds with finite automorphism group, see~\cite{PCS19}. The analysis in the general case will be presented elsewhere and presumably by other authors, the analysis in the case $\rho(\mathcal{X})=1$ being already tedious enough.

\smallskip

The paper is organized as follows. Section~\ref{section: prelim} is the preliminary part of the paper in which we recall the main ideas from the previous paper~\cite{GK17}. The key results are Theorem~\ref{thmquotslice}, Theorem~\ref{maintheorem}, and Corollary~\ref{corollary: main corollary}. We also obtain numerical restrictions for smooth Fano threefolds which can be obtained as complete intersections either in a projective space, or in a smooth quadric, or in a Grassmann variety $\Gr(k,n)$, see Corollaries~\ref{corollary: automorphism of fano, small},~\ref{corollary: bound fano, quadric, simple},~\ref{corollary: automorphism of fano, middle}, and~\ref{corollary: automorphism of fano, double covers}. We describe $\aut(\mathcal{X})$ for a general Fano threefold $\mathcal{X}$ in Remark~\ref{remark: automorphism groups of general Fano}.

Sections~\ref{section: lagr},~\ref{section: ogr}, and~\ref{section: g2gr} are devoted to the cases $g(\mathcal{X})=9$, $g(\mathcal{X})=7$, and $g(\mathcal{X})=10$, respectively. Since in these cases $\mathcal{X}$ is a linear section in a homogeneous variety $G/P$, we closely follow the general recipe from~\cite[Section~4.1]{GK17}. In Section~\ref{section: weighted} we apply this recipe with a few modifications to obtain restrictions on the automorphism group of a smooth hypersurface of degree~$6$ in the weighted projective space $\PP(1,1,1,2)$. In Section~\ref{section: quadric}, we deal with a hypersurface of degree~$3$ in a singular quadric of dimension~$4$. Finally, in Section~\ref{section: sgm} we consider hypersurfaces in a quintic del Pezzo threefold.

\medskip

\noindent {\bf Acknowledgments. }The author is grateful to Alexey Gorinov and Sasha Kuznetsov for many helpful discussions, and thanks Ivan Cheltsov for pointing out a mistake in Remarks~\ref{remark: previous estimates degree 4} and~\ref{remark: previous estimates g2} of an earlier version of the paper. The author appreciates the hospitality and financial support of the Max Planck Institute for Mathematics in Bonn. 

\section{Preliminaries}\label{section: prelim}
\subsection{Regular sections}
Throughout this section $E$ is an algebraic vector bundle of rank $r$ over an algebraic complex proper variety $X$ of dimension $d$. We assume that $X$ is smooth if not stated otherwise.

\begin{dfn}[Notation~2.2.2,~\cite{GK17}]\label{definition: regular section} Suppose the variety $X$ is Cohen--Macaulay (not necessarily smooth). A section $s\in \Gamma(X,E)$ is \emph{regular} if the scheme-theoretic zero locus $Z(s)$ of $s$ is a smooth subvariety of $X$ of dimension~$d-r$.
\end{dfn}

Let $\reg{X,E} \subset \Gamma(X,E)$ denote the subset of regular sections with the induced topology. We note that, if $X$ is smooth, then a section is regular if and only if $s$ is transversal to the zero section. Moreover, if $X$ is smooth and $E$ is globally generated, then $\reg{X,E}\neq \emptyset$.

\begin{ntt}\label{notation: jet bundle}
Let $J(E)$ denote the \emph{(first) jet bundle} of $E$, see~\cite[Chapter~16.7]{EGAIV}. 
\end{ntt}

We recall that there exists a short exact sequence
\begin{equation}\label{sesjetbundle}
0 \to \Omega_{X} \otimes E \to J(E) \to E \to 0,
\end{equation}
where $\Omega_X$ is the (algebraic) cotangent bundle of $X$. Moreover, there is a linear map
\begin{equation}\label{equation: jet of a section}
j\colon \Gamma(X,E) \hookrightarrow \Gamma(X,J(E)),
\end{equation}
which informally sends a section $s\in\Gamma(X,E)$ to the Jacobian matrix of the map $s\colon X\to \Tot(E)$ from the smooth variety $X$ to the total space of $E$. In particular, $s\in \Gamma(X,E)$ is regular if and only if $j(s)$ is \emph{nowhere vanishing}, i.e. $j(s)(x)\neq 0$ for all $x\in X$. 

Let $\PP(E^*)$ be the projectivisation of the dual vector bundle $E^*$ and let $$\pi\colon \PP(E^*)\to X$$ be the projection. The direct image $\pi_*(\Oh_{\PP(E^*)}(1))$ of the relative twisting line bundle $\Oh_{\PP(E^*)}(1)$ is canonically isomorphic to the vector bundle~$E$ itself. So there is an isomorphism
\begin{equation*}
\Gamma\pi_*\colon \Gamma(\PP(E^*),\Oh_{\PP(E^*)}(1)) \xrightarrow{\cong} \Gamma(X,E),
\end{equation*}
which induces a natural isomorphism
\begin{equation}\label{equation: Cayley trick}
\Gamma\pi_*\colon \reg{\PP(E^*),\Oh_{\PP(E^*)}(1)} \xrightarrow{\cong} \reg{X,E}
\end{equation}
on the subspaces of regular sections. We will often refer to the last isomorphism as the \emph{Cayley trick}.

\subsection{Orbit map}
Let $G$ be a connected complex affine group which acts algebraically on the variety $X$. Suppose that $E$ is a $G$-equivariant vector bundle over~$X$. Then $G$ acts continuously on the space $\reg{X,E}$. Fix a section $s_0\in \reg{X,E}$ and set 
\begin{equation}\label{equation: orbit map}
O\colon G \to \reg{X,E}, \;\; g\mapsto g\cdot s_0
\end{equation}
to be the orbit map. We note that the homotopy class of the orbit map does \emph{not} depend on the choice of $s_0$ since the space $\reg{X,E}$ is path-connected (if it is non-empty).

\begin{thm}\label{thmquotslice}
Suppose that $G$ is a connected complex reductive group, $U=\reg{X,E}$ is affine, where $X$ is not necessarily smooth, and the induced map
$$O^*\colon H^*(U,\Q) \to H^*(G,\Q) $$
is surjective. Then the following is true.
\begin{enumerate}
\item The stabiliser $G_s$ of every point $s\in U$ is finite and the geometric quotient of $U$ by $G$ exists and is affine.
\item If there is a class $a\in H^k(U,\Z)$, where $k$ is the dimension of a maximal compact subgroup of $G$, such that $O^*(a)$ generates a subgroup of $H^k(G,\Z)\cong\Z$ of index $m$, then for every $s\in U$ the order of the stabiliser $G_s$ divides~$m$.
\end{enumerate}
\end{thm}

\begin{proof}
See Theorem~3.1.1 and Proposition~3.2.1 in~\cite{GK17}.
\end{proof}

\begin{prop}\label{proposition: discriminant is codim 1}
Suppose that $E=L$ is a very ample line bundle. If the Chern number $\langle c_{\dim(X)}(J(L)),[X]\rangle$ is non-zero, then the variety $\reg{X,L}$ is affine. 
\end{prop}

\begin{proof}
Since $L$ is a very ample line bundle, the image of the map~\eqref{equation: jet of a section} globally generates the jet bundle $J(L)$, see~\cite[Proposition II.7.3]{Hartshorne}. Therefore, the discriminant
$$\Delta = \Gamma(X,L) \setminus \reg{X,L} $$
is an irreducible variety. By~\cite[Proposition~3.2.17]{GK17} (see also~\cite[Theorem 3.3.10]{GKZ08}), $\langle c_{\dim(X)}(J(L)),[X]\rangle \neq 0$ yields that $\Delta$ is a divisor.
\end{proof}

In the previous paper~\cite[Definition~2.2.6]{GK17}, A.~Gorinov and the author constructed a non-trivial linking class homomorphism
\begin{equation}\label{equation: linking class}
\Lk\colon H_p(\PP(E^*),\Z) \to H^{2(r+d)-p-1}(\reg{X,E},\Z), \;\; p\geq 0, 
\end{equation}
which is our main source of non-trivial cohomology classes in $H^*(\reg{X,E},\Z)$. We briefly explain the construction in the case where $E=L$ is a line bundle, then the general case can be deduced by the Cayley trick. Let $E'=J(L)$ denote the jet bundle and let $\Tot_0(E')$ denote the total space of $E'$ minus the zero section. Note that the Euler class of $E'$ is zero for dimension reasons, so there exists a cohomology class $a_{E'}\in H^{2d+1}(\Tot_0(E'),\Z)$ which restricts to a generator over each fibre. For $y\in H_*(X,\Z)$ we then set $$\Lk(y)=-j_{ev}^*(a_{E'})/y,$$ where $${j_{ev}}\colon\reg{X,L}\times X\to \Tot_0(E')$$ is the jet evaluation map defined by ${j_{ev}}(s,x)=j(s)(x)$, and $/$ denotes the slant product $$H^p(\reg{X,L}\times X,\Z)\otimes H_q(X,\Z)\to H^{p-q}(\reg{X,L},\Z).$$
We will explain how to calculate $O^*(\Lk(y))$, $y\in H_*(\PP(E^*),\Z)$ in Theorem~\ref{maintheorem}.

\subsection{Equivariant cohomology} Recall that the \emph{universal $G$-bundle} $EG$ is a contractible space with a free continuous right $G$-action such that $EG \to BG=EG/G$ is a locally trivial fibre bundle. We denote by $H^*_{G}(X',R)=H^*(X'_{hG},R)$ the \emph{equivariant cohomology} of a $G$-space $X'$ with coefficients in a ring $R$, that is the cohomology of the \emph{homotopy quotient} $$X'_{hG}=EG\times_{G} X'.$$ 
For a more detailed account of equivariant cohomology we refer the reader e.g.\ to~\cite[Part I]{Tu20} or~\cite[Chapter 2]{AF24}.

We write
\begin{equation}\label{equation: restriction map}
\alpha^*\colon H^*_G(X',R) \to H^*(X',R)
\end{equation}
for the restriction map induced by the quotient map $\alpha\colon X' \to X'_{hG}$ (which depends on a choice of a point in $EG$, but its homotopy class is defined canonically) and 
\begin{equation}\label{equation: structure map}
\beta^*\colon H^*(BG,R) \to H^*_G(X',R)
\end{equation}
for the structure map induced by the projection $\beta\colon X'_{hG} \to BG$. Finally, let $\bar\gamma\colon H^*(BG,R) \to H^{*-1}(G,R)$ be the \emph{cohomology suspension map} (see~\cite[Notation~1.2.12]{GK17} or~\cite[\S6.2]{Mccleary01}), i.e. the composite
\begin{equation}\label{equation: transgression map}
\bar\gamma\colon H^*(BG,R) \xrightarrow{\gamma^*} H^*(\Sigma G,R) \xrightarrow{\cong} H^{*-1}(G,R),
\end{equation}
where $\gamma^*$ is induced by the map $\Sigma G \to BG$ obtained from the canonical homotopy equivalence $G \simeq \Omega BG$ by the suspension-loop space adjunction. We note that $\bar\gamma$ is a right inverse to the \emph{transgression map} $d_p\colon E_p^{0,p-1} \to E_p^{p,0}$ in the Leray--Serre spectral sequence
$$E^{p,q}_2 = H^p(BG,H^q(G,R)) \Rightarrow H^{p+q}(EG,R), \;\; d_r\colon E^{p,q}_r \to E^{p+r,q-r+1} $$
for the fibre bundle $G\to EG \to BG$, see~\cite[Proposition~6.10]{Mccleary01} and~\cite[Remark~1.2.13]{GK17}.

\begin{lmm}\label{lemma: smooth proper0}
Suppose that $X'$ is a smooth proper $G$-variety. Then the Leray--Serre spectral sequence 
$$E^{p,q}_2 = H^p(BG,\Q)\otimes H^q(X',\Q) \Rightarrow H^{p+q}_G(X',\Q)$$
of the fibre bundle $X' \to X'_{hG} \to BG$ degenerates at the $E_2$-term over~$\Q$. In particular, there exists a non-canonical isomorphism
$$ H^*(BG,\Q)\otimes H^*(X',\Q) \cong H^{*}_G(X',\Q)$$
of $H^*(BG,\Q)$-modules. 
\end{lmm}

\begin{proof}
Follows from the Hodge theory, cf.~\cite{HodgeIII}. See~\cite[Lemma~1.3.1]{GK17} for details.
\end{proof}

\begin{cor}\label{corollary: smooth proper}
Suppose that $X'$ is a smooth proper $G$-variety. Then the sequence
$$\widetilde{H}^*(BG,\Q) \otimes H^*_G(X',\Q) \to H^*_G(X',\Q) \xrightarrow{\alpha^*} H^*(X',\Q) \to 0 $$
is exact. Here, $\widetilde{H}^*(BG,\Q)=\bigoplus_{p \geq 1} H^p(BG,\Q)$ are reduced cohomology groups of $BG$ and the first map is given by the rule $a\otimes b \mapsto \beta^*(a)b$. \qed
\end{cor}

\begin{dfn}\label{definition: mapS}
Suppose that $X'$ is a smooth proper $G$-variety. We define a bilinear map
$$S\colon \ker(\alpha^*) \times H_*(X',\Q) \to H^*(G,\Q) $$
as follows. By Corollary~\ref{corollary: smooth proper}, an element $x \in \ker(\alpha^*) \subset H^{q}_G(X',\Q)$ has a decomposition $x=\sum_i\beta^*(a_i)b_i$; we then set 
$$S(x,y)=\sum_i \langle \alpha^*(b_i), y\rangle \bar\gamma(a_i) \in H^{q-p-1}(G,\Q), \;\; y\in H_p(X',\Q).$$
By~\cite[Sections~1.2--1.3]{GK17}, the map $S$ is well-defined, i.e. does not depend on the choice of a decomposition of $x$.
\end{dfn}

\begin{ntt}
If $X'$ is a topological space, we denote the ring $H^*(X',\Z)/\rm{torsion}$ by $\fr{*}(X',\Z)$.
\end{ntt}

Recall that the multiplication $m\colon G\times G \to G$ induces the Hopf algebra structure on the cohomology ring $H^*(G,\Q)$. We write $P^*_{\Q}$ for the graded vector subspace of the primitive elements in the Hopf algebra $H^*(G,\Q)$. Since $G$ is a manifold, $P^*_{\Q}$ is concentrated in odd degrees and freely generates $H^*(G,\Q)$ as a graded commutative $\Q$-algebra. We note that the integral analogue of this is also true: the graded group $P^*=P^*_{\Q}\cap\fr{*}(G,\Z)$ of the primitive elements of $\fr{*}(G,\Z)$ freely generates $\fr{*}(G,\Z)$ as a graded commutative ring, see e.g.~\cite[Theorem VII.1.22]{MT91}.

\begin{prop}\label{proposition: imageofs}
The subspace of primitive elements $P^*_{\Q}\subset H^*(G,\Q)$ is exactly the image of the map $$\bar\gamma\colon H^*(BG,\Q) \to H^{*-1}(G,\Q).$$ In particular, $S(x,y) \in P^*_{\Q}$ for every $x\in \ker(\alpha^*)$ and $y\in H_*(X',\Q)$. 
\end{prop}

\begin{proof}
See~\cite[Theorem~VII.2.12 and Proposition~VII.2.27]{MT91}.
\end{proof}

\begin{dfn}[Chapter~2.3, \cite{AF24}]\label{definition: equivariant Euler class}
	Let $E'$ be a complex $G$-equivariant vector bundle over a $G$-space $X'$, $\rk(E')=r$. Then the homotopy quotient $E'_{hG}$ is a complex vector bundle over the quotient $X'_{hG}$. We define the \emph{$i$-th equivariant Chern class} $c_i^G(E')$ of $E'$ as follows
$$c_i^G(E') = c_i(E'_{hG}) \in H^{2i}_G(X',\Z). $$
We define the \emph{equivariant Euler class} $e_G(E')$ as $e_G(E')=c_r^G(E')$.
\end{dfn}

\begin{thm}\label{maintheorem}
Set $E'=J(\Oh_{\PP(E^*)}(1))$ to be the jet bundle of the twisting line bundle $\Oh_{\PP(E^*)}(1)$ over the projectivisation $\PP(E^*)$. Suppose that the variety $X$ is smooth. Then there exists a decomposition
$$e_G(E') = \sum_i \beta^*(a_i)b_i, $$
where $a_i \in \widetilde{H}^*(BG,\Q)$, $b_i \in H^*_G(\PP(E^*),\Q)$. Moreover, 
\begin{equation}\label{orbit_prelim}
	O^*(\Lk(y)) = S(e_G(E'),y) = \sum_i \langle \alpha^*(b_i), y\rangle \bar\gamma(a_i) \in H^*(G,\Q) 
\end{equation}
for a homology class $y\in H_*(\PP(E^*),\Z)$.
\end{thm}

\begin{proof}
This is the main result of~\cite{GK17}, see Theorem~A, ibid. See also Corollary~2.2.12 and Corollary~1.3.5, ibid.
\end{proof}

\begin{rmk}\label{remark: interger coefficients}
We note that the right hand side of the equation \eqref{orbit_prelim} is integral, because the left hand side is, but the individual ingredients of the right hand side may not be.
\end{rmk}

\subsection{Cohomology of \texorpdfstring{$BG$}{BG}} Let $H$ be a complex algebraic group. We recall that any (complex) $H$-representation $V$ can be considered as an $H$-equivariant (complex) vector bundle over a point. 
\begin{ntt}\label{notation: character group}
If $H$ is a complex algebraic group, then we denote its {\it character group} $\Hom(H,\Co^\times)$ by $\mathfrak{X}(H)$.
\end{ntt}

Let $T$ be a maximal torus of $G$. There is an isomorphism between the character group $\mathfrak{X}(T)$ and $H^2(BT,\Z)$ that takes a character $\chi:T\to\Co^\times$ to the $T$-equivariant Chern class $c^T_1(\chi)$ of the $T$-equivariant line bundle $\chi$ over a point.
Let $\mathfrak{t}$ be the Lie algebra of $T$. There is an embedding $$\mathfrak{X}(T)\to \mathfrak{t}^*=\mathop{\mathrm{Hom}}(\mathfrak{t},\Co)$$ that takes $\chi\in\mathfrak{X}(T)$ to the differential $d\chi$ at the identity element. This gives us an integral structure on $\mathfrak{t}^*$. We identify $H^*(BT,\Co)$ with the symmetric {$\Co$\nobreakdash-algebra} $\mathrm{Sym}^*(\mathfrak{t}^*)$ on $\mathfrak{t}^*$ and $H^*(BT,\Z)$ with the symmetric $\Z$-algebra $\mathrm{Sym}^*(\mathfrak{X}(T))$ on the lattice $\mathfrak{X}(T)\subset \mathfrak{t}^*$.

\begin{prop}\label{proposition: Weyl-invariants}
	Let $W=N_G(T)/Z_G(T)$ be the Weyl group of $G$. Then the structure map~\eqref{equation: structure map} 
$$\beta^*\colon H^*(BG,\Q)\to H^*_G(G/T,\Q) \cong H^*(BT,\Q)$$ induces an isomorphism between the cohomology ring $H^*(BG,\Q)$ and the ring of invariants $H^*(BT,\Q)^W$. \qed
\end{prop}
\begin{rmk}\label{remark: abuse with gamma}
In the sequel we will often identify an element $s\in H^*(BG,\Q)$ with its image in $H^*(BT,\Q)$. In particular, for $x\in H^*(BT,\Q)^W$ we will write $\bar\gamma(x)$ instead of $\bar\gamma((\beta^*)^{-1}(x))$.
\end{rmk}

\begin{ntt}\label{notation: elementary poly} We denote the $i$-th elementary symmetric polynomial in $x_1,\ldots, x_n$ by $\sigma_i(x_1,\ldots, x_n)$. 
\end{ntt}

\begin{rmk}\label{remark: supply for generators}
The elementary symmetric polynomials and their modifications are our main source for constructing the generators in the ring of Weyl-invariants $H^*(BT,\Q)^W$, see~\cite[Section~4.1.1]{GK17}.
\end{rmk}

\begin{ntt}\label{notation: basis for diagonal}
We denote the set of all diagonal square complex matrices by $D\subset \mathrm{Mat}_{n\times n}(\Co)$ and let $\varepsilon_i:D\to\Co$ be the map that takes a matrix in $D$ to the $i$-th entry on the diagonal.
\end{ntt}

\begin{lmm}\label{lemma: basis for diagonal}
Let $T=GL_n(\Co)\cap D$, the group of all invertible diagonal matrices. Then the Lie algebra $\mathfrak{t}$ of $T$ is $D$, and the elements $\varepsilon_1,\ldots, \varepsilon_n$ generate the integral lattice $\mathfrak{X}(T)\subset\mathfrak{t}^*$ and freely generate the algebra $\mathrm{Sym}^*(\mathfrak{X}(T))=H^*(BT,\Z)$. \qed
\end{lmm}

\subsection{Automorphism groups} 
\begin{ntt}\label{notatiton: stabiliser} If $H$ is a topological group acting on a space $X'$ and $Z\subset X'$, then we denote the subgroup of $H$ which preserves $Z$ (not necessarily pointwise) by~$H_Z$.
\end{ntt}

\begin{ntt}\label{notation: automorphisms of vector bundle}
If $p\colon E\to X$ is a complex vector bundle, then let $\aut_X(E)$ be the group of all couples $(f,g)$ where $g\in\aut(X)$ and $f$ is a fibrewise automorphism of $E$ that covers $g$. Let $\aut(E/X)\subset \aut_X(E)$ be the normal subgroup of all couples $(f,\id_X)$.
\end{ntt}

\begin{ntt}\label{tildeg}
Suppose a group $G$ acts on $X$. Let $\widetilde{G}$ be the fibre product $$\widetilde{G}=G\times_{\aut(X)}\aut_X(E),$$ and let $p\colon \widetilde{G} \to G$ be the projection map.
\end{ntt}
\begin{lmm}\label{tildeGandG}
If $E$ is a $G$-equivariant vector bundle, then there exists an exact sequence of groups
\begin{equation*}
1\to \aut(E/X) \to \widetilde{G} \xrightarrow{p} G \to 1.
\end{equation*}
Moreover, this exact sequence naturally splits. 
\end{lmm}

\begin{proof}
The construction of all maps in the exact sequence is straightforward. Moreover, the $G$-equivariant structure on $E$ allows one to lift the map $G\to \aut(X)$ to a homomorphism $G\to \aut_X(E)$, which gives us a splitting $G\to \widetilde{G}$.
\end{proof}

\begin{rmk}\label{central}
The group $\aut(E/X)$ always contains the central subgroup of scalar automorphisms isomorphic to $\Co^{\times}$. If $\aut(E/X)$ coincides with this subgroup (e.g. $E=L$ is a line bundle), then by Lemma~\ref{tildeGandG} we have $\widetilde{G}\cong\Co^\times \times G$.
\end{rmk}

We note that the group $\widetilde{G}$ acts on $\reg{X,E}$ and the action of $G$ on $\reg{X,E}$ induced by the splitting $G\to\widetilde{G}$ is the same as the one given by the $G$-equivariant structure on $E$. Let $s\in \reg{X,E}$ be a regular section. Then the projection $\widetilde{G} \to G$ induces a homomorphism $$p_s\colon \widetilde{G}_s \to G_{Z(s)}.$$

\begin{prop}\label{Sum of line bundles}
Suppose that $X$ is a smooth and projective $G$-variety. Let $L$ be an ample $G$-equivariant line bundle over $X$ and suppose that $E=L^{\oplus r}$, $r=\rk(E)\leq \dim(X)$. Then for any $s\in \reg{X,E}$ the map $p_s$ is an isomorphism. Moreover, if $r=1$, then the map $p_s$ is an isomorphism even if~$L$ is not ample.
\end{prop}

\begin{proof}
See~\cite[Corollary~3.2.13]{GK17} for the general case and Lemma~3.2.8, ibid. for the case $r=1$.
\end{proof}

Let $G$ be a connected complex reductive group that acts on a smooth proper complex variety $X$ of dimension $d$ and let $L$ be a $G$-equivariant line bundle over~$X$. Set $E'=J(L)$ to be the jet bundle of $L$ and recall that $e_G(E') \in H^*_G(X,\Q)$ is the equivariant Euler class of $E'$. Then, by Proposition~\ref{proposition: imageofs}, the image of 
$S(e_{G}(E'),y)$ in $\fr{*}(G,\Z)$ is primitive for every $y\in H_*(X,\Z)$. Let $i_l$ be the order of the cokernel of the map 
\begin{equation}\label{yetanothers_original}
H_{2d-2(l-1)}(X,\Z)\rightarrow P^{2l-1}, \; y\mapsto S(e_{G}(E'),y),
\end{equation}
where $P^*\subset \fr{*}(G,\Z)$ denotes the graded group of primitive elements.

\begin{cor}\label{corollary: main corollary}
Suppose that the space of regular sections $\reg{X,L}$ is affine. Then the order $|G_s|$ of the stabiliser $G_s$ divides the product $\prod i_l$ for every regular section $s\in \reg{X,L}$. Moreover, if $H^1(G,\Q)=0$, then the order $|G_{Z(s)}|$ divides the product $\langle c_d(E'),[X]\rangle \cdot \prod i_l$.
\end{cor}

\begin{proof}
We only briefly sketch the proof here, we refer to Theorem~3.3.1 and Corollary~3.3.5 in~\cite{GK17} for a detailed account. For the last part, see Corollary 3.2.16, ibid. 

By~\cite[Theorem~VII.1.22]{MT91}, we have 
$$\fr{*}(G,\Z)\cong \Lambda_{\Z}[x^{(l)}_i\; | \; l \in J, i \in I_l] $$
is a free exterior algebra (over $\Z$) generated by primitive elements $x^{(l)}_i \in P^{2l-1}$, $i\in I_l$, $l\in J$. Here, $J$ and $I_l$, $l\in J$ are finite sets. Let $k$ be the dimension of a maximal compact subgroup of $G$. Then $H^k(G,\Z)\cong \Z$ is generated by the product 
$$b=\prod_{l\in J} \prod_{i\in I_l} x^{(l)}_i.$$
Let $O\colon G\to U=\reg{X,L}$ be the orbit map. By Theorem~\ref{thmquotslice}, it suffices to find a cohomology class $a\in H^k(U,\Z)$ such that $O^*(a)=\left(\prod i_l\right) b$. By Theorem~\ref{maintheorem} and the definition of the numbers $i_l$, $l\in J$, there are homology classes $y^{(l)}_i \in H_*(X,\Z)$, $i\in I_l$ such that
$$\prod_{i\in I_l}O^*(\Lk(y^{(l)}_i)) = \pm i_l \prod_{i\in I_l}x^{(l)}_i.$$
Therefore, we choose $a = \prod_{l\in J}\prod_{i\in I_l}O^*(\Lk(y^{(l)}_i)) \in H^k(U,\Z)$ for Theorem~\ref{thmquotslice}.
\end{proof}

\begin{rmk}\label{remark: main corollary cayley trick}
Using the Cayley trick~\eqref{equation: Cayley trick}, we can extend the previous corollary to a vector bundle $E$ of arbitrary rank. Namely, we replace $X$ with the projectivisation $\PP(E^*)$ and $L$ with $\Oh_{\PP(E^*)}(1)$. We also note that one can use any connected reductive subgroup of $\widetilde{G}$ instead of $G$.
\end{rmk}

\begin{cor}\label{corollary: projective hypersurface}
The order of the projective automorphism group of every smooth {degree~$d$ hypersurface of $\PP^n(\Co)$} divides
\begin{equation}\label{prototype2}
(d-1)^n \prod_{i=2}^{n+1} ((d-1)^{n+1}+(-1)^{i+1}(d-1)^{n+1-i}).
\end{equation}
\end{cor}

\begin{proof}
See Theorem~4.5.1 and Remark~4.5.6 in~\cite{GK17}.
\end{proof}

\begin{cor}\label{corollary: automorphism of fano, small}
Let $\mathcal{X}$ be a smooth Fano threefold of Picard rank 1. 
\begin{enumerate}
\item 
If $\mathcal{X}$ is of index $2$ and degree $3$, then $|\aut(\mathcal{X})|$ divides $2^{10} \cdot 3^5 \cdot 5 \cdot 11$.
\item
	If $\mathcal{X}$ is of index $1$, genus $3$, and the anticanonical line bundle $\omega^{-1}_{\mathcal{X}}$ is very ample, then $|\aut(\mathcal{X})|$ divides $2^{11} \cdot 3^{10} \cdot 5 \cdot 7 \cdot 61$.
\end{enumerate}
\end{cor}

\begin{proof}
By~\cite[\S12.2]{AG5}, $\mathcal{X}$ is a smooth hypersurface of degree $3$ in $\PP^4$ in the first case and a smooth hypersurface of degree $4$ in $\PP^4$ in the second case.
\end{proof}

\begin{rmk}\label{remark: cubic threefold}
By~\cite{WY20}, the least common multiple of $|\aut(\mathcal{X})|$ over all smooth cubic hypersurfaces $\mathcal{X}$ in $\PP^4$ (i.e. over all smooth Fano threefolds $\mathcal{X}$ of index~$2$ and degree~$3$) is $2^{4} \cdot 3^5 \cdot 5 \cdot 11$.
\end{rmk}

\begin{cor}\label{corollary: hypersurface in quadric}
Let $Q_k \subset \PP^{k+1}(\Co)$ be a non-singular quadratic hypersurface of dimension $k$ and let $Z\subset Q_k$ be a smooth intersection of $Q_k$ and a hypersurface of degree $d\geq 2$.  
Then the order $|\aut(Q_k)_Z|$ divides
\begin{equation}\label{prototype3}
2^{\left\lfloor\frac{3n}{2}\right\rfloor+ \varepsilon}\left(\sum_{i=0}^{k} (i+1)(d-1)^i\right)\prod_{i=1}^{n} \frac{(d-1)^{k+2}-(d-1)^{k-2(i-1)}}{d-2},
\end{equation}
where $n=\lfloor \frac{k}{2}\rfloor +1$ and $\varepsilon =1$ for $k$ odd and zero for $k$ even.
\end{cor}

\begin{proof}
See Theorem~4.5.1 and Remark~4.5.6 in~\cite{GK17}.
\end{proof}

\begin{cor}\label{corollary: bound fano, quadric, simple}
Let $\mathcal{X}$ be a smooth Fano threefold of Picard rank 1, index 1 and genus 4 such that the anticanonical embedding of $\mathcal{X}$ into $\PP^5$ is a complete intersection of a non-singular quadric $Q_4$ and a cubic hypersurface. Then $|\aut(\mathcal{X})|$ divides $2^{10} \cdot 3^5 \cdot 5 \cdot 7 \cdot 43$.
\end{cor}

\begin{proof}
Let $\phi\colon \mathcal{X} \hookrightarrow \PP^5$ be the regular embedding defined by the anticanonical line bundle. By~\cite[Proposition~4.1.12]{AG5}, the image of $\phi$ is a complete intersection in $\PP^5$ of a quadratic hypersurface $Q$ and a cubic hypersurface $R$; in particular, there exists a Koszul resolution
$$0\to \Oh_{\PP^5}(-6) \to \Oh_{\PP^5}(-2) \oplus \Oh_{\PP^5}(-3) \to \Oh_{\PP^5} \to \phi_*\Oh_\mathcal{X} \to 0. $$
Therefore, the quadric $Q$ is uniquely determined by the Fano variety $\mathcal{X}$, and $\aut(\mathcal{X}) = \aut(Q)_{Q\cap R}$. The assertion follows now by Corollary~\ref{corollary: hypersurface in quadric}.
\end{proof}

\begin{cor}\label{corollary: automorphism of fano, middle}
Let $\mathcal{X}$ be a smooth Fano threefold of Picard rank 1. 
\begin{enumerate}
\item
If $\mathcal{X}$ is of index $2$ and degree $4$, then $|\aut(\mathcal{X})|$ divides $2^{14} \cdot 3^2 \cdot 5$.
\item
If $\mathcal{X}$ is of index $1$ and genus $5$, then $|\aut(\mathcal{X})|$ divides $2^{24} \cdot 3^{2} \cdot 5^2 \cdot 7^2$.
\item If $\mathcal{X}$ is of index $1$ and genus $8$, then $|\aut(\mathcal{X})|$ divides $2^{12} \cdot 3^{2} \cdot 5 \cdot 11$.
\end{enumerate}
\end{cor}

\begin{proof}
In the first case, by~\cite[\S12.2]{AG5}, $\mathcal{X}$ is a smooth complete intersection of two quadrics in $\PP^5$. Therefore, by~\cite{MM63}, $\aut(\mathcal{X})=PGL_6(\Co)_{Z(s)}$ for some regular section $s\in \reg{\PP^5,\Oh(2)^{\oplus 2}}$. The order $|PGL_6(\Co)_{Z(s)}|$ was restricted in~\cite[Theorem~4.5.1, Line~8]{GK17}, see also Appendix~B.2, Table~12, ibid. 

Similarly, in the second case, $\mathcal{X}$ is a smooth complete intersection of three quadrics in $\PP^6$. Therefore, $\aut(\mathcal{X})=PGL_7(\Co)_{Z(s)}$ for some $s\in \reg{\PP^6,\Oh(2)^{\oplus 3}}$. The order $|PGL_7(\Co)_{Z(s)}|$ was restricted in~\cite[Theorem~4.5.1, Line~9]{GK17}, see also Appendix~B.2, Table~14, ibid. 

Finally, in the last case, $\mathcal{X}$ is a linear section of the Grassmann variety $X=\Gr(2,6)$ in $\PP^{14}$. By~\cite[Remark~4.5.6]{GK17}, $\aut(\mathcal{X})=PGL_6(\Co)_{Z(s)}$ for some regular section $s\in \reg{\Gr(2,6),\Oh_X(1)^{\oplus 5}}$. The order $|PGL_6(\Co)_{Z(s)}|$ was restricted in~\cite[Theorem~4.5.1, Line~11]{GK17}, see also Section~4.4, ibid. 
\end{proof}

\begin{rmk}\label{remark: previous estimates degree 4}
The first part of Corollary~\ref{corollary: automorphism of fano, middle} can be obtained more geometrically, cf. Remarks~\ref{remark: previous estimates lgr} and~\ref{remark: previous estimates ogr}, and in a sharper form. Indeed, in this case $\mathcal{X}\subset \PP^5$ is a smooth intersection of two quadrics, so in suitable coordinates
$$ \mathcal{X} = \Big\{ \sum\nolimits_{i=0}^{5} x_i^2 = \sum\nolimits_{i=0}^{5}\alpha_i x_i^2 = 0 \Big\}, \;\; \alpha_0,\ldots,\alpha_5\in \Co \ \text{ pairwise distinct}, $$
see e.g.~\cite[Lemma~8.6.1]{Dolg12}. This pencil $\Xi\cong \PP^1$ of quadrics through $\mathcal{X}$ is canonical, and its six singular members form the subset $\Lambda = \{\alpha_0,\ldots,\alpha_5\}\subset \Xi$. Arguing as in~\cite[Theorem~8.6.8]{Dolg12}, this yields a short exact sequence
$$ 1 \to \Gamma \to \aut(\mathcal{X}) \to \aut(\Xi;\Lambda) \to 1, $$
where $\aut(\Xi;\Lambda)\subset \mathrm{PGL}_2(\Co)$ is the stabiliser of $\Lambda$, and $\Gamma$ is the group of projective transformations preserving each member of the pencil. In particular, $\Gamma$ is the group of diagonal sign changes $\Gamma = (\Z/2)^6/\{\pm 1\}\cong (\Z/2)^5$. A finite subgroup of $\mathrm{PGL}_2(\Co)$ preserving a set of six distinct points of $\PP^1$ has order $1$, $2$, $3$, $4$, $5$, $6$, $8$, $12$ or $24$, see e.g.~\cite[Section~8.8.4]{Dolg12}; the least common multiple of these orders is $2^3\cdot 3\cdot 5$. Therefore, $|\aut(\mathcal{X})|$ divides $2^8\cdot 3\cdot 5$.
\end{rmk}

\subsection{Double covers}
\begin{prop}\label{proposition: automorphisms of double cover}
Let $\phi\colon \mathcal{X} \to \mathcal{Y}$ be a finite surjective morphism of degree~$2$ between normal algebraic varieties.  Suppose that the Weil divisor class group $Cl(\mathcal{Y})$ has no $2$-torsion and the morphism $\phi$ is $\aut(\mathcal{X})$-equivariant. Then there exists a short sequence

$$0 \to \Z/2 \to \aut(\mathcal{X}) \to \aut(\mathcal{Y})_{B} \to 1, $$

where $\aut(\mathcal{Y})_{B} \subset \aut(\mathcal{Y})$ is the subgroup which preserves the branch locus $B\subset \mathcal{Y}$ of $\phi$.
\end{prop}

\begin{proof}
This is essentially~\cite[Lemma~4.4.1]{KPS18}. We sketch the argument here for the reader's convenience. Since $\phi$ is of degree~$2$, there exists an involution $\sigma\colon \mathcal{X} \to \mathcal{X}$ such that $\mathcal{Y}\cong \mathcal{X}/\langle \sigma \rangle$ and $\phi$ is the quotient map. In particular, the kernel of the natural map
\begin{equation}\label{equation: double cover_eq1}
\aut(\mathcal{X}) \to \aut(\mathcal{Y})_B 
\end{equation}
is generated by $\sigma$. Therefore, it suffices to show that the map~\eqref{equation: double cover_eq1} is surjective. However, $\mathcal{X}$ is the relative $\Spec$ over $\mathcal{Y}$, i.e.
	$$\mathcal{X} \cong \Spec_{\mathcal{Y}}\left( \Oh_{\mathcal{Y}} \oplus \Oh_{\mathcal{Y}}\left(-\frac{1}{2}B\right)\right), $$
where $\Oh_{\mathcal{Y}}\left(-\frac{1}{2}B\right)$ is the reflexive sheaf corresponding to the Weil divisor class $-\frac{1}{2}B$, and the algebra structure is determined by the composite
$$\Oh_{\mathcal{Y}}\left(-\frac{1}{2}B\right) \otimes \Oh_{\mathcal{Y}}\left(-\frac{1}{2}B\right) \to \Oh_{\mathcal{Y}}\left(-B\right) \to \Oh_{\mathcal{Y}}. $$
Hence, any automorphism of $\mathcal{Y}$ that fixes $B$ lifts to an automorphism of $\mathcal{X}$, i.e. the map~\eqref{equation: double cover_eq1} is surjective.
\end{proof}

\begin{cor}\label{corollary: automorphism of fano, double covers}
Let $\mathcal{X}$ be a smooth Fano threefold of Picard rank 1. 
\begin{enumerate}
\item 
If $\mathcal{X}$ is of index $2$ and degree $2$, then $|\aut(\mathcal{X})|$ divides $2^{10} \cdot 3^6 \cdot 5 \cdot 7$.
\item
If $\mathcal{X}$ is of index $1$ and genus $2$, then $|\aut(\mathcal{X})|$ divides $2^{9} \cdot 3^{4} \cdot 5^6 \cdot 7 \cdot 13$.
\item If $\mathcal{X}$ is of index $1$, genus $3$, and the anticanonical bundle $\omega^{-1}_{\mathcal{X}}$ is not very ample, then $|\aut(\mathcal{X})|$ divides $2^{11} \cdot 3^{4} \cdot 5 \cdot 71$.
\end{enumerate}
\end{cor}

\begin{proof}
Let $L$ be an ample line bundle which generates the Picard group $\Pic(\mathcal{X})$. By e.g.~\cite[\S12.2]{AG5}, the line bundle $L$ is globally generated and defines a finite morphism
$$\phi_L \colon \mathcal{X} \to \PP^n, $$
where $n=3$ in the first two cases and $n=4$ in the last case. Moreover, in the case~(1) (resp.~(2)), $\phi_L$ is a surjective morphism of degree~$2$ such that the branch locus is a smooth hypersurface of degree~$4$ (resp. degree~$6$). So, we obtain the assertion by Corollary~\ref{corollary: projective hypersurface} and Proposition~\ref{proposition: automorphisms of double cover}. 

In the case~(3), $L\cong \omega^{-1}_{\mathcal{X}}$ and the image of $\phi_L$ is a smooth quadratic hypersurface $Q \subset \PP^4$. Moreover, the induced map $\mathcal{X} \to Q $ is a finite surjective morphism of degree $2$ branched in a smooth intersection of $Q$ with a quartic hypersurface in $\PP^4$. We deduce the assertion from Corollary~\ref{corollary: hypersurface in quadric} and Proposition~\ref{proposition: automorphisms of double cover}.
\end{proof}

\begin{rmk}\label{remark: automorphism groups of general Fano}
For the sake of completeness, we also describe the automorphism group $\aut(\mathcal{X})$ of a \emph{general} smooth Fano threefold $\mathcal{X}$ with $\rho(\mathcal{X})=1$. The group $\aut(\mathcal{X})$ for a general $\mathcal{X}$ is \emph{trivial} in the following cases
\begin{itemize}
	\item if $\mathcal{X}$ is a complete intersection of type $(d_1,\ldots, d_c)\neq (2,2)$ in the projective space $\PP^n$ (see~\cite[Theorem~5]{MM63} and~\cite[Theorem~1.3(i)]{CPZ24}), i.e. $\iota(\mathcal{X})=2$ and $d(\mathcal{X})=3$, or $\mathcal{X}$ as in item~(2) of Theorem~\ref{theorem:B}, or $g(\mathcal{X}) = 4,5$;
\item if $g(\mathcal{X})=6$, see~\cite[Proposition~3.21(d)]{DK18};
\item if $7 \leq g(\mathcal{X})\leq 12$, see~\cite[Corollary~2]{DeMa22}.
\end{itemize}
The group $\aut(\mathcal{X})\cong \Z/2$ for a general $\mathcal{X}$ in the following cases
\begin{itemize}
	\item $\mathcal{X}$ is a double cover branched in a complete intersection in $\PP^n$, i.e. $\iota(\mathcal{X})=2$ and $d(\mathcal{X})=2$, or $g(\mathcal{X})=2$, or $\mathcal{X}$ as in item~(3) of Theorem~\ref{theorem:B};
	\item $\mathcal{X}$ is as in item~(7) of Theorem~\ref{theorem:B}, see~\cite[Proposition~3.21(d)]{DK18}.
\end{itemize}
If $\iota(\mathcal{X})=2$, $d(\mathcal{X})=4$, then $\mathcal{X}$ is a complete intersection of two  quadrics in $\PP^5$, and if $\mathcal{X}$ is general, then $\aut(\mathcal{X})\cong (\Z/2)^{\times 5}$ by~\cite[Theorem~1.3(ii)]{CPZ24}. The case $\iota(\mathcal{X})=2$, $d(\mathcal{X})=1$, and $\mathcal{X}$ general seems to be unknown, cf.~\cite[Section~4]{Ess24}.

\end{rmk}

\section{Lagrangian Grassmannian \texorpdfstring{$\LGr(3,6)$}{LGr(3,6)}}\label{section: lagr}
Let $G=Sp_{6}(\Co)$ be the symplectic group. We embed $G$ in $GL_{6}(\Co)$ as the stabiliser of the skew-symmetric bilinear form with matrix 
$$\begin{pmatrix} 0& I_3\\-I_3 & 0\end{pmatrix}.$$ 
We take $X$ to be the Grassmann variety $\LGr(3,6)$ of \emph{isotropic} $3$-planes in $\Co^{6}$ with the natural $G$-action, $\dim(X)=6$. Let $\Oh_X(1)$ be the very ample line bundle over $X$ that corresponds to the Pl\"{u}cker embedding 
$$\LGr(3,6) \hookrightarrow \mathrm{Gr}(3,6)\hookrightarrow \PP^{\binom{6}{3}-1}.$$
Let $E=\Oh_X(1)^{\oplus 3}$. 
The group $\widetilde G$ (see Notation~\ref{tildeg}) in this case is a split extension of $G=Sp_{6}(\Co)$ by $\aut(E/X)\cong GL_{3}(\Co)$ (see Lemma~\ref{tildeGandG}), so it is isomorphic to $Sp_{6}(\Co)\times GL_{3}(\Co)$. This group acts on $\PP(E^*)$, and the line bundle $\Oh_{\PP(E^*)}(1)$ is $\widetilde G$-equivariant. Moreover, by the Cayley trick~\eqref{equation: Cayley trick} we have 
$$\reg{X,E}\cong\reg{\PP(E^*),\Oh_{\PP(E^*)}(1)},$$ so $\widetilde G$ acts on $\reg{X,E}$. Set $L=\Oh_{\PP(E^*)}(1)$ and $E'=J(L)$. In this section we calculate
the classes $S(e_{\widetilde G}(E'),y)$, where $y\in H_*(\PP(E^*),\Z)$.

Let us identify $\PP(E^*)$ with $\LGr(3,6)\times \PP^2$. Then $L$ is identified with $\Oh_X(1)\boxtimes \Oh_{\PP^2}(1)$, and the action of $\widetilde{G}\cong Sp_{6}(\Co)\times GL_{3}(\Co)$ on $L$ with the direct product of the action of~$Sp_{6}(\Co)$ on $\Oh_X(1)$ and the action of $GL_{3}(\Co)$ on $\Oh_{\PP^2}(1)$.

Let $T_1=G\cap D\subset Sp_{6}(\Co)$ and $T_2\subset GL_{3}(\Co)$ be the subgroups of diagonal matrices, and set $T=T_1\times T_2\subset \widetilde{G}$. Then $T$ is a maximal torus of $\widetilde{G}$. Let $P_1\subset Sp_{6}(\Co)$ be the stabiliser of the isotropic $3$-plane spanned by the first~$3$ basis vectors in $\Co^{6}$ and $P_2\subset GL_{3}(\Co)$ be the stabiliser of the point $[1:0:0]\in \PP^2$, and set $P=P_1\times P_2$. Then $P$ is a parabolic subgroup of $\widetilde{G}$, and $\widetilde{G}/P\cong \PP(E^*)$. 

We now identify the rational cohomology of $BP_1, BP_2, BP$ and $B\widetilde{G}$ with subrings of $H^*(BT,\Q)$. Let $\varepsilon_1,\varepsilon_2, \varepsilon_{3}\in \mathfrak{X}(T_1)$ be the same elements as in Notation~\ref{notation: basis for diagonal}; in this subsection we denote the elements $\varepsilon_1, \varepsilon_2, \varepsilon_{3}\in \mathfrak{X}(T_2)$ from Notation~\ref{notation: basis for diagonal} by $\zeta_1,\zeta_2, \zeta_{3}$ respectively to avoid confusion. Recall that the Weyl group $W_G$ of the group $G$ is isomorphic to $\mathfrak{S}_3 \ltimes (\Z/2)^{\times 3}$ and the Weyl group $W_{P_1}$ is the subgroup of $W_G$ isomorphic to the symmetric group $\mathfrak{S}_3$. We have 
$$H^*(BT,\Q)\cong \Sym(\mathfrak{X}(T_1)\oplus \mathfrak{X}(T_2))\otimes \Q\cong \Q[\varepsilon_1,\varepsilon_2,\varepsilon_{3},\zeta_1,\zeta_2, \zeta_{3}].$$ We set 
\begin{align*}
&w_i=\sigma_i(\varepsilon_1,\varepsilon_2,\varepsilon_{3}), i=1,2,3; \;\; 
u=\zeta_1; \;\; u_i=\sigma_{i}(\zeta_2,\zeta_{3}), i=1,2;\\
&s_i=(-1)^i\sigma_i(\varepsilon^2_1,\varepsilon^2_2, \varepsilon^2_{3}), i=1,2, 3; \;\;  t_i=\sigma_i(\zeta_1,\zeta_2,\zeta_{3}), i=1,2,3.
\end{align*}
We have then 
\begin{align*}
&H^*(BP_1,\Q)\cong\Q[w_1,w_2,w_3], H^*(BP_2,\Q)\cong \Q[u,u_1, u_2],\\
&H^*(B\widetilde{G},\Q)\cong \Q[s_1,s_2, s_{3}, t_1,t_2,t_{3}], H^*(BP,\Q)\cong \Q[w_1,w_2,w_3,u,u_1, u_{2}],
\end{align*}
see Examples 4.1.9 and 4.1.11 in~\cite{GK17}. Here we use the fact the classifying space of a parabolic subgroup is homotopy equivalent to the classifying space of its Levi subgroup.

With these identifications the map $\beta^*\colon H^*(B\widetilde{G},\Q)\to H^*(BP,\Q)$ is simply the inclusion, so $\beta^*(t_i)=u_i+uu_{i-1}$ (where we set $u_0=1$, $u_{3}=0$), and $\beta^*(s_i)$ is the degree $2i$ part of
\begin{equation*}
(1+w_1+w_2+w_3)(1-w_1+w_2-w_3).
\end{equation*}

The weight of the $P$-representation which corresponds to the line bundle $L\cong \Oh_X(1)\boxtimes \Oh_{\PP^2}(1)$ is $-w_1-\zeta_1$. The cotangent bundle $\Omega_{\PP(E^*)}$ is isomorphic to the direct sum 
$$\Omega_{\PP(E^*)} \cong \pi_1^*\Omega_{\LGr(3,6)}\oplus \pi_2^*\Omega_{\PP^2},$$ where $\pi_1\colon \PP(E^*)\to \LGr(3,6)$ and $\pi_2\colon \PP(E^*)\to\PP^2$ are the projections. Let $U$ be the tautological rank $3$ vector bundle over $\LGr(3,6)$. We observe that
$$\Omega_{\LGr(3,6)}\cong\Sym^2(U)$$ is obtained from the $P_1$-representation with weights $$\varepsilon_i+\varepsilon_j, 1\leq i\leq j\leq 3.$$ Similarly, the weights of the $P_2$-representation that induces $\Omega_{\PP^2}$ are $\zeta_1-\zeta_2$ and $\zeta_1-\zeta_3$, see~\cite[Example~4.1.1]{GK17}.

So by the exact sequence~\eqref{sesjetbundle} the weights of the $P$-representation such that the associated vector bundle over $\widetilde G/P$ is $J(L)$ are
\begin{equation}\label{equation: weights of jet lgr}
-w_1-\zeta_i,i=1,2,3, \varepsilon_i+\varepsilon_j +(-w_1 -\zeta_1), 1\leq i\leq j\leq 3
\end{equation}
and the product of these is the Euler class $$e_{\widetilde G}(J(L))\in H^*_{\widetilde G}(\widetilde G/P,\Q) \cong H^*(BP,\Q)\subset H^*(BT,\Q),$$
see e.g.~\cite[Lemma~4.1.6]{GK17}.

Let us describe the ring homomorphism~\eqref{equation: restriction map} 
$$\alpha^*\colon H^*(BP,\Q) \cong H^*_{\widetilde{G}}(\widetilde{G}/P,\Q) \to H^*(\widetilde{G}/P,\Q)\cong H^*(X,\Q)\otimes H^*(\PP^2,\Q).$$
Recall that $H^*(\PP^2,\Z)\cong \Z[h]/h^3$, where $h=c_1(\Oh_{\PP^2}(1))$. We have 
\begin{equation}\label{equation: alpha_p2}
\alpha^*(u)=-c_1(\Oh_{\PP^2}(1))=-h, \alpha^*(u_i)=(-\alpha^*(u))^i=h^i, \; i=1,2,
\end{equation}
see e.g.~\cite[Section~4.2]{GK17}. Calculating $\alpha^*(w_i)$ is also straightforward. We note that $\alpha^*(w_i)=c_i(U)$ by e.g.~\cite[Lemma~4.1.2]{GK17}. Set $c_i = c_i(U) \in H^*(\LGr(3,6),\Z)$. 
\begin{prop}\label{proposition: cohomology of lgr}
There is a ring isomorphism
$$H^*(\LGr(3,6),\Z) \cong \Z[c_1,c_2,c_3]/\left(\sum_{i=0}^{2k}(-1)^i c_ic_{2k-i} =0 ; \; k=1,2,3\right), $$
where $c_0=1$ and $c_i=0$ for $i\geq 4$. Moreover, the set $$\{1, c_1, c_2,c_3,c_1c_2, c_1c_3,c_2c_3, c_1c_2c_3\}$$ forms a $\Z$-basis of the cohomology groups $H^*(\LGr(3,6),\Z)$.
\end{prop}
\begin{proof}
See \cite[Theorem~III.6.9(1)]{MT91}.
\end{proof}

\begin{lmm}\label{lemma: decomposition lgr}
There exists a decomposition 
\begin{align*}
e_{\widetilde{G}}(E')&= \prod_{i=1}^3(-w_1-\zeta_i) \times \prod_{1\leq i \leq j \leq 3}(\varepsilon_i+\varepsilon_j+(-w_1-\zeta_1))\\
&=\sum_{i=1}^{3} s_i p_i+\sum_{j=1}^3 t_j q_j \in H^*(BP,\Q),
\end{align*}
where
\begin{align*}
	p_1 &= 12w_2w_3u_2 +\ker(\alpha^*) \in H^{14}(BP,\Q), \\
	p_2 &=-8w_2w_3+20w_1w_3u_1-12w_1w_2u_2-20w_3u_2 + \ker(\alpha^*) \in H^{10}(BP,\Q), \\
	p_3 &=-16w_2u_1+40w_1u_2 + \ker(\alpha^*) \in H^{6}(BP,\Q), \\
	q_1 &=-36w_1w_2w_3u_2 + \ker(\alpha^*) \in H^{16}(BP,\Q), \\
	q_2 &=-24w_1w_2w_3u_1+30w_2w_3u_2 + \ker(\alpha^*) \in H^{14}(BP,\Q), \\ 
	q_3 &=-28w_1w_2w_3+42w_2w_3u_1-21w_1w_3u_2 + \ker(\alpha^*) \in H^{12}(BP,\Q).	
\end{align*}
\end{lmm}
\begin{proof}
By Proposition~\ref{proposition: cohomology of lgr} and formula~\eqref{equation: weights of jet lgr}, the given decomposition can be checked by a straightforward computation in the polynomial ring $H^*(BT,\Q)$.
\end{proof}

\begin{rmk}\label{remark: singular lgr}
We originally found the decomposition in Lemma~\ref{lemma: decomposition lgr} by using Singular~\cite{Singular}.
\end{rmk}

Recall that for every $y\in H_*(\PP(E^*),\Z)$ the image of 
$S(e_{\widetilde{G}}(J(L)),y)$ in $\fr{*}(\widetilde{G},\Z)$ is primitive by Proposition~\ref{proposition: imageofs}.

\begin{lmm}\label{lemma: cokernels lgr}
Let $P^*\subset \fr{*}(\widetilde{G},\Z)$ denote the graded group of primitive elements. 
The map
\begin{equation}\label{yetanothers}
	f_l\colon H_{18-2l}(\PP(E^*),\Z)\rightarrow P^{2l-1}, y\mapsto S(e_{\widetilde{G}}(J(L)),y)
\end{equation}
is
	\begin{enumerate}
		\item the multiplication by $36$ for $l=1$;
		\item is given by the matrix $\begin{pmatrix} 0& 12\\-24 & 30\end{pmatrix}$ for $l=2$;
\item is given by the matrix $\begin{pmatrix} -28 & -42 & 21\end{pmatrix}$ for $l=3$;
	\item is given by the matrix $\begin{pmatrix} -8 & 20 & 12 & 20 \end{pmatrix}$ for $l=4$;
		\item is given by the matrix $\begin{pmatrix} 16 & -40 & 0 & 0 \end{pmatrix}$ for $l=6$
	\end{enumerate}
for some $\Z$-bases of $H_*(\PP(E^*), \Z)$ and $P^*$. Moreover, if $l\neq 1,2,3, 4, 6$, then $P^{2l-1}=0$.
\end{lmm}

\begin{proof}
By Lemma~\ref{lemma: decomposition lgr}, we  have
\begin{equation}\label{lklagrgrassmannian}
	S(e_{\widetilde{G}}(E'),y)=\sum^{3}_{i=1}\langle \alpha^*(p_i),y \rangle \bar\gamma(s_i) +\sum_{j=1}^{3}\langle \alpha^*(q_j),y\rangle \bar\gamma(t_j)
\end{equation}
for every $y\in H_*(\PP(E^*),\Z)$. The cohomology classes $\bar\gamma(s_i)$, and $\bar\gamma(t_j)$ were calculated in \cite[Example~4.1.11]{GK17} and \cite[Example~4.1.9]{GK17}, respectively. In particular, these classes form a $\Z$-basis of $P^*$. We deduce the statement from the description of the cohomology groups $H^*(\LGr(3,6),\Z)$ in Proposition~\ref{proposition: cohomology of lgr}. For example, let $l=2$, then we span  the homology group $H_{14}(\PP(E^*),\Z)\cong \Z^{\oplus 2}$ on the elements dual to $$c_1c_2c_3c = \alpha^*(w_1w_2w_3u_1),\; c_2c_3c^2 = \alpha^*(w_2w_3u_2) \in H^{14}(\PP(E^*),\Z)$$
	and $P^3=H^3(\widetilde{G},\Z)\cong \Z^{\oplus 2}$ on the elements $\bar\gamma(s_1)$ and $\bar\gamma(t_2)$. Therefore, by Lemma~\ref{lemma: decomposition lgr} and the formula~\eqref{lklagrgrassmannian}, the linear map $f_2$ is given by the matrix
$$\begin{pmatrix} 0& 12\\-24 & 30\end{pmatrix}.$$ 
The other cases are done similarly.
\end{proof}

Let $i_l$ be the order of the cokernel of the map $f_l$, $l\geq 1$. We calculate that $\prod i_l = 2^{12} \cdot 3^4 \cdot 7$.

\begin{prop}\label{propositon: bound lgr}
For any regular section $$s \in \reg{X, E} = \reg{\LGr(3,6), \Oh_X(1)^{\oplus 3}}$$ the order of the stabiliser $|\widetilde{G}_{s}|$ divides $2^{12} \cdot 3^4 \cdot 7$ and the order of $|PSp_6(\Co)_{Z(s)}|$ divides $2^{11}\cdot  3^4 \cdot 7$, where $PSp_6(\Co) = Sp_6(\Co)/Z(Sp_6(\Co))= Sp_6(\Co)/\{\pm I\}$ is the projective symplectic group and $PSp_6(\Co)_{Z(s)}$ is the stabiliser of the zero locus $Z(s) \subset \LGr(3,6)$ under the effective $PSp_6(\Co)$-action.
\end{prop}

\begin{proof}
We show first that $\reg{X,E}\cong \reg{\PP(E^*),\Oh_{\PP(E^*)}(1)}$ is an affine variety. Since $\Oh_{\PP(E^*)}(1)$ is a box product of very ample bundles, it suffices to show that $\langle c_8(J(\Oh_{\PP(E^*)}(1)),[\PP(E^*)]\rangle$ is non-zero, see Proposition~\ref{proposition: discriminant is codim 1}. However, one calculates
$$\langle c_8(J(\Oh_{\PP(E^*)}(1)),[\PP(E^*)]\rangle = 108. $$
Therefore, by Corollary~\ref{corollary: main corollary}, Remark~\ref{remark: main corollary cayley trick}, and Lemma~\ref{lemma: cokernels lgr}, we obtain that the stabiliser $\widetilde{G}_s$ are finite and the order $|\widetilde{G}_s|$ divides $\prod i_l = 2^{12} \cdot 3^4 \cdot 7$ for every $s\in \reg{X,E}$. Finally, by Proposition~\ref{Sum of line bundles}, we get $|G_{Z(s)}|=|\widetilde{G}_s|$, which imply the last part.
\end{proof}

\begin{cor}\label{corollary: bound fano lgr}
Let $\mathcal{X}$ be a smooth Fano threefold of Picard rank $1$, index $1$ and genus $9$. Then $|\aut(\mathcal{X})|$ divides $2^{11}\cdot  3^4 \cdot 7$.
\end{cor}

\begin{proof}
We show that $\aut(\mathcal{X}) \cong PSp_6(\Co)_{Z(s)}$ for some regular section $s$ from $\reg{\LGr(3,6), \Oh_X(1)^{\oplus 3}}$. By~\cite[\S12.2]{AG5} or~\cite{Mukai89}, any smooth Fano threefold of genus~$9$ is a linear section of the Grassmann variety $\LGr(3,6)$ of isotropic $3$-planes. So, it suffices to show that every automorphism of $\mathcal{X}$ is induced by an element of $\aut(\LGr(3,6))=PSp_6(\Co)$.

By~\cite{Mukai89} and~\cite{Mukai92} (see also~\cite[Theorem~1.1]{BKM24} for a better treatment), there exists a unique (up to isomorphism) stable vector bundle $\mathcal{E}_3$ over $\mathcal{X}$ such that $\rk(\mathcal{E}_3)=3$, $\Lambda^3\mathcal{E}_3=\omega_{\mathcal{X}}$ is the canonical line bundle, $H^*(\mathcal{X},\mathcal{E}_3)=0$, and $\Ext^*(\mathcal{E}_3,\mathcal{E}_3)=0$. Moreover, the dual bundle $\mathcal{E}_3^*$ is globally generated, $\dim \Gamma(\mathcal{X}, \mathcal{E}^*_3)=6$, and the kernel of the natural map
$$\Lambda^2\Gamma(\mathcal{X}, \mathcal{E}_3^*) \to  \Gamma(\mathcal{X}, \Lambda^2 \mathcal{E}_3^*)$$
is one-dimensional and spanned by a non-degenerate skew-symmetric form $\sigma_{\mathcal{X}} \in \Lambda^2\Gamma(\mathcal{X}, \mathcal{E}_3^*)$. By~\cite[Proposition~5.7]{Kuz22} for $S=B\aut(\mathcal{X})$ to be the classifying stack of the algebraic group $\aut(\mathcal{X})$, $\mathcal{E}^*_3$ is an $\aut(\mathcal{X})$-equivariant vector bundle. Since $\mathcal{E}^*_3$ is globally generated, it defines an $\aut(\mathcal{X})$-equivariant closed embedding
$$\mathcal{X} \hookrightarrow \Gr(3, \Gamma(\mathcal{X}, \mathcal{E}^*_3))$$
such that the image of $\mathcal{X}$ is a transversal section of $\LGr(3,\Gamma(\mathcal{X}, \mathcal{E}^*_3))$ and a projective subspace of codimension $3$. Therefore, any automorphism of $\mathcal{X}$ can be extended to an automorphism of $\LGr(3,6)$.
\end{proof}

\begin{rmk}\label{remark: previous estimates lgr}
We point out that Corollary~\ref{corollary: bound fano lgr} can be obtained more geometrically. By~\cite[Corollary~4.3.5]{KPS18}, one has $\aut(\mathcal{X}) \subset \aut(S)$, where $S=\PP(\mathcal{E}_2)$ is the projectivisation of a simple rank $2$ vector bundle on a smooth irreducible curve $C$ of genus $3$. By Corollary to Proposition~2 in \cite[p.~202]{Gro59}, there exists a short exact sequence
$$ 1 \to \Gamma \to \aut(S) \to \aut(C), $$
where $\Gamma$ is a $2$-torsion subgroup of the Picard group $\Pic(C)$, cf.~\cite[Lemma~3.2.7]{KPS18}. In particular, $|\Gamma|$ divides $2^6$. By the main result of~\cite{KK79}, we find that the least common multiple of the orders $|\aut(C)|$ over all smooth irreducible curves $C$ of genus $3$ is $2^5\cdot 3^2\cdot 7$. Therefore, $|\aut(\mathcal{X})|$ divides $2^{11}\cdot 3^2 \cdot 7$.
\end{rmk}

\section{Orthogonal Grassmannian \texorpdfstring{$\OGr_+(5,10)$}{OGr(5,10)}}\label{section: ogr}
Let $G=Spin_{10}(\Co)$ be the Spin group. We consider $G$ as the universal cover of the special orthogonal group $SO_{10}(\Co)$ which is embedded in $GL_{10}(\Co)$ as the stabiliser of the symmetric bilinear form with matrix 
$$\begin{pmatrix} 0& I_5\\I_5 & 0\end{pmatrix}.$$ 
We take $X$ to be the connected component of the Grassmann variety $\OGr_+(5,10)$ of \emph{isotropic} $5$-planes in $\Co^{10}$ with the natural $G$-action such that $X$ contains the isotropic $5$-plane spanned by the first $5$ basis vector in $\Co^{10}$, $\dim(X)=10$. Let $\Oh_X(1)$ be the very ample line bundle over $X$ that corresponds to the $G$-equivariant embedding 
$$\OGr_+(5,10) \hookrightarrow \PP^{16-1}$$
of $X$ into the projectivisation of the half-spinor $G$-representation, see e.g.~\cite[Section~2.2]{Kuznetsov_tenfold}. We note that the very ample line bundle $\Oh_X(2)=\Oh_X(1)^{\otimes 2}$ corresponds to the Pl\"{u}cker embedding 
$$\OGr_+(5,10) \hookrightarrow \mathrm{Gr}(5,10)\hookrightarrow \PP^{\binom{10}{5}-1}.$$
Let $E=\Oh_X(1)^{\oplus 7}$. 
The group $\widetilde G$ (see Notation~\ref{tildeg}) in this case is a split extension of $G=Spin_{10}(\Co)$ by $\aut(E/X)\cong GL_{7}(\Co)$ (see Lemma~\ref{tildeGandG}), so it is isomorphic to $Spin_{10}(\Co)\times GL_{7}(\Co)$. This group acts on $\PP(E^*)$, and the line bundle $\Oh_{\PP(E^*)}(1)$ is $\widetilde G$-equivariant. Moreover, by the Cayley trick~\eqref{equation: Cayley trick} we have 
$$\reg{X,E}\cong\reg{\PP(E^*),\Oh_{\PP(E^*)}(1)},$$ so $\widetilde G$ acts on $\reg{X,E}$. Set $L=\Oh_{\PP(E^*)}(1)$ and $E'=J(L)$. In this section we calculate
the classes $S(e_{\tilde G}(E'),y)$ where $y\in H_*(\PP(E^*),\Z)$.

Let us identify $\PP(E^*)$ with $\OGr_+(5,10)\times \PP^6$. Then $L$ is identified with $\Oh_X(1)\boxtimes \Oh_{\PP^6}(1)$, and the action of $\widetilde{G}\cong Spin_{10}(\Co)\times GL_{7}(\Co)$ on $L$ with the direct product of the action of~$Spin_{10}(\Co)$ on $\Oh_X(1)$ and the action of $GL_{7}(\Co)$ on $\Oh_{\PP^6}(1)$.

Let $T_1=SO_{10}(\Co)\cap D\subset SO_{10}(\Co)$ and $T_2\subset GL_{7}(\Co)$ be the subgroups of diagonal matrices. We take $T'_1=\pi^{-1}(T_1)\subset Spin_{10}(\Co)$ to be the preimage of $T_1$ under the covering map
$$\pi\colon Spin_{10}(\Co) \to SO_{10}(\Co). $$
Set $T=T'_1\times T_2\subset \widetilde{G}$. Then $T$ is a maximal torus of $\widetilde{G}$. Let $P_1\subset Spin_{10}(\Co)$ be the stabiliser of the isotropic $5$-plane spanned by the first~$5$ basis vectors in $\Co^{10}$ and $P_2\subset GL_{7}(\Co)$ be the stabiliser of the point $[1:0:\cdots:0]\in \PP^6$, and set $P=P_1\times P_2$. Then $P$ is a parabolic subgroup of $\widetilde{G}$, and $\widetilde{G}/P\cong \PP(E^*)$. 

We now identify the rational cohomology of $BP_1, BP_2, BP$ and $B\widetilde{G}$ with subrings of $H^*(BT,\Q)$. Note that $\mathfrak{X}(T_1) \hookrightarrow \mathfrak{X}(T'_1) \subset \mathfrak{t}_1^*$ is a subgroup of index $2$, where $\mathfrak{t}_1$ is the Lie algebra of both $T_1$ and $T'_1$; so, 
$$\mathfrak{X}(T_1)\otimes \Q = \mathfrak{X}(T'_1)\otimes \Q.$$
Let $\varepsilon_1,\ldots, \varepsilon_{5}\in \mathfrak{X}(T_1)$ be the same elements as in Notation~\ref{notation: basis for diagonal}; in this subsection we denote the elements $\varepsilon_1, \ldots, \varepsilon_{7}\in \mathfrak{X}(T_2)$ from Notation~\ref{notation: basis for diagonal} by $\zeta_1,\ldots, \zeta_{7}$ respectively to avoid confusion. Recall that the Weyl group $W_G$ of the group $G$ is isomorphic to $\mathfrak{S}_5 \ltimes (\Z/2)^{\times 4}$ and the Weyl group $W_{P_1}$ is the subgroup of $W_G$ isomorphic to the symmetric group $\mathfrak{S}_5$. We have 
$$H^*(BT,\Q)\cong \Sym(\mathfrak{X}(T'_1)\oplus \mathfrak{X}(T_2))\otimes \Q \cong \Q[\varepsilon_1,\ldots,\varepsilon_{5},\zeta_1,\ldots, \zeta_{7}].$$ We set 
\begin{align*}
&w_i=\sigma_i(\varepsilon_1,\ldots,\varepsilon_{5}), i=1,\ldots,5; \;\;
u=\zeta_1; u_i=\sigma_{i}(\zeta_2,\ldots, \zeta_7), i=1,\ldots, 6;\\
&s_i=(-1)^i\sigma_i(\varepsilon^2_1,\ldots, \varepsilon^2_{5}), i=1,\ldots, 4; \;\; s=w_5=\varepsilon_1 \cdot \ldots \cdot \varepsilon_5, \\
&t_i=\sigma_i(\zeta_1,\ldots,\zeta_{7}), i=1,\ldots,7.
\end{align*}
We have then 
\begin{align*}
&H^*(BP_1,\Q)\cong\Q[w_1,\ldots,w_4, s], H^*(BP_2,\Q)\cong \Q[u,u_1, \ldots,u_6],\\
&H^*(B\widetilde{G},\Q)\cong \Q[s_1,\ldots, s_{4}, s, t_1,\ldots,t_{7}], H^*(BP,\Q)\cong \Q[w_1,\ldots,w_4,s,u,\ldots, u_{6}],
\end{align*}
see Examples~4.1.9 and 4.1.16 in~\cite{GK17}. Here we use the fact the classifying space of a parabolic subgroup is homotopy equivalent to the classifying space of its Levi subgroup.

With these identifications the map $\beta^*\colon H^*(B\widetilde{G},\Q)\to H^*(BP,\Q)$ is simply the inclusion, so $\beta^*(t_i)=u_i+uu_{i-1}$ (where we set $u_0=1$, $u_{7}=0$), and $\beta^*(s_i)$ is the degree $2i$ part of
\begin{equation*}
(1+w_1+w_2+w_3+w_4+s)(1-w_1+w_2-w_3+w_4-s).
\end{equation*}

We recall that the character group $\mathfrak{X}(T'_1)\subset \mathfrak{t}_1^*$ is the set of all $x\in\mathfrak{t}_1^*$ such that the coordinates of $x$ in the basis $\{\varepsilon_1,\ldots,\varepsilon_5\}$ are either all integer or all half-integer. The weight of the $P$-representation which corresponds to the line bundle $L\cong \Oh_X(1)\boxtimes \Oh_{\PP^6}(1)$ is $$-\frac{1}{2}w_1-\zeta_1 \in \mathfrak{X}(T) = \mathfrak{X}(T'_1)\oplus \mathfrak{X}(T_2).$$ The cotangent bundle $\Omega_{\PP(E^*)}$ is isomorphic to the direct sum 
$$\Omega_{\PP(E^*)} \cong \pi_1^*\Omega_{\OGr_+(5,10)}\oplus \pi_2^*\Omega_{\PP^6},$$ where $\pi_1\colon \PP(E^*)\to \OGr_+(5,10)$ and $\pi_2\colon \PP(E^*)\to\PP^6$ are the projections. Let $U$ be the tautological rank $5$ vector bundle over $\OGr_+(5,10)$. We have
$$\Omega_{\OGr_+(5,10)}\cong\Lambda^2(U)$$ is obtained from the $P_1$-representation with weights $$\varepsilon_i+\varepsilon_j \in \mathfrak{X}(T'_1), 1\leq i < j\leq 5.$$ Similarly, the weights of the $P_2$-representation that induces $\Omega_{\PP^6}$ are $$\zeta_1-\zeta_2, \ldots, \zeta_1-\zeta_7 \in \mathfrak{X}(T_2),$$ see~\cite[Section~4.2]{GK17}.

So by the exact sequence~\eqref{sesjetbundle} the weights of the $P$-representation such that the associated vector bundle over $\widetilde G/P$ is $J(L)$ are
\begin{equation}\label{equation: weights of jet ogr}
-\frac{1}{2}w_1-\zeta_i,i=1,\ldots,7, \varepsilon_i+\varepsilon_j +\left(-\frac{1}{2}w_1 -\zeta_1\right), 1\leq i< j\leq 5
\end{equation}
and the product of these is the Euler class $$e_{\widetilde G}(J(L))\in H^*_{\widetilde G}(\widetilde G/P,\Q) \cong H^*(BP,\Q)\subset H^*(BT,\Q),$$
see e.g.~\cite[Lemma~4.1.6]{GK17}.

Let us describe the ring homomorphism~\eqref{equation: restriction map} 
$$\alpha^*\colon H^*(BP,\Q) \cong H^*_{\widetilde{G}}(\widetilde{G}/P,\Q) \to H^*(\widetilde{G}/P,\Q)\cong H^*(X,\Q)\otimes H^*(\PP^6,\Q).$$
Recall that $H^*(\PP^6,\Z)\cong \Z[h]/h^7$, $h=c_1(\Oh_{\PP^6}(1))$. We have 
\begin{equation}\label{equation: alpha_pp6}
\alpha^*(u)=-c_1(\Oh_{\PP^6}(1))=-h, \alpha^*(u_i)=(-\alpha^*(u))^i=h^i, \; 1\leq i\leq 6.
\end{equation}
Calculating $\alpha^*(w_i)$ is also straightforward. 
We note that $$\alpha^*(w_i)=c_i(U) \in H^*(\OGr_+(5,10),\Z)$$ and $\alpha^*(s) = 0$ as $s\in H^*(B\widetilde{G},\Q)$. In contrast to Proposition~\ref{proposition: cohomology of lgr}, the cohomology classes $c_i(U)$ are divisible by $2$ and so, they do not generate the cohomology ring $H^*(\OGr_+(5,10),\Z)$. 
\begin{prop}\label{proposition: cohomology of ogr}
There are cohomology classes $e_i \in H^{2i}(\OGr_+(5,10),\Z)$, $1\leq i \leq 4$ such that $c_i(U) = 2e_i$. Moreover, there is a ring isomorphism
$$H^*(\OGr_+(5,10),\Z) \cong \Z[e_1,\ldots,e_4]/\left(e_{2k}+\sum_{i=1}^{2k-1}(-1)^i e_ie_{2k-i} =0 ; \; 1\leq k \leq 4\right), $$
	where $e_i=0$ for $i\geq 5$. In particular, the set $$\{1, e_{i_1}\cdots e_{i_r} \; : \; 1\leq i_1< \ldots < i_r \leq 4\}$$ is a $\Z$-basis of the cohomology groups $H^*(\OGr_+(5,10),\Z)$.
\end{prop}
\begin{proof}
See \cite[Theorem III.6.11]{MT91}.
\end{proof}

\begin{lmm}\label{lemma: decomposition_ogr}
There exists a decomposition
	\begin{align*}
		e_{\widetilde{G}}(E')&= \prod_{i=1}^7\left(-\frac{1}{2}w_1-\zeta_i\right) \times \prod_{1\leq i < j \leq 5}\left(\varepsilon_i+\varepsilon_j+(-\frac{1}{2}w_1-\zeta_1)\right)\\
		&=sr +\sum_{i=1}^{4} s_i p_i+\sum_{j=1}^7 t_j q_j \in H^{34}(BP,\Q), 
	\end{align*}
where
\begingroup
\allowdisplaybreaks
\begin{align*}
r =& - \frac{3}{16}w_1w_2w_3w_4u_2 + \frac{3}{2}w_2w_3w_4u_3 - 6w_1w_3w_4u_4 + \frac{81}{8}w_1w_2w_4u_5 \\
&- \frac{15}{2}w_3w_4u_5 - 8w_1w_2w_3u_6 + 23w_2w_4u_6 + \ker(\alpha^*) \in H^{24}(BP,\Q),
\end{align*}
\begin{align*}
p_1 =&\frac{7}{4}w_2w_3w_4u_6  + \ker(\alpha^*) \in H^{30}(BP,\Q), \\
	p_2 =&\frac{3}{16}w_1w_2w_3w_4u_3 - \frac{7}{8}w_2w_3w_4u_4 + \frac{17}{8}w_1w_3w_4u_5 - \frac{27}{16}w_1w_2w_4u_6 \\
	&- w_3w_4u_6 + \ker(\alpha^*) \in H^{26}(BP,\Q), \\
	p_3 =&\frac{1}{4}w_2w_3w_4u_2 - \frac{5}{4}w_1w_3w_4u_3 + \frac{15}{8}w_1w_2w_4u_4 - w_3w_4u_4 - \frac{3}{4}w_1w_2w_3u_5 \\
	&- \frac{1}{2}w_2w_4u_5 + \frac{3}{2}w_2w_3u_6 + \frac{5}{2}w_1w_4u_6 + \ker(\alpha^*) \in H^{22}(BP,\Q), \\
	p_4 =&\frac{1}{16}w_2w_3w_4 - \frac{5}{16}w_1w_3w_4u_1 + \frac{15}{32}w_1w_2w_4u_2 - \frac{1}{4}w_3w_4u_2 - \frac{3}{16}w_1w_2w_3u_3 \\&+ \frac{1}{4}w_2w_4u_3 - \frac{3}{8}w_2w_3u_4 + \frac{1}{4}w_1w_4u_4 + \frac{27}{8}w_1w_3u_5 - 4w_4u_5 - \frac{11}{2}w_1w_2u_6 \\&+ \frac{13}{4}w_3u_6 + \ker(\alpha^*) \in H^{18}(BP,\Q),
\end{align*}
\begin{align*}
	q_1 =& - \frac{15}{4}w_1w_2w_3w_4u_6 + \ker(\alpha^*) \in H^{32}(BP,\Q), \\ 
	q_2 =& - \frac{3}{2}w_1w_2w_3w_4u_5 + \frac{5}{4}w_2w_3w_4u_6 + \ker(\alpha*) \in H^{30}(BP,\Q), \\
	q_3 =& - \frac{9}{4}w_1w_2w_3w_4u_4 + \frac{23}{4}w_2w_3w_4u_5 - \frac{25}{4}w_1w_3w_4u_6 + \ker(\alpha^*) \in H^{28}(BP,\Q), \\
	q_4 =& - \frac{9}{4}w_1w_2w_3w_4u_3 + \frac{17}{4}w_2w_3w_4u_4 - \frac{7}{4}w_1w_3w_4u_5 - \frac{9}{8}w_1w_2w_4u_6 \\&+ w_3w_4u_6 + \ker(\alpha^*) \in H^{26}(BP,\Q), \\
	q_5 =& - \frac{9}{4}w_1w_2w_3w_4u_2 + \frac{17}{4}w_2w_3w_4u_3 - \frac{13}{4}w_1w_3w_4u_4 + \frac{15}{8}w_1w_2w_4u_5 \\&- 2w_3w_4u_5 - \frac{3}{4}w_1w_2w_3u_6 + 2w_2w_4u_6 + \ker(\alpha^*) \in H^{24}(BP,\Q), \\
	q_6 =& - \frac{9}{4}w_1w_2w_3w_4u_1 + \frac{17}{4}w_2w_3w_4u_2 - \frac{13}{4}w_1w_3w_4u_3 + \frac{3}{8}w_1w_2w_4u_4 \\&+ w_3w_4u_4 + \frac{3}{4}w_1w_2w_3u_5 - 4w_2w_4u_5 - \frac{3}{2}w_2w_3u_6 \\&+ 2w_1w_4u_6 + \ker(\alpha^*) \in H^{22}(BP,\Q), \\
	q_7 =& - \frac{9}{4}w_1w_2w_3w_4 + \frac{17}{4}w_2w_3w_4u_1 - \frac{13}{4}w_1w_3w_4u_2 + \frac{3}{8}w_1w_2w_4u_3 \\&+ w_3w_4u_3 - \frac{3}{4}w_1w_2w_3u_4 + 5w_2w_4u_4 + \frac{3}{2}w_2w_3u_5 - 4w_1w_4u_5 \\&- \frac{3}{2}w_1w_3u_6 + 4w_4u_6 + \ker(\alpha^*) \in H^{20}(BP,\Q).
\end{align*}
\endgroup
\end{lmm}
\begin{proof}
	The proof is a direct computation in the polynomial ring $H^*(BT,\Q)$, cf. Lemma~\ref{lemma: decomposition lgr}. Again, we originally found this decomposition by using Singular~\cite{Singular}, cf. Remark~\ref{remark: singular lgr}.
\end{proof}

\begin{lmm}\label{lemma: cokernels ogr}
Let $P^*\subset \fr{*}(\widetilde{G},\Z)$ denote the graded group of primitive elements. Then 
the map
\begin{equation}\label{yetanothers ogr}
	f_l\colon H_{34-2l}(\PP(E^*),\Z)\rightarrow P^{2l-1}, y\mapsto S(e_{\widetilde{G}}(J(L)),y)
\end{equation}
is
	\begin{enumerate}
		\item the multiplication by $60$ for $l=1$;
		\item is given by the matrix $\begin{pmatrix} 28& 0\\10 & -24\end{pmatrix}$ for $l=2$;
\item is given by the matrix $\begin{pmatrix} 36 & -46 & 50 \\ \end{pmatrix}$ for $l=3$;
	\item is given by the matrix $\begin{pmatrix} 6 & -14 & 34 & -27 & -8  \\ -36 & 34 & -14 & -9 & 4\end{pmatrix}$ 
 for $l=4$;
		\item is given by the matrix $\begin{pmatrix} 3 & -12 & 48 & -81 & 30 & 64 & -92 \\ 36 & -34 & 26 & -15 & 8 & 6 & -8\end{pmatrix}$ for $l=5$;
		\item is given by the matrix $$\begin{pmatrix} 0 & -4 & 20 & -30 & 8 & 12 & 4 & -12 & -20 \\ 36 & -34 & 26 & -3 & -4 & -6 & 16 & 6 & -8\end{pmatrix}$$ for $l=6$;
			\item is given by the matrix $$\begin{pmatrix} 8 & -6 & -16 & 6 & 20 &-6 & 4 & 3 & -26 & 34 & -36\end{pmatrix}$$ for $l=7$;
				\item is given by the matrix $$\begin{pmatrix}26 & -88 & -32 & 54 & 4& -6 & 4 & -6 & -4 & 15 &-10 & 2 \end{pmatrix}$$ for $l=8$
	\end{enumerate}
for some $\Z$-bases of $H_*(\PP(E^*), \Z)$ and $P^*$. Moreover, if $l\geq 9$, then $P^{2l-1}=0$.
\end{lmm}

\begin{proof}
The proof is similar to the proof of Lemma~\ref{lemma: cokernels lgr}. By Lemma~\ref{lemma: decomposition_ogr}, we have
\begin{equation}\label{lkogrgrassmannian}
	S(e_{\widetilde{G}}(E'),y) =\langle \alpha^*(r),y\rangle \bar\gamma(s)+\sum^{4}_{i=1}\langle \alpha^*(p_i),y\rangle \bar\gamma(s_i) +\sum_{j=1}^{7}\langle \alpha^*(q_j),y\rangle \bar\gamma(t_j)
\end{equation}
for every $y \in H_*(\PP(E^*),\Z)$. The cohomology classes $\bar\gamma(s)$, $\bar\gamma(s_i)$, and $\bar\gamma(t_j)$ were calculated in \cite[Example~4.1.16]{GK17} and \cite[Example~4.1.9]{GK17}, respectively. In particular, $\frac{1}{2}\bar\gamma(s_1), \frac{1}{2}\bar\gamma(s_2), \frac{1}{2}\bar\gamma(s_3), \frac{1}{4}\bar\gamma(s_4)$, $\bar{\gamma}(s)$, and $\bar{\gamma}(t_1), \ldots, \bar{\gamma}(t_7)$ is a $\Z$-basis for $P^*$. We deduce the statement from the description of the cohomology groups $H^*(\OGr_+(5,10),\Z)$ in Proposition~\ref{proposition: cohomology of ogr}. For example, let $l=4$, then we span the homology group $H_{26}(\PP(E^*),\Z)\cong \Z^{\oplus 5}$ on the elements dual to the cohomology classes 
\begin{align*}
	e_1e_2e_3e_4c^3 &= \alpha^*\left(\frac{1}{16}w_1w_2w_3w_4u_3\right),\; e_2e_3e_4c^4 = \alpha^*\left(\frac{1}{8}w_2w_3w_4u_4\right), \\
	e_1e_3e_4c^5 &= \alpha^*\left(\frac{1}{8}w_1w_3w_4u_5\right), \; e_1e_2e_4c^6 = \alpha^*\left(\frac{1}{8}w_1w_2w_4u_6\right), \\ e_3e_4c^6 &= \alpha^*\left(\frac{1}{4}w_3w_4u_6\right)
\end{align*}
and we span $P^7\cong \Z^{\oplus 2}$ on $\frac{1}{2}\bar\gamma(s_2)$ and $\bar{\gamma}(t_4)$. Therefore, by Lemma~\ref{lemma: decomposition_ogr} and the formula~\eqref{lkogrgrassmannian}, the linear map $f_4$ is given by the matrix
$$\begin{pmatrix} 6 & -14 & 34 & -27 & -8  \\ -36 & 34 & -14 & -9 & 4\end{pmatrix}.$$ 
The other cases are done similarly.
\end{proof}

Let $i_l$ be the order of the cokernel of the map $f_l$, $l\geq 1$. We calculate that $\prod i_l = 2^{11} \cdot 3^5 \cdot 5 \cdot 7$, so by Corollary~\ref{corollary: main corollary} and Remark~\ref{remark: main corollary cayley trick}, we conclude the next proposition.

\begin{prop}\label{propositon: bound ogr}
For any regular section $$s \in \reg{X, E} = \reg{\OGr_+(5,10), \Oh_X(1)^{\oplus 7}}$$ the order of the stabiliser $|\widetilde{G}_{s}|$ divides $2^{11} \cdot 3^5 \cdot 5 \cdot 7$ and the order of $|PSO_{10}(\Co)_{Z(s)}|$ divides $2^{9}\cdot  3^5 \cdot 5\cdot 7$, where $PSO_{10}(\Co) = Spin_{10}(\Co)/Z(Spin_{10}(\Co))= SO_{10}(\Co)/\{\pm I\}$ is the projective orthogonal group and $PSO_{10}(\Co)_{Z(s)}$ is the stabiliser of the zero locus $Z(s) \subset \OGr_+(5,10)$ under the effective $PSO_{10}(\Co)$-action.
\end{prop}

\begin{proof}
As in Proposition~\ref{propositon: bound lgr}, it suffices to show that $\reg{X,E}$ is an affine variety. Since $\Oh_{\PP(E^*)}(1)$ is a box product of very ample line bundles, it is enough to check that $\langle c_{16}(e(J(\Oh_{\PP(E^*)(1)})), [\PP(E^*)]\rangle \neq 0$, see Proposition~\ref{proposition: discriminant is codim 1}. One calculates
$$\langle c_{16}(J(\Oh_{\PP(E^*)}(1))), [\PP(E^*)]\rangle = 420. $$
The assertion follows now by Corollary~\ref{corollary: main corollary} and Proposition~\ref{Sum of line bundles}.
\end{proof}

\begin{cor}\label{corollary: bound fano ogr}
Let $\mathcal{X}$ be a smooth Fano threefold of Picard rank $1$, index $1$ and genus $7$. Then $|\aut(\mathcal{X})|$ divides $2^{9}\cdot  3^5 \cdot 5 \cdot 7$.
\end{cor}

\begin{proof}
The proof is similar to the proof of Corollary~\ref{corollary: bound fano lgr}. We show that $\aut(\mathcal{X}) = PSO_{10}(\Co)_{Z(s)}$ for some regular section $s$ from $\reg{\OGr_+(5,10), \Oh_X(1)^{\oplus 7}}$. By~\cite[\S12.2]{AG5} or~\cite{Mukai89}, any smooth Fano threefold of genus~$7$ is a linear section of the Grassmann variety $\OGr_+(5,10)$ of isotropic $5$-planes. So, it suffices to show that every automorphism of $\mathcal{X}$ is induced by an element of $\aut(\OGr_+(5,10))=PSO_{10}(\Co)$.

By~\cite{Mukai89} and~\cite{Mukai92}, $\mathcal{X}$ can be equipped with a unique (up to isomorphism) stable vector bundle $\mathcal{E}_5$ of rank $5$ such that $\mathcal{E}_5$ is globally generated, $\dim \Gamma(\mathcal{X}, \mathcal{E}_5)=10$, and $\Lambda^5\mathcal{E}_5=(\omega_{\mathcal{X}})^{-2}$ is the second tensor power of the anticanonical line bundle. Moreover, the kernel of the natural map
$$\Sym^2\Gamma(\mathcal{X}, \mathcal{E}_5) \to  \Gamma(\mathcal{X}, \Sym^2 \mathcal{E}_5)$$
is one-dimensional and it is spanned by a non-degenerate symmetric form $\sigma_{\mathcal{X}} \in \Sym^2\Gamma(\mathcal{X}, \mathcal{E}_5)$. By~\cite[Proposition~5.11]{Kuz22} for $S=B\aut(\mathcal{X})$ to be the classifying stack of the algebraic group $\aut(\mathcal{X})$, $\mathcal{E}_5$ is an $\aut(\mathcal{X})$-equivariant vector bundle. Since $\mathcal{E}_5$ is globally generated, it defines an $\aut(\mathcal{X})$-equivariant closed embedding
$$\mathcal{X} \hookrightarrow \Gr(5, \Gamma(\mathcal{X}, \mathcal{E}_5))$$
such that the image of $\mathcal{X}$ is a transversal section of $\OGr_+(5,\Gamma(\mathcal{X}, \mathcal{E}_5))$ and a projective subspace of codimension $7$. Therefore, any automorphism of $\mathcal{X}$ can be extended to an automorphism of $\OGr_+(5,10)$.
\end{proof}

\begin{rmk}\label{remark: previous estimates ogr}
Similarly to Corollary~\ref{corollary: bound fano lgr}, Corollary~\ref{corollary: bound fano ogr} also can be obtained more geometrically, cf. Remark~\ref{remark: previous estimates lgr}. By~\cite[Corollary~4.3.5]{KPS18}, one has $\aut(\mathcal{X}) \subset \aut(C)$, where $C$ is a smooth irreducible curve of genus $7$. According to Tables~2, 5, 6, and Section~6.1 in~\cite{Zom10}, the least common multiple of the orders $|\aut(C)|$ over all smooth irreducible curves $C$ of genus $7$ is $2^6\cdot 3^3\cdot 5 \cdot  7$. Therefore, $|\aut(\mathcal{X})|$ divides $2^{6}\cdot 3^3 \cdot 5 \cdot 7$.
\end{rmk}

\section{\texorpdfstring{$G_2$}{G2}-Grassmannian}\label{section: g2gr}
We recall first some basic properties of the complex simple group $G_2$ and the constructions of $G_2$-homogeneous varieties. We will be mostly interested in the $G_2$-Grassmannian $X=G_2/P_1$ and we will follow in our calculations the general approach of~\cite[Section~4.1]{GK17}. However, since $G_2/P_1$ is a zero locus of a section of a certain vector bundle over the Grassmann variety $\Gr(2,7)$, many computations can be done directly. 

Let $SO_{7}(\Co)$ be the special orthogonal group embedded in $GL_{7}(\Co)$ as the stabiliser of the symmetric bilinear form with matrix 
$$\begin{pmatrix} 0& I_3 &0\\I_3 & 0& 0 \\ 0 & 0 & 1\end{pmatrix}$$ 
and let $T_1 = SO_7(\Co)\cap D\subset SO_7(\Co)$ be the subgroup of diagonal matrices in $SO_7(\Co)$. We take $T'_1=\pi^{-1}(T_1)\subset Spin_{7}(\Co)$ to be the preimage of $T_1$ under the covering map
$$\pi\colon Spin_{7}(\Co) \to SO_{7}(\Co). $$
Let $\varepsilon_1,\varepsilon_2, \varepsilon_{3}\in \mathfrak{X}(T_1)$ be the same elements as in Notation~\ref{notation: basis for diagonal}. Note that $\mathfrak{X}(T_1) \hookrightarrow \mathfrak{X}(T'_1) \subset \mathfrak{t}_1^*$ is a subgroup of index $2$, where $\mathfrak{t}_1$ is the Lie algebra of both $T_1$ and $T'_1$. Moreover, the character group $\mathfrak{X}(T'_1)\subset \mathfrak{t}_1^*$ is the set of all $x\in\mathfrak{t}_1^*$ such that the coordinates of $x$ in the basis $\{\varepsilon_1,\varepsilon_2,\varepsilon_3\}$ are either all integer or all half-integer.

Let $G$ be a (unique) complex simple group of type $G_2$. We embed $G$ into the Spin group $Spin_7(\Co)$ such that $T_G=G\cap T'_1\subset Spin_7(\Co)$ is a maximal torus in $G$ and the kernel of the surjective map
\begin{equation}\label{equation: G2-grass, weight lattice}
\mathfrak{X}(T_1') \twoheadrightarrow \mathfrak{X}(T_G) 
\end{equation}
is spanned by the weight $\frac{1}{2}\left(\varepsilon_1+\varepsilon_2+\varepsilon_3\right)$, see \cite[Example~4.1.17]{GK17}. 
Choose the simple roots in $\mathfrak{X}(T_G)$ as follows 
$$\alpha_1 = \varepsilon_2 \in \mathfrak{X}(T_G)\;\; \text{and} \;\; \alpha_2 = \varepsilon_1-\varepsilon_2 \in \mathfrak{X}(T_G)$$
and let $B_1\subset G$ be the corresponding Borel subgroup. We note that $\alpha_1$ is the shortest root and $\alpha_2$ is the longest root. 

For a character $\chi \in \mathfrak{X}(B_1) =\mathfrak{X}(T)$ we denote the line bundle $G\times^{B_1}\chi$ over $G/B_1$ by $\Oh(\chi)$. Set $\omega_1=\varepsilon_1 + \varepsilon_2$ (resp. $\omega_2= \varepsilon_1-\varepsilon_3=2\varepsilon_1+\varepsilon_2$) to be the fundamental weight which correspond to $\alpha_1$ (resp. $\alpha_2$). By the Borel-Weil-Bott theorem (see e.g.\ \cite{S95}, and also \cite[\S7.4]{S07} for a detailed account), $\Oh(\omega_1)$ and $\Oh(\omega_2)$ are globally generated (but not ample) line bundles over $G/B_1$, and moreover, $\Gamma(G/B_1,\Oh(\omega_1))$ (resp. $\Gamma(G/B_1,\Oh(\omega_2))$) is the $7$-dimensional (resp. $14$-dimensional) fundamental $G$-representation,~\cite[\S22.3]{FH91}.

Consider the regular closed $G$-equivariant morphism
$$\phi\colon G/B_1 \to \PP(\Gamma(G/B_1,\Oh(\omega_2))^*) \cong \PP^{13}$$
defined by the line bundle $\Oh(\omega_2)$. Let $X$ be the image of $\phi$ equipped with the natural $G$-action; $X$ is the closed $G$-orbit in the projective space $\PP(\Gamma(G/B_1,\Oh(\omega_2))^*)$, $\dim(X)=5$. In particular, $X\cong G/P_1$ for the parabolic subgroup $P_1\supset B_1$ defined by the simple roots $\alpha_1, \alpha_2$ and the negative root $-\alpha_1$. Let $\Oh_X(1)$ be the restriction to $X$ of the line bundle $\Oh_{\PP^{13}}(1)$; we have $p^*(\Oh_X(1))=\Oh(\omega_2)$, where $$p\colon G/B_1 \twoheadrightarrow X\cong G/P_1$$ is the projection map.

Let $E=\Oh_X(1)^{\oplus 2}$. 
The group $\widetilde G$ (see Notation~\ref{tildeg}) in this case is a split extension of $G$ by $\aut(E/X)\cong GL_{2}(\Co)$ (see Lemma~\ref{tildeGandG}), so it is isomorphic to $G\times GL_{2}(\Co)$. This group acts on $\PP(E^*)$, and the line bundle $\Oh_{\PP(E^*)}(1)$ is $\widetilde G$-equivariant. Moreover, by the Cayley trick~\eqref{equation: Cayley trick} we have 
$$\reg{X,E}\cong\reg{\PP(E^*),\Oh_{\PP(E^*)}(1)},$$ so $\widetilde G$ acts on $\reg{X,E}$. Set $L=\Oh_{\PP(E^*)}(1)$ and $E'=J(L)$. We calculate
the classes $S(e_{\tilde G}(E'),y)$ where $y\in H_*(\PP(E^*),\Z)$.

Let us identify $\PP(E^*)$ with $X\times \PP^1$. Then $L$ is identified with $\Oh_X(1)\boxtimes \Oh_{\PP^1}(1)$, and the action of $\widetilde{G}\cong G\times GL_{2}(\Co)$ on $L$ with the direct product of the action of~$G$ on $\Oh_X(1)$ and the action of $GL_{2}(\Co)$ on $\Oh_{\PP^1}(1)$.

Let $T_2\subset GL_{2}(\Co)$ be the subgroup of diagonal matrices, and set $T=T_G\times T_2\subset \widetilde{G}$. Then $T$ is a maximal torus of $\widetilde{G}$. Let $P_2\subset GL_{2}(\Co)$ be the stabiliser of the point $[1:0]\in \PP^1$; we set $P=P_1\times P_2$ and $B=B_1\times P_2$. Then $P$ (resp. $B$) is a parabolic (resp. Borel) subgroup of $\widetilde{G}$, and $\widetilde{G}/P\cong \PP(E^*)$. With an abuse of notation, we denote by
\begin{equation}\label{equation: G2-grass, projection from borel}
p\colon \widetilde{G}/B \twoheadrightarrow \widetilde{G}/P \cong X\times \PP^1 
\end{equation}
the projection map.

We now identify the rational cohomology of $B\widetilde{G}$ with a subring of $H^*(BT,\Q)$. In this section we denote the elements $\varepsilon_1, \varepsilon_{2}\in \mathfrak{X}(T_2)$ from Notation~\ref{notation: basis for diagonal} by $\zeta_1, \zeta_{2}$ respectively to avoid confusion. We have 
$$H^*(BT,\Q)\cong \Sym(\mathfrak{X}(T_G)\oplus \mathfrak{X}(T_2))\otimes \Q \cong \Q[\varepsilon_1,\varepsilon_2,\varepsilon_{3},\zeta_1,\zeta_2]/(\varepsilon_1+\varepsilon_2+\varepsilon_3).$$ We set 
\begin{align*}
&s_1=\varepsilon_1^2+\varepsilon_2^2+\varepsilon_3^2; \; \; s_2=(\varepsilon_1\varepsilon_2\varepsilon_3)^2; \;\; t_1=\zeta_1+\zeta_2; \; t_2=\zeta_1\zeta_2.
\end{align*}
By~\cite[Example~4.1.17]{GK17}, we have 
$$H^*(B\widetilde{G},\Q)\cong \Q[s_1,s_2, t_1,t_{2}]. $$
With these identifications the map $\beta^*\colon H^*(B\widetilde{G},\Q)\to H^*(BT,\Q)$ is simply the inclusion.

The weight of the $P$-representation which corresponds to the line bundle $L\cong \Oh_X(1)\boxtimes \Oh_{\PP^1}(1)$ is $\omega_2-\zeta_1=\varepsilon_1-\varepsilon_3-\zeta_1$. The cotangent bundle $\Omega_{\PP(E^*)}$ is isomorphic to the direct sum 
$$\Omega_{\PP(E^*)} \cong \pi_1^*\Omega_{G/P_1}\oplus \pi_2^*\Omega_{\PP^1},$$ where $\pi_1\colon \PP(E^*)\to G/P_1$ and $\pi_2\colon \PP(E^*)\to\PP^1$ are the projections. We have $$\Omega_{G/P_1}\cong G\times^{P_1}(\mathfrak{g}/\mathfrak{p}_1)^*,$$
where $\mathfrak{g}$ is the Lie algebra of $G$ and $\mathfrak{p}_1$ is the Lie algebra of the parabolic subgroup $P_1$. Therefore, $\Omega_{G/P_1}$ is obtained from the $P_1$-representation with weights 
$$-\alpha_2=\varepsilon_2-\varepsilon_1, \; \varepsilon_3,\; \varepsilon_3-\varepsilon_1, \; \varepsilon_3-\varepsilon_2, \; -\varepsilon_1 \in \mathfrak{X}(T_G),$$
i.e. the weights are all negative roots except $-\alpha_1$. Similarly, the weight of the $P_2$-representation that induces the line bundle $\Omega_{\PP^1}$ is $\zeta_1-\zeta_2$, see e.g.~\cite[Section~4.2]{GK17}.

So by the exact sequence \eqref{sesjetbundle} the weights of the $P$-representation such that the associated vector bundle over $\widetilde G/P$ is $J(L)$ are
$$\varepsilon_1-\varepsilon_3-\zeta_i,i=1,2, \; \varepsilon_2-\varepsilon_3-\zeta_1, \; \varepsilon_1-\zeta_1,\; -\zeta_1, \;\varepsilon_1-\varepsilon_2-\zeta_1, \;-\varepsilon_3 -\zeta_1$$ and the product of these is the Euler class $$e_{\widetilde G}(p^*(J(L)))\in H^{14}_{\widetilde G}(\widetilde G/B,\Q) \cong H^{14}(BT,\Q),$$
see e.g.~\cite[Lemma~4.1.6]{GK17}.

\begin{lmm}\label{lemma: decomposition g2}
There exists a decomposition 
	\begin{align*}
		e_{\widetilde{G}}(p^*(J(L)))&=-\zeta_1(\varepsilon_1-\varepsilon_3-\zeta_1)(\varepsilon_1-\varepsilon_3-\zeta_2)(\varepsilon_2-\varepsilon_3-\zeta_1)(\varepsilon_1-\zeta_1)\\
		&\times (\varepsilon_1-\varepsilon_2-\zeta_1)(-\varepsilon_3 -\zeta_1) \\
		&=\sum_{i=1}^{2} s_i p_i+\sum_{j=1}^2 t_j q_j \in H^{14}(BT,\Q),
\end{align*}
where
\begin{align*}
	p_1 = &\varepsilon_1\varepsilon_2^3\zeta_1-\varepsilon_1^2\varepsilon_2^2\zeta_1-4\varepsilon_1^3\varepsilon_2\zeta_1-2\varepsilon_1^4\zeta_1 + J \in H^{10}(BT,\Q), \\
	p_2 = &9\zeta_1 \in H^{2}(BT,\Q), \\
	q_1 =& -2\varepsilon_2^5\zeta_1-12\varepsilon_1\varepsilon_2^4\zeta_1-15\varepsilon_1^2\varepsilon_2^3\zeta_1+20\varepsilon_1^3\varepsilon_2^2\zeta_1+45\varepsilon_1^4\varepsilon_2\zeta_1\\&+18\varepsilon_1^5\zeta_1 + J \in H^{12}(BT,\Q), \\
	q_2 =&-2\varepsilon_2^4\zeta_2-8\varepsilon_1\varepsilon_2^3\zeta_2+\varepsilon_1^2\varepsilon_2^2\zeta_2+18\varepsilon_1^3\varepsilon_2\zeta_2+9\varepsilon_1^4\zeta_2-3\varepsilon_2^4\zeta_1-4\varepsilon_1\varepsilon_2^3\zeta_1\\&+28\varepsilon_1^2\varepsilon_2^2\zeta_1+64\varepsilon_1^3\varepsilon_2\zeta_1+32\varepsilon_1^4\zeta_1+2\varepsilon_2^5+10\varepsilon_1\varepsilon_2^4+10\varepsilon_1^2\varepsilon_2^3-20\varepsilon_1^3\varepsilon_2^2\\&-40\varepsilon_1^4\varepsilon_2-16\varepsilon_1^5 + J \in H^{10}(BT,\Q),
\end{align*}
and $J\subset H^*(BT,\Q)$ is the ideal generated by $(\zeta_1^2,\zeta_1\zeta_2,\zeta_2^2)$. 
\end{lmm}

\begin{proof}
The proof is a direct computation in the polynomial ring $H^*(BT,\Q)$, cf. Lemma~\ref{lemma: decomposition lgr}. Again, we originally found this decomposition by using Singular~\cite{Singular}, cf. Remark~\ref{remark: singular lgr}.
\end{proof}

Let us describe the ring homomorphism 
$$\alpha^*\colon H^*(BT,\Q) \to H^*(\widetilde{G}/P,\Q)\cong H^*(X,\Q)\otimes H^*(\PP^1,\Q).$$
Recall that $H^*(\PP^1,\Z)\cong \Z[h]/h^2$, $h=c_1(\Oh_{\PP^1}(1))$. We have 
$$\alpha^*(\zeta_1) = -\alpha^*(\zeta_2) =-c_1(\Oh_{\PP^1}(1))=-h \in H^2(\PP^1,\Q).$$ So, $J\subset \ker(\alpha^*)$. We calculate $\alpha^*(g)$, $g \in H^*(BT_G,\Q) \cong H^*_{G}(G/B_1,\Q)$ by using the generalised Schubert calculus \cite{BGG}. Let $W=N_G(T_G)/T_G$ be the Weyl group of $G$. Let $$\sigma_i\colon \mathfrak{X}(T'_1) \to \mathfrak{X}(T'_1), \; i=1,2$$ be linear transformations given by
$$\sigma_1(\varepsilon_1,\varepsilon_2,\varepsilon_3) = (\varepsilon_1+\varepsilon_2, -\varepsilon_2,\varepsilon_3+\varepsilon_2), \;\;  \sigma_2(\varepsilon_1,\varepsilon_2,\varepsilon_3) = (\varepsilon_2, \varepsilon_1,\varepsilon_3).$$
Then $\sigma_i$ preserves the kernel of~\eqref{equation: G2-grass, weight lattice} and induces the reflection on $\mathfrak{X}(T_G)\otimes\R$ in the hyperplane orthogonal to $\alpha_i, i=1,2$. In particular, we consider $\sigma_1$ and $\sigma_2$ as the generators of the Weyl group $W$. Following~\cite{BGG}, we define operators 
\begin{equation}\label{equation: BGG differential}
A_w\colon \mathrm{Sym}^*(\mathfrak{X}(T_G))\otimes\Q \to \mathrm{Sym}^*(\mathfrak{X}(T_G))\otimes\Q, \; w\in W
\end{equation}
by setting $A_{w_1w_2}= A_{w_1}A_{w_2}$ if $\ell(w_1 w_2)=\ell(w_1)+\ell(w_2)$ and, for $w=\sigma_1,\sigma_2$ and $g\in \mathrm{Sym}^*(\mathfrak{X}(T_G))\otimes\Q$,
\begin{align}
A_{\sigma_1}g(\varepsilon_1,\varepsilon_2,\varepsilon_3)&=\frac{g(\varepsilon_1,\varepsilon_2,\varepsilon_3)-g(\varepsilon_1+\varepsilon_2,-\varepsilon_2,\varepsilon_3)}{\alpha_1},\label{equation: BGG, simple1}\\
A_{\sigma_2}g(\varepsilon_1,\varepsilon_2,\varepsilon_3)&=\frac{g(\varepsilon_1,\varepsilon_2,\varepsilon_3)-g(\varepsilon_2,\varepsilon_1,\varepsilon_3)}{\alpha_2}.\label{equation: BGG, simple2}
\end{align}
Here, $l(w)$ is the length of $w\in W$ with respect to the generators $\sigma_1,\sigma_2$.
\begin{prop}\label{proposition: cohomology of G2 grass}
There is a $\Z$-basis $\{e_w\}_{w\in W}$ of $H_*(G/B_1,\Z)$ such that 
	\begin{enumerate}
	\item $e_w\in H_{2l(w)}(G/B_1,\Z)$;
	\item $\langle \alpha^*(g), e_w\rangle = (A_w g)(0)$ for $g\in \mathrm{Sym}^*(\mathfrak{X}(T_G))\otimes\Q\cong H^*(BT_G,\Q)$;
	\item $p_*(e_w)=0$ if $w=w_1\sigma_1$, $l(w)=l(w_1)+1$, and $p\colon G/B_1 \twoheadrightarrow G/P_1=X$ is the projection.
	\end{enumerate}
\end{prop}
\begin{proof}
See~\cite[Theorem 4.1]{BGG}. 
\end{proof}


\begin{lmm}\label{lemma: cokernels g2}
Let $P^*\subset \fr{*}(\widetilde{G},\Z)$ denote the graded group of primitive elements. Then 
the map
\begin{equation}\label{yetanothers_g2}
	f_l\colon H_{14-2l}(\widetilde{G}/B,\Z)\rightarrow P^{2l-1}, y\mapsto S(e_{\widetilde{G}}(p^*J(L)),y)
\end{equation}
is
\begin{enumerate}
\item the multiplication by $30$ for $l=1$;
\item is given by the matrix $\begin{pmatrix} 0 & 12 & 0 & 0 \\ -24 & -36 & 0 & 0\end{pmatrix}$ for $l=2$;
\item is given by the matrix $\begin{pmatrix} 0 & 0 & -18 \end{pmatrix}$ for $l=6$
\end{enumerate}
for some $\Z$-bases of $H_*(\widetilde{G}/B,\Z)$ and $P^*$. 
Moreover, if $l\neq 1,2, 6$, then $P^{2l-1}=0$.
\end{lmm}

\begin{proof}
By Lemma~\ref{lemma: decomposition g2}, we have
\begin{equation}\label{lkg2grassmannian}
	S(e_{\widetilde{G}}(p^*E'),y) =\sum^{2}_{i=1}\langle\alpha^*(p_i),y\rangle\bar\gamma(s_i) +\sum_{j=1}^{2}\langle \alpha^*(q_j),y\rangle \bar\gamma(t_j)
\end{equation}
for every $y\in H_*(\widetilde{G}/B,\Z)$. The cohomology classes $\bar\gamma(s_i)$, and $\bar\gamma(t_j)$ were calculated in \cite[Example~4.1.17]{GK17} and \cite[Example~4.1.9]{GK17}, respectively. In particular, $\frac{1}{2}\bar\gamma(s_1)$, $\frac{1}{2}\bar\gamma(s_2)$, and $\bar\gamma(t_1)$, $\bar\gamma(t_2)$ form a $\Z$-basis of $P^*$. We deduce the statement from the generalised Schubert calculus described in Proposition~\ref{proposition: cohomology of G2 grass}. For example, let $l=2$, then we span  the homology group $H_{10}(\PP(E^*),\Z)\cong \Z^{\oplus 4}$ on the classes 
$$e_{\sigma_2(\sigma_1\sigma_2)^2},  \; e_{(\sigma_1\sigma_2)^2}\times [\PP^1], \; e_{\sigma_1(\sigma_2\sigma_1)^2}, \; e_{(\sigma_2\sigma_1)^2}\times [\PP^1],$$
and we span $H^3(\widetilde{G},\Z)\cong \Z^{\oplus 2}$ on the elements $\frac{1}{2}\bar\gamma(s_1)$ and $\bar\gamma(t_2)$. Note that $p_*e_{\sigma_1(\sigma_2\sigma_1)^2}=p_*e_{(\sigma_2\sigma_1)^2} =0$. Therefore, by Lemma~\ref{lemma: decomposition g2}, Proposition~\ref{proposition: cohomology of G2 grass}, and the formula~\eqref{lkg2grassmannian}, the linear map $f_2$ is given by the matrix
$$\begin{pmatrix} 2(A_{\sigma_2(\sigma_1\sigma_2)^2} p_1)(0)& 2\langle\alpha^*((A_{(\sigma_1\sigma_2)^2}p_1)(0,0,0,\zeta_1,\zeta_2)),[\PP^1]\rangle & 0 &0\\ (A_{\sigma_2(\sigma_1\sigma_2)^2} q_2)(0) & \langle\alpha^*((A_{(\sigma_1\sigma_2)^2}q_2)(0,0,0,\zeta_1,\zeta_2)),[\PP^1] \rangle&0 &0 \end{pmatrix},$$ 
where $A_w, w\in W$ is the operator~\eqref{equation: BGG differential}, which we compute directly by applying the formulas~\eqref{equation: BGG, simple1} and~\eqref{equation: BGG, simple2}. The other cases are done similarly.
\end{proof}

Let $i_l$ be the order of the cokernel of the map $f_l$, $l\geq 1$. We calculate that $\prod i_l = 2^{7} \cdot 3^5 \cdot 5$, so by Corollary~\ref{corollary: main corollary} and Remark~\ref{remark: main corollary cayley trick}, we conclude the next proposition.

\begin{prop}\label{propositon: bound g2}
For any regular section $s \in \reg{X, E} = \reg{X, \Oh_X(1)^{\oplus 2}}$ the order of the stabiliser $|\widetilde{G}_{s}|$ and the order of $|G_{Z(s)}|$ both divide $2^{7} \cdot 3^5 \cdot 5$, where $G_{Z(s)}$ is the stabiliser of the zero locus $Z(s) \subset X$ under the effective $G$-action.
\end{prop}

\begin{proof}
As in Proposition~\ref{propositon: bound lgr}, it suffices to show that $\reg{X,E}$ is an affine variety. Since $\Oh_{\PP(E^*)}(1)$ is a box product of very ample line bundles, it is enough to check that $\langle c_{6}(e(J(\Oh_{\PP(E^*)(1)})), [\PP(E^*)]\rangle \neq 0$, see Proposition~\ref{proposition: discriminant is codim 1}. One calculates
$$\langle c_{6}(J(\Oh_{\PP(E^*)}(1))), [\PP(E^*)]\rangle = 60. $$
The assertion follows now by Corollary~\ref{corollary: main corollary} and Proposition~\ref{Sum of line bundles}.
\end{proof}

\begin{cor}\label{corollary: bound fano g2}
Let $\mathcal{X}$ be a smooth Fano threefold of Picard rank $1$, index $1$ and genus $10$. Then $|\aut(\mathcal{X})|$ divides $2^{7} \cdot 3^5 \cdot 5$.
\end{cor}

\begin{proof}
The proof is similar to the proof of Corollary~\ref{corollary: bound fano lgr}. We show that $\aut(\mathcal{X}) = G_{Z(s)}$ for some regular section $s$ from $\reg{X, \Oh_X(1)^{\oplus 2}}$. By~\cite[\S12.2]{AG5} or~\cite{Mukai89}, any smooth Fano threefold of genus~$10$ is a linear section of the variety $X \subset \PP^{13}$. So, it suffices to show that every automorphism of $\mathcal{X}$ is induced by an element of $\aut(X)=G$.

By~\cite{Mukai89} and~\cite{Mukai92} (see also~\cite[Theorem~1.1]{BKM24} for a better treatment), there exists a unique (up to isomorphism) stable vector bundle $\mathcal{E}_5$ over $\mathcal{X}$ such that $\rk(\mathcal{E}_5)=5$, $\Lambda^5\mathcal{E}_5=\omega_{\mathcal{X}}$ is the canonical line bundle, $H^*(\mathcal{X},\mathcal{E}_5)=0$, and $\Ext^*(\mathcal{E}_5,\mathcal{E}_5)=0$. Moreover, the dual bundle $\mathcal{E}_5^*$ is globally generated, $\dim \Gamma(\mathcal{X}, \mathcal{E}^*_5)=7$, and the kernel of the natural map
$$\Lambda^4\Gamma(\mathcal{X}, \mathcal{E}_5^*) \to  \Gamma(\mathcal{X}, \Lambda^4 \mathcal{E}_5^*)$$
is one-dimensional and spanned by a non-degenerate $4$-form $\sigma_{\mathcal{X}} \in \Lambda^4\Gamma(\mathcal{X}, \mathcal{E}_5^*)$, i.e. $\sigma_{\mathcal{X}}$ lies in the open orbit under the $GL(\Gamma(\mathcal{X}, \mathcal{E}_5^*))$-action and the stabiliser of $\sigma_{\mathcal{X}}$ is isomorphic to $G$, see~\cite[Proposition~22.12]{FH91}. By~\cite[Proposition~5.3]{Kuz22} for $S=B\aut(\mathcal{X})$ to be the classifying stack of the algebraic group $\aut(\mathcal{X})$, $\mathcal{E}^*_5$ is an $\aut(\mathcal{X})$-equivariant vector bundle. Since $\mathcal{E}^*_5$ is globally generated, it defines an $\aut(\mathcal{X})$-equivariant closed embedding
$$\mathcal{X} \hookrightarrow \Gr(5, \Gamma(\mathcal{X}, \mathcal{E}^*_5))$$
which factors through the closed subvariety $Z$ of $5$-planes isotropic with respect to $\sigma_{\mathcal{X}}$. The variety $Z$ is $G$-homogeneous and, by the construction of $X$ as the closed $G$-orbit in $\PP^{13}$, $Z$ is isomorphic to $X$. Moreover, the image of $\mathcal{X}$ in $\Gr(5, \Gamma(\mathcal{X}, \mathcal{E}^*_5))$ is a transversal section of $Z\cong X$ and a projective subspace of codimension $2$. Therefore, any automorphism of $\mathcal{X}$ can be extended to an automorphism of $X$.
\end{proof}

\begin{rmk}\label{remark: previous estimates g2}
It seems likely that Corollary~\ref{corollary: bound fano g2} also can be obtained more geometrically, cf. Remarks~\ref{remark: previous estimates lgr} and~\ref{remark: previous estimates ogr}. However, we do not know any of such estimates at the moment of writing. 
\end{rmk}

\begin{rmk}\label{remark: genus 12}
Let $\mathcal{X}$ be a smooth Fano threefold of Picard rank~$1$, index~$1$, and genus~$12$. Then, by~\cite{Prokh90} or \cite[Sections~5.3-5.4]{KPS18}, $\aut(\mathcal{X})$ is finite except in the following cases
\begin{enumerate}
\item $\mathcal{X}\cong \mathcal{X}^{MU}$ is the Mukai-Umemura threefold, $\aut(\mathcal{X}^{MU})\cong PGL_2(\Co)$;
\item $\mathcal{X}\cong \mathcal{X}^{a}$, see~\cite[Example~5.3.2]{KPS18}, $\aut(\mathcal{X}^{a})\cong \Z/4 \ltimes \mathbb{G}_a \cong \Z/4 \ltimes \Co$;
\item $\mathcal{X}\cong \mathcal{X}^{m}(u)$, i.e. $\mathcal{X}$ belongs to a $1$-dimensional family of~\cite[Example~5.3.4]{KPS18}, $\aut(\mathcal{X}^{m}(u))\cong \Z/2 \ltimes \mathbb{G}_m \cong \Z/2 \ltimes \Co^{\times}$, $u\in \Co$.
\end{enumerate}
We explain a possible strategy to restrict $|\aut(\mathcal{X})|$ in the remaining cases by using Theorem~\ref{maintheorem}.

We set $X=\Gr(3,7)$, $G=PGL_7(\Co)$, and $E=(\Lambda^2 U^*)^{\oplus 3}$, where $U$ is the tautological bundle. By~\cite[Theorem~3]{Mukai92}, we have $\mathcal{X}\cong Z(s)$ and $\aut(\mathcal{X})\cong G_{Z(s)}$ for some regular section $s\in \reg{X,E}$. Furthermore, $\widetilde{G}_{s} = G_{Z(s)}$, where $\widetilde{G}\cong GL_{3}(\Co)\times G$ is the extended group as in Notation~\ref{tildeg}. Consider the map
\begin{equation}\label{equation: orbit map, genus 12}
O^*\colon H^q(\reg{X,E},\Q) \to H^q(\widetilde{G},\Q)
\end{equation}
induced by the orbit map. By Theorem~\ref{maintheorem}, we calculate that the map~\eqref{equation: orbit map, genus 12} is of rank~$1$ for $q=3$, whereas $\dim H^3(\widetilde{G},\Q)=2$. Therefore, the map $O^*$ is not surjective as expected. However, for $q=1$ and $q\geq 5$, $P^q_{\Q} \subset \im(O^*)$, where $P^*_{\Q} \subset H^*(\widetilde{G},\Q)$ is the graded group of primitive elements. In other words, the map $O^*$ hits all free multiplicative generators of the cohomology ring $H^*(\widetilde{G},\Q)$ except one.

We conjecture that there exists a closed subvariety $\Delta_{\mathrm{inf}} \subset \reg{X,E}$ such that $s\in \Delta_{\mathrm{inf}}$ if and only if $\aut(Z(s))$ is an infinite group, $\codim(\Delta_{\mathrm{inf}})\leq 2$, and the map
\begin{equation}\label{equation: orbit map finite genus 12}
{O'}^*\colon H^*(\reg{X,E}\setminus \Delta_{\mathrm{inf}},\Q) \to H^*(\widetilde{G},\Q)
\end{equation}
induced by the orbit map $O'\colon \widetilde{G} \to \reg{X,E}\setminus \Delta_{\mathrm{inf}}$ \emph{is} surjective. Then the computation of the map~\eqref{equation: orbit map finite genus 12} with integral coefficients will give a restriction for the order $|\aut(\mathcal{X})|$, $\mathcal{X}$ is of genus~$12$ and $|\aut(\mathcal{X})|$ is finite.
\end{rmk}

\begin{rmk}\label{remark: genus 12, hilbert schemes of lines}
In the setting of the previous remark, let $\Sigma(\mathcal{X})$ be the \emph{Hilbert scheme of lines} on $\mathcal{X}$, see~\cite[Section~2.1]{KPS18}. Then it seems likely that the natural map
$$\aut(\mathcal{X}) \to \aut(\Sigma(\mathcal{X}))$$
is injective (cf. Section~5, ibid.), and moreover, we conjecture that $\Sigma(\mathcal{X})$ is a smooth curve of genus~$3$ if and only if the automorphism group $\aut(\mathcal{X})$ is finite (cf.~\cite[Theorem~4.2.7]{AG5}). If the latter is true, then $|\aut(\mathcal{X})|$ must divide $2^5\cdot 3^2 \cdot 7$ provided $\aut(\mathcal{X})$ is finite, see~\cite{KK79}.
\end{rmk}

\section{Weighted projective space}\label{section: weighted}
We first construct the singular \emph{weighted projective space} $X=\PP(1,1,1,2)$ together with its resolution of singularities as $(SL_3(\Co)\times\Co^{\times})$-varieties.

Let $H=SL_3(\Co)$ and let $Y=\PP(\Co^3)$ be the projective plane equipped with the natural $H$-action. We write $\Oh_Y$ for the structure sheaf on $Y$ and $\Oh_Y(1)$ for the $H$-equivariant line bundle dual to the tautological one. We set $\mathcal{E}=\Oh_Y \oplus \Oh_Y(-2)$. Let us denote by $\widetilde{H}$ the fibre product
$$\widetilde{H} = H \times_{\aut(Y)} \aut_Y(\mathcal{E}), $$
see Notation~\ref{tildeg}. Since $\mathcal{E}$ is an $H$-equivariant vector bundle, the group $\widetilde{H}$ is the split extension of $H$ by 
\begin{align*}
	\aut(\mathcal{E}/Y) &\cong \left(\aut(\Oh_Y(-2)/Y) \times \aut(\Oh_Y/Y) \right)\ltimes \Hom(\Oh_Y(-2),\Oh_Y) \\ 
&\cong (\Co^{\times} \times \Co^{\times})\ltimes \Co^{\oplus 6}, 
\end{align*}
where the first component acts on $\Co^{\oplus 6}$ with weight~$2$ and the second one with weight~$1$.

Let $G$ be the subgroup of $\widetilde{H}$ which is the split extension of $H$ by only the first component $\aut(\Oh_Y(-2)/Y)\cong \Co^{\times}$, i.e. $G\cong SL_3(\Co)\times \Co^{\times}$. We note that $G$ is a reductive subgroup of the non-reductive group $\widetilde{H}$ and $G$ is a Levi subgroup of the quotient $\widetilde{H}/Z$ of $\widetilde{H}$ by the central subgroup $Z\cong \Co^\times$ of scalar fibrewise automorphisms.

Let $\PP(\mathcal{E})$ be the projectivisation of $\mathcal{E}$ over $Y$ and let $\Oh_{\PP}(1)$ be the line bundle over $\PP(\mathcal{E})$ dual to the tautological one. Then $\pi_*(\Oh_{\PP}(1)) \cong \mathcal{E}^*,$
where $$\pi\colon \PP(\mathcal{E}) \to Y$$
is the projection map. Since $\mathcal{E}$ is $\widetilde{H}$-equivariant, $\PP(\mathcal{E})$ is an $\widetilde{H}$-variety and $\Oh_{\PP}(1)$ is an $\widetilde{H}$-equivariant line bundle. Note that $\Oh_{\PP}(1)$ is a globally generated (but not ample) line bundle such that 
$$\Gamma(\PP(\mathcal{E}),\Oh_{\PP}(1)) \cong \Gamma(Y,\mathcal{E}^*) \cong \Co \oplus \Sym^2(\Co^3). $$

Consider the regular closed $\widetilde{H}$-equivariant morphism
$$\phi_{\Oh_{\PP}(1)}\colon \PP(\mathcal{E}) \to \PP(\Gamma(\PP(\mathcal{E}),\Oh_{\PP}(1))^*)\cong \PP(\Co\oplus \Sym^2(\Co^3)) \cong \PP^6 $$
defined by the line bundle $\Oh_{\PP}(1)$. Let $\sigma\colon Y \to \PP(\mathcal{E})$ be a unique $\widetilde{H}$-equivariant section of $\pi$ defined by the $\widetilde{H}$-equivariant quotient bundle $\Oh_Y(-2)$ of $\mathcal{E}$. We note that the restriction $\phi_{\Oh_{\PP}(1)}|_{\PP(\mathcal{E})\setminus \sigma(Y)}$ is an isomorphism on its image and $\phi_{\Oh_{\PP}(1)}$ contracts the divisor $\sigma(Y)$ to a point since $\sigma^*(\Oh_{\PP}(1))\cong \Oh_Y$. Let $X$ be the image $\phi_{\Oh_{\PP}(1)}(\PP(\mathcal{E}))$. Then $X$ is the projective cone over the Veronese embedding $Y=\PP^2\subset \PP^5$ or, in other words, $X$ is the \emph{weighted projective space} $\PP(1,1,1,2)$. Moreover, the morphism 
$$\phi_{\Oh_{\PP}(1)} \colon \PP(\mathcal{E}) \to X$$
is the blow-up of $X$ in the singular point. Since $\Oh_{\PP}(1)$ is a $G$-equivariant line bundle, $X$ is a $G$-variety.

\begin{lmm}\label{lemma: levi subgroup, weighted}
There exists an exact sequence
$$1\to K \to G \to \aut(X) $$
such that $K=(\lambda I_3, \lambda^2)$, $\lambda \in \Co$, $\lambda^3=1$ and $G/K$ is a Levi subgroup of the linear algebraic group $\aut(X)$.
\end{lmm}

\begin{proof}
By~e.g.~\cite[Proposition~5.12]{Bunnett24}, there exists an isomorphism $$\aut(X)\cong \left((GL_3(\Co)\times \Co^{\times})\ltimes \Sym^2((\Co^3)^*) \right)/Z$$
such that $Z\cong \Co^{\times}$ is the normal subgroup with the elements $(\lambda I_3, \lambda^2, 0)$, $\lambda\in \Co^{\times}$, and the map $G \to \aut(X)$ is induced by the obvious inclusion of $SL_3(\Co) \times \Co^{\times}$ into $GL_3(\Co)\times \Co^{\times}$. The lemma follows.
\end{proof}

Set $\Oh_{X}(1)=\Oh_{\PP^6}(1)|_X$, $\Oh_{X}(3)=\Oh_X(1)^{\otimes 3}$, $L=\Oh_{\PP}(3)$, and $E'=J(L)$. Then, we have $L\cong \phi_{\Oh_{\PP}(1)}^*(\Oh_X(3))$ and 
$$\reg{\PP(\mathcal{E}), L} \cong \reg{X, \Oh_X(3)}. $$

\begin{lmm}\label{lemma: weighted affine}
The space of regular sections $\reg{X,\Oh_X(3)}$ is an affine variety.
\end{lmm}

\begin{proof}
By~\cite[Example~4.15]{Bunnett24}, the discriminant $$\Delta=\Gamma(X,\Oh_X(3))\setminus \reg{X,\Oh_X(3)}$$ is a union of two irreducible components $\Delta_0$ and $\Delta_1$. More precisely, $s\in \Delta_0$ if and only if $s(x_0)=0$, where $x_0$ is the conical point of $X$, and $\Delta_1$ is the $A$-discriminant (see~\cite[Section~9.1.A]{GKZ08}), i.e. $\Delta_1$ is the closure of the subset in $\Gamma(X,\Oh_X(3))$ formed by the sections $s$ such that the zero locus $Z(s)$ is singular outside the point $x_0\in X$. The assertion follows since $\Delta_0$ is a hyperplane and $\Delta_1$ is of codimension~$1$. 
\end{proof}

\begin{rmk}\label{remark: degree of components, weighted}
The degree $\deg(\Delta_1)$ of the divisor $\Delta_1$ is $212$. The computation requires the fact that $X$ is a toric variety and it can be done by using e.g.~\cite[Corollary~1.6]{MT11} or~\cite[Theorem~1.2]{HS18}.
\end{rmk}

We calculate the classes $S(e_G(E'),y) \in H^*(G,\Z)$, where $y\in H_*(\PP(\mathcal{E}),\Z)$ and $E'=J(L)$.

Let $T_1=H\cap D\subset SL_{3}(\Co)$ be the subgroup of diagonal matrices and set $T_2=\aut(\Oh_Y(-2)/Y)\cong \Co^\times$, $T=T_1\times T_2\subset G$. Then $T$ is a maximal torus of $G$. Let $P_1\subset SL_{3}(\Co)$ be the stabiliser of the point $[1:0:0]\in \PP^2$, and set $P=P_1\times T_2$. Then $P$ is a parabolic subgroup of $G$ and $G/P\cong Y$. 

We now identify the rational cohomology of $BP_1, BP$ and $BG$ with subrings of $H^*(BT,\Q)$. Let $\epsilon_1,\epsilon_2, \epsilon_{3}\in \mathfrak{X}(T_1)$ be the restriction of the elements $\varepsilon_1,\varepsilon_2, \varepsilon_3$ from Notation~\ref{notation: basis for diagonal} to the subgroup of invertible diagonal matrices with determinant~$1$. In this subsection we denote the elements $\varepsilon_1 \in \mathfrak{X}(T_2)$ from Notation~\ref{notation: basis for diagonal} by $\zeta$ to avoid confusion. We have 
$$H^*(BT,\Q)\cong \Sym(\mathfrak{X}(T_1)\oplus \mathfrak{X}(T_2))\otimes \Q \cong \Q[\epsilon_1,\epsilon_2,\epsilon_{3},\zeta]/(\epsilon_1+\epsilon_2+\epsilon_3).$$ We set 
$$u=\epsilon_1, u_i=\sigma_{i}(\epsilon_2,\epsilon_{3}), i=1,2; \;\; s_i=\sigma_i(\epsilon_1,\epsilon_2, \epsilon_{3}), i=1,2, 3.$$
We have then 
\begin{align*}
	&H^*(BP_1,\Q)\cong\Q[u,u_1,u_2]/(u+u_1), \; H^*(BG,\Q)\cong \Q[s_2, s_{3}, \zeta], \\
	&H^*(BP,\Q)\cong \Q[u, u_1,u_2,\zeta]/(u+u_1),
\end{align*}
see~\cite[Section~4.2]{GK17}. We also identify the equivariant cohomology $H^*_G(\PP(\mathcal{E}),\Q)$ with
\begin{align*}
H^*_G(\PP(\mathcal{E}),\Q) &\cong H^*_G(Y,\Q)[c]/(c^2+c_1^G(\mathcal{E})c+c_2^G(\mathcal{E}))\\
	&\cong H^*(BP,\Q)[c]/(u+u_1, c^2+c_1^G(\mathcal{E})c+c_2^G(\mathcal{E}))\\
	&\cong \Q[u, u_1,u_2,\zeta, c]/(u+u_1, c^2+c_1^G(\mathcal{E})c+c_2^G(\mathcal{E})), 
\end{align*}
where $c=c_1^G(\Oh_{\PP}(1)) \in H^2_G(\PP(\mathcal{E}),\Q)$ is the first $G$-equivariant Chern class of the line bundle $\Oh_{\PP}(1)$ and 
\begin{align*}
	c_1^G(\mathcal{E})&= c_1^G(\Oh_Y\oplus \Oh_{Y}(-2)) = \zeta + 2\epsilon_1 \in H^2_G(Y,\Q), \\
	c_2^G(\mathcal{E})& = 0 \in H^4_G(Y,\Q)
\end{align*}
are equivariant Chern classes of the vector bundle $\mathcal{E}$ over $Y$. With these identifications the map 
$$\beta^*\colon H^*(BG,\Q)\to H^*_G(Y,\Q) \cong H^*(BP,\Q)[c]/(c^2+c_1^G(\mathcal{E})c+c_2^G(\mathcal{E}))$$
is simply the inclusion, so $\beta^*(s_2)=u_2+uu_1$ and $\beta^*(s_3)=uu_2$.

The first Chern class $c_1^G(L) \in H^2_G(\PP(\mathcal{E}),\Q)$ of the line bundle $L\cong \Oh_{\PP}(3)$ is $3c$. There is a short exact sequence  
\begin{equation}\label{equation: ses_projbundle}
0 \to \pi^*\Omega_{Y} \to \Omega_{\PP(\mathcal{E})} \to \Omega_{\PP(\mathcal{E})/Y} \to 0,
\end{equation}
where $\Omega_{\PP(\mathcal{E})/Y}$ is the relative cotangent bundle of rank $1$. We know that the weights of the $P_1$-representation that induces $\Omega_{Y}$ are $\epsilon_1-\epsilon_2$ and $\epsilon_1-\epsilon_3$,
see~\cite[Section~4.2]{GK17}. Moreover, by the Euler exact sequence
\begin{equation}\label{equation: euler sequence}
0 \to \Omega_{\PP(\mathcal{E})/Y} \otimes \Oh_{\PP}(1) \to \pi^*(\mathcal{E}^*) \to \Oh_{\PP}(1) \to 0, 
\end{equation}
we deduce that $c_1^G(\Omega_{\PP(\mathcal{E})/Y})= -2c +c_1^G(\mathcal{E}^*)= -2c -2\epsilon_1 - \zeta. $
So by the splitting principle and the exact sequence \eqref{sesjetbundle}, we have 
$$e_{G}(J(L)) = 3c(\epsilon_1-\epsilon_2 + 3c)(\epsilon_1-\epsilon_3 + 3c)(-2\epsilon_1+c -\zeta) \in H^8_{G}(\PP(\mathcal{E}),\Q).$$

\begin{lmm}\label{lemma: decomposition weighted}
There exists a decomposition
\begin{equation*}\label{equation: decomposition weighted}
e_G(J(L))=s_2 p_2+ s_3p_3+ \zeta q\in H^*_{G}(\PP(\mathcal{E}),\Q),
\end{equation*}
where
\begin{align*}
	p_2 &= -240u_1c \in H^{4}_G(\PP(\mathcal{E}),\Q), \\
	p_3 &= -252c \in H^{2}_G(\PP(\mathcal{E}),\Q), \\
	q &= -444u_1^2c - 6u_2c + (\zeta) \in H^{6}_G(\PP(\mathcal{E}),\Q).
\end{align*}
\end{lmm}
\begin{proof}
We find this decomposition using Singular~\cite{Singular}.
\end{proof}

We note that $\alpha^*(\zeta)=0$, where $\alpha^*$ is the ring homomorphism
$$\alpha^*\colon H^*_G(\PP(\mathcal{E}),\Q) \to H^*(\PP(\mathcal{E}),\Q).$$
Then using Theorem \ref{maintheorem} and Lemma~\ref{lemma: decomposition weighted} we have
\begin{align}\label{lkweighted}
	O^*(\Lk(y))&=S(e_G(E'),y)\\
	&= \langle \alpha^*(p_2), y \rangle\bar\gamma(s_2)+ \langle\alpha^*(p_3),y\rangle\bar\gamma(s_3)+ \langle\alpha^*(q),y\rangle\bar\gamma(\zeta) \nonumber
\end{align}
for every $y\in H_*(\PP(\mathcal{E}),\Z)$. The cohomology classes $\bar\gamma(s_2)$, $\bar\gamma(s_3)$, and $\bar\gamma(\zeta)$ were calculated in \cite[Example~4.1.10]{GK17} and \cite[Example~4.1.9]{GK17}, respectively. In particular, these classes are free multiplicative generators of $H^*(G,\Z)$. 

Let us describe the ring homomorphism $\alpha^*$. Let $d=c_1(\Oh_{\PP}(1))\in H^2(\PP(\mathcal{E}),\Z)$ be a (non-equivariant) Chern class of the line bundle $\Oh_{\PP}(1)$. Then $\alpha^*(c)=d$ and we have
$$H^*(\PP(\mathcal{E}),\Q) \cong H^*(Y,\Q)[d]/(d^2 +c_1(\mathcal{E})d+c_2(\mathcal{E})),$$
where $c_i(\mathcal{E})=\alpha^*(c^G_i(\mathcal{E}))$, $i=1,2$ are (non-equivariant) Chern classes of the vector bundle $\mathcal{E}$. Set $h=c_1(\Oh_Y(1))\in H^2(Y,\Z)$; then $H^*(Y,\Z)\cong \Z[h]/h^3$. We have 
$$\alpha^*(u) = -\alpha^*(u_1) =-c_1(\Oh_{\PP^2}(1))=-h \in H^2(\PP^2,\Q)$$
and $\alpha^*(u_2)=h^2$. Therefore, $c_1(\mathcal{E})=-2h$ and $c_2(\mathcal{E})=0$. Finally, we note that the set $\{h^{i}d^{j}\}_{i=0,1,2; j=0,1}$ is a $\Z$-basis in $H^*(\PP(\mathcal{E}),\Z)$ by the Leray-Hirsch theorem.

Using the decomposition of Lemma~\ref{lemma: decomposition weighted}, we obtain the next proposition.
\begin{prop}\label{proposition: cokernels weighted}
Let $P^* \subset \fr{*}(G,\Z)$ denote the graded group of primitive elements. Then the map
\begin{equation}\label{yetanothers_weighted}
f_l\colon H_{8-2l}(\PP(\mathcal{E}),\Z)\rightarrow P^{2l-1}, \;\; y\mapsto S(e_{G}(J(L)),y)
\end{equation}
is 
\begin{enumerate}
\item  the multiplication by $-450$ for $l=1$;

\item is given by the matrix $\begin{pmatrix} -240 & 0 \end{pmatrix}$ for $l=2$;

\item is given by the matrix $\begin{pmatrix} -252 & 0 \end{pmatrix} $ for $l=3$
\end{enumerate}
for some $\Z$-bases of $H_*(\PP(\mathcal{E}),\Z)$ and $P^*$. Moreover, if $l\geq 4$, then $P^{2l-1}=0$. \qed
\end{prop}

Let $i_l$ be the order of the cokernel of the map $f_l$, $l\geq 1$. We calculate that $\prod i_l = 2^{7} \cdot 3^5 \cdot 5^3 \cdot 7$, so by Lemma~\ref{lemma: weighted affine} and Corollary~\ref{corollary: main corollary}, we conclude the next proposition.

\begin{prop}\label{propositon: bound weighted}
For any regular section $s \in \reg{X, \Oh_X(3)} \cong \reg{\PP(\mathcal{E}), L}$ the order of the stabiliser $|G_{s}|$ divides $2^7 \cdot 3^5 \cdot 5^3 \cdot 7$. \qed
\end{prop}

Let $\widetilde{G}$ be the fibre product $\widetilde{G} = G \times_{\aut(\PP(\mathcal{E}))} \aut_{\PP(\mathcal{E})}(L), $
see Notation~\ref{tildeg}. Since $L$ is a $G$-equivariant vector bundle, the group $\widetilde{G}$ is the split extension of~$G$ by $\aut(L/\PP(\mathcal{E}))\cong \Co^\times$. We would like to obtain the bound for the order of the stabiliser group $\widetilde{G}_{s}$ as well. However, since $H^1(G,\Q)\neq 0$, we cannot apply Corollary~3.2.16 from \cite{GK17} directly and we need to make some modifications. 

We calculate the orbit map 
\begin{equation}\label{equation: orbit, first, weighted}
	O^*\colon H^1(\reg{\PP(\mathcal{E}), L},\Z) \to H^1(\widetilde{G},\Z)
\end{equation}
for the action by the extended group $\widetilde{G}$. Set $T_3=\aut(L/\PP(\mathcal{E})) \cong \Co^\times$ be the group of scalar multiplication. We denote the element $\varepsilon_1 \in \mathfrak{X}(T_3)$ from Notation~\ref{notation: basis for diagonal} by~$\xi$ to avoid confusion. Then 
$$H^*(BT_3,\Z) \cong \Sym(\mathfrak{X}(T_3))\cong \Z[\xi],$$
and $\bar\gamma(\xi)$ is the generator of $H^1(T_3,\Z)$. By \cite[Lemma~3.2.15]{GK17}, we obtain
$$O^*(\Lk([\PP(\mathcal{E})])) = \langle \alpha^*(q), [\PP(\mathcal{E})] \rangle \cdot \bar\gamma(\zeta) + \langle c_3(J(L)), [\PP(\mathcal{E})] \rangle \cdot \bar\gamma(\xi).$$
Previously, we showed that $\langle \alpha^*(q), [\PP(\mathcal{E})] \rangle = -450$. By using the exact sequences \eqref{equation: ses_projbundle} and~\eqref{equation: euler sequence}, one can show that $c_3(J(L))=210h^2d \in H^6(\PP(\mathcal{E}),\Z)$, so 
\begin{equation}\label{equition: orbit, fundamental class, weighted}
O^*(\Lk([\PP(\mathcal{E})])) = -450 \bar\gamma(\zeta) + 210 \bar\gamma(\xi).
\end{equation}

Recall that $\sigma\colon Y \to \PP(\mathcal{E})$ is the $\widetilde{H}$-equivariant section of the projection $\pi$ defined by the $\widetilde{H}$-equivariant quotient bundle $\Oh_Y$ of $\mathcal{E}$. Fix a point $y\in Y$ and let $L_{\sigma(y)}$ be the fibre of the line bundle $L = \Oh_{\PP}(3)$ over the point $\sigma(y)$. Note that $\sigma^*(\Oh_{\PP}(1)) \cong \Oh_Y$ with a trivial $T_2$-action. Therefore, the action of the subgroup $T_2\times T_3 \subset \widetilde{G}$ preserves the point $\sigma(y) \in \PP(\mathcal{E})$, and $L_{\sigma(y)}$ is a $1$-dimensional $T_2\times T_3$-representation of weight $$\xi \in \mathfrak{X}(T_2\oplus T_3).$$

Fix a generator $a_1 \in H^1(L^0_{\sigma(y)},\Z)$, where $L^0_{\sigma(y)} = L_{\sigma(y)} \setminus \{0\}$, and let $O_1\colon T_2\times T_3 \to L^0_{\sigma(y)}$ be the orbit map. Then we observe that
\begin{equation}\label{equation: orbit, 1-dim, weighted}
O_1^*(a_1) = \pm \bar\gamma(\xi) \in H^1(T_2\times T_3,\Z).
\end{equation}

Recall that $\phi_{\Oh_{\PP}(1)}(\sigma(y))$ is the singular point of the projective cone $X \subset \PP^6$. Therefore, by~\cite[Example~4.15]{Bunnett24}, for a regular section $$s\in \reg{\PP(\mathcal{E}),L}\cong \reg{X,\Oh_X(3)},$$ we have $s(\sigma(y))\neq 0$, and so the evaluation map
$$ev_{\sigma(y)}\colon \reg{\PP(\mathcal{E}),L} \to L^0_{\sigma(y)} = L_{\sigma(y)}\setminus \{0\} $$
is well-defined. We set $b_1=ev^*_{\sigma(y)}(a_1) \in H^1(\reg{\PP(\mathcal{E}),L},\Z)$.

\begin{prop}\label{proposition: extended group, weighted}
	The orbit map~\eqref{equation: orbit, first, weighted} sends $b_1$ to $\pm \bar\gamma(\xi) \in H^1(\widetilde{G},\Z)$. Moreover, the order $j_1$ of the cokernel of map~\eqref{equation: orbit, first, weighted} divides $\tilde{j}_1=450=2\cdot 3^2\cdot 5^2$.
\end{prop}

\begin{proof}
	The first part follows from the previous discussion and the fact that the following diagram
\begin{equation*}
\begin{tikzcd}
	\reg{\PP(\mathcal{E}),L} \arrow{r}{ev_{\sigma(y)}}& L^0_{\sigma(y)} \\
	G \arrow{u}{O} \arrow[ur, swap, "O_1"]
\end{tikzcd}
\end{equation*}
	commutes (up to a homotopy). For the second part, let $\Lambda \subset H^1(\reg{\PP(\mathcal{E}), L},\Z)$ be the sublattice spanned by $\Lk([\PP(\mathcal{E})])$ and the class $b_1$. By formulas~\eqref{equition: orbit, fundamental class, weighted} and~\eqref{equation: orbit, 1-dim, weighted}, the linear map
	\begin{equation}\label{equation: sublattice, weighted}
	\Lambda \subset H^1(\reg{\PP(\mathcal{E}), L},\Z) \xrightarrow{O^*} H^1(\widetilde{G},\Z)
	\end{equation}
	is given by matrix 
	$$\begin{pmatrix} -450 & 0 \\ 210 & \pm 1 \end{pmatrix}.$$ 
		Therefore, the order $\tilde{j}_1$ of the cokernel of~\eqref{equation: sublattice, weighted} is $450$. Since $j_1$ must divide $\tilde{j}_1$, the assertion follows.
\end{proof}

\begin{rmk}\label{remark: components of discriminant}
By Lemma~\ref{lemma: weighted affine}, $H^1(\reg{X,\Oh_X(3)},\Z)$ is generated by the classes $\Lk_{\Delta_0}$ and $\Lk_{\Delta_1}$ which are Alexander dual to the fundamental classes $[\Delta_0], [\Delta_1] \in H_*(\Delta,\Z)$, respectively. One can check that $a_1=\pm \Lk_{\Delta_0}$ and the linking class $\Lk([\PP(\mathcal{E})])$ is a sum $m \Lk_{\Delta_0} + n \Lk_{\Delta_1}$, $n\neq 0$. At the moment of writing, we do not know how to compute the coefficients $m,n\in\Z$.
\end{rmk}

\begin{cor}\label{corollary: bound extended group, weighted}
For any regular section $$s \in \reg{\PP(\mathcal{E}), L} \cong \reg{X, \Oh_X(3)}$$ the stabiliser $\widetilde{G}_s$ is finite and the order $|\widetilde{G}_{s}|$ divides $2^7 \cdot 3^5 \cdot 5^3 \cdot 7$. Moreover, the order of $|(G/K)_{Z(s)}|$ divides $2^7 \cdot 3^4 \cdot 5^3 \cdot 7$, where $(G/K)_{Z(s)}$ is the stabiliser of the zero locus $Z(s) \subset X$ under the effective $G/K$-action.
\end{cor}

\begin{proof}
By Lemma~\ref{lemma: weighted affine} and Theorem~\ref{thmquotslice}, it suffices to show that the map
$$O^*\colon H^*(\reg{\PP(\mathcal{E}),L},\Q) \to H^*(\widetilde{G},\Q) $$
is surjective and provide an integral class $a\in H^{10}(\reg{\PP(\mathcal{E}),L},\Z)$ such that $O^*(a) \in H^{10}(\widetilde{G},\Z)\cong \Z$ generates a subgroup of index $2^7 \cdot 3^5 \cdot 5^3 \cdot 7$. By Proposition~\ref{proposition: cokernels weighted} and Proposition~\ref{proposition: extended group, weighted}, we can take $a$ to be the cup-product of $\Lk(b_3)$, $\Lk(b_2)$, $\Lk([\PP(\mathcal{E})])$, and $b_1$. Here, $b_3 \in H_2(\PP(\mathcal{E}),\Z)$ (resp. $b_2 \in H_4(\PP(\mathcal{E}),\Z)$) is a homology class such that $f_3(b_3)$ (resp. $f_2(b_2)$) generates a subgroup of index $252$ (resp. $240$). The last part follows from Proposition~\ref{Sum of line bundles} and Lemma~\ref{lemma: levi subgroup, weighted}.
\end{proof}

\begin{cor}\label{corollary: bound fano, weighted}
Let $\mathcal{X}$ be a smooth Fano threefold of Picard rank 1, index 2 and degree 1. Then $|\aut(\mathcal{X})|$ divides $2^8 \cdot 3^4 \cdot 5^3 \cdot 7$.
\end{cor}

\begin{proof}
By e.g.~\cite[Theorem~2.4.5(i)]{AG5} or~\cite[Theorem~II.1.1(iii)]{Isk79}, the anticanonical line bundle $\omega_{\mathcal{X}}^{-1}$ defines the $\aut(\mathcal{X})$-equivariant regular morphism 
$$ \phi\colon \mathcal{X} \to X \subset \PP^6$$
such that $\phi$ is a double cover branched in the disjoint union $B_0=B\cup \{x_0\}$, where $B \subset X$ is a smooth divisor of degree~$3$ and $\{x_0\}$ is the singular point. 
Therefore, $B=Z(s)$ for some regular section $s\in \reg{X,\Oh_X(3)}\cong \reg{\PP(\mathcal{E}),L}.$ By Proposition~\ref{proposition: automorphisms of double cover}, there exists a short exact sequence
$$0\to \Z/2 \to \aut(\mathcal{X}) \to \aut(X)_B=\aut(X)_{Z(s)} \to 1. $$
By~\cite[Lemma~3.1.2]{KPS18}, the group $\aut(X)_{Z(s)}$ is a subgroup of a linear algebraic group which acts faithfully on a smooth algebraic variety $Z(s)$. By the adjunction formula, the canonical bundle of $Z(s)$ is nef; hence, the group $\aut(X)_{Z(s)}$ is finite by Corollary~3.2.2, ibid. Finally, since $G/K$ is a Levi subgroup of $\aut(X)$ (see Lemma~\ref{lemma: levi subgroup, weighted}), we have $(G/K)_{Z(s)}=\aut(X)_{Z(s)}$, and the assertion follows by Corollary~\ref{corollary: bound extended group, weighted}.
\end{proof}

\section{Singular quadric}\label{section: quadric}
We first construct a \emph{singular quadric $X\subset \PP^5$ of rank $5$}  together with its resolution of singularities as $(SO_5(\Co)\times\Co^{\times})$-varieties.

Let $H=SO_{5}(\Co)$ be the special orthogonal group which embedded in $GL_{5}(\Co)$ as the stabiliser of the symmetric bilinear form with matrix 
$$\begin{pmatrix} 0& I_2 & 0\\I_2 & 0 &0 \\ 0 & 0 & 1\end{pmatrix}.$$ 
We take $Y \subset \PP(\Co^5)$ to be the zero divisor of the quadratic form $\Co^5$ preserved by $H$, i.e. $Y$ is a smooth three-dimensional quadric. We write $\Oh_Y$ for the structure sheaf on $Y$ and $\Oh_Y(1)$ for the restriction of $\Oh_{\PP^4}(1)$ to $Y$. We set $\mathcal{E}=\Oh_Y \oplus \Oh_Y(-1)$. Let us denote by $\widetilde{H}$ the fibre product
$$\widetilde{H} = H \times_{\aut(Y)} \aut_Y(\mathcal{E}), $$
see Notation~\ref{tildeg}. Since $\mathcal{E}$ is an $H$-equivariant vector bundle, the group $\widetilde{H}$ is the split extension of $H$ by 
\begin{align*}
\aut(\mathcal{E}/Y) &\cong \left(\aut(\Oh_Y(-1)/Y) \times \aut(\Oh_Y/Y) \right)\ltimes \Hom(\Oh_Y(-1),\Oh_Y) \\ 
&\cong (\Co^{\times} \times \Co^{\times})\ltimes \Co^{\oplus 5}. 
\end{align*}
Let $G$ be the subgroup of $\widetilde{H}$ which is the split extension of $H$ by $\aut(\Oh_Y(-1)/Y)\cong \Co^{\times}$, i.e. $G\cong SO_5(\Co)\times \Co^{\times}$. As in Section~\ref{section: weighted}, $G$ is a reductive subgroup of $\widetilde{H}$ and $G$ is a Levi subgroup of the quotient $\widetilde{H}/Z$ of $\widetilde{H}$ by the subgroup of scalar fibrewise automorphisms.

Let $\PP(\mathcal{E})$ be the projectivisation of $\mathcal{E}$ over $Y$ and let $\Oh_{\PP}(1)$ be the line bundle over $\PP(\mathcal{E})$ dual to the tautological one; then $\pi_*(\Oh_{\PP}(1)) \cong \mathcal{E}^*,$
where $\pi\colon \PP(\mathcal{E}) \to Y$ is the projection map. Since $\mathcal{E}$ is $\widetilde{H}$-equivariant, $\PP(\mathcal{E})$ is an $\widetilde{H}$-variety and $\Oh_{\PP}(1)$ is an $\widetilde{H}$-equivariant line bundle. Note that $\Oh_{\PP}(1)$ is a globally generated (but not ample) line bundle such that 
$$\Gamma(\PP(\mathcal{E}),\Oh_{\PP}(1)) \cong \Gamma(Y,\mathcal{E}^*) \cong \Co \oplus \Co^5. $$

Consider the regular closed $\widetilde{H}$-equivariant morphism
$$\phi_{\Oh_{\PP}(1)}\colon \PP(\mathcal{E}) \to \PP(\Gamma(\PP(\mathcal{E}),\Oh_{\PP}(1))^*)\cong \PP(\Co\oplus \Co^5)\cong \PP^5 $$
defined by the line bundle $\Oh_{\PP}(1)$. Let $\sigma\colon Y \to \PP(\mathcal{E})$ be a unique $\widetilde{H}$-equivariant section of $\pi$ defined by the $\widetilde{H}$-equivariant quotient bundle $\Oh_Y(-1)$ of $\mathcal{E}$. We note that the restriction $\phi_{\Oh_{\PP}(1)}|_{\PP(\mathcal{E})\setminus \sigma(Y)}$ is an isomorphism on its image and $\phi_{\Oh_{\PP}(1)}$ contracts the divisor $\sigma(Y)$ to a point because $\sigma^*(\Oh_{\PP}(1))\cong \Oh_Y$. Let $X$ be the image $\phi_{\Oh_{\PP}(1)}(\PP(\mathcal{E}))$. Then $X$ is the projective cone over the embedding $Y \subset \PP^4$ or, in other words, $X$ is a \emph{singular quadric of rank $5$}. Since $\Oh_{\PP}(1)$ is a $G$-equivariant line bundle, $X$ is a $G$-variety.

\begin{lmm}\label{lemma: levi subgroup, quadric}
The homomorphism $G \to \aut(X) $ is injective and its image is a Levi subgroup of the linear algebraic group $\aut(X)$.
\end{lmm}

\begin{proof}
By the Grothendieck--Lefschetz theorem for Picard groups (see~\cite[Corollaire~XII.3.7]{Gro05} or~\cite[Corollary IV.3.2]{Hartshorne70}), we observe that $\Pic(X)\cong \Z$ and the ample line bundle $\Oh_X(1)$ is a generator. Since $\Gamma(X,\Oh_X(1))\cong \Co^6$, we have
$$\aut(X)\cong PGL_6(\Co)_{X}, $$
i.e. any automorphism of $X$ is induced by the linear transformation of $\Co^6=\Co^5\oplus \Co$. One can show that the only linear transformations that preserve the quadric $X$ are the block matrices
$$\begin{pmatrix} \lambda \cdot A& * \\0 & \mu \end{pmatrix},$$ 
where $A\in SO_5(\Co)$ and $\lambda,\mu\in \Co^{\times}$. Rescaling by $\lambda^{-1}$ identifies the stabiliser modulo scalars with $SO_5(\Co) \times \Co^\times$, and the scalar matrices meet the image of $G$ trivially. 
\end{proof}

Set $\Oh_{X}(1)=\Oh_{\PP^5}(1)|_X$, $\Oh_{X}(3)=\Oh_X(1)^{\otimes 3}$, $L=\Oh_{\PP}(3)$, and $E'=J(L)$. Then, we have $L\cong \phi_{\Oh_{\PP}(1)}^*(\Oh_X(3))$ and 
$$\reg{\PP(\mathcal{E}), L} \cong \reg{X, \Oh_X(3)}. $$

\begin{lmm}\label{lemma: quadric affine}
The space of regular sections $\reg{X,\Oh_X(3)}$ is an affine variety.
\end{lmm}

\begin{proof}
As in Lemma~\ref{lemma: weighted affine}, the discriminant $$\Delta=\Gamma(X,\Oh_X(3))\setminus \reg{X,\Oh_X(3)}$$ is the union of two irreducible components $\Delta_0$ and $\Delta_1$. Again, $s\in \Delta_0$ if and only if $s(x_0)=0$, where $x_0$ is the conical point of $X$, and $\Delta_1$ is the closure of the subset in $\Gamma(X,\Oh_X(3))$ formed by the sections $s$ such that the zero locus $Z(s)$ is singular outside the singular point $x_0\in X$. The assertion follows since $\Delta_0$ is a hyperplane and $\Delta_1$ is of codimension~$1$. 
\end{proof}

\begin{rmk}\label{remark: degree of composnents quadric}
Since $X$ is not a toric variety as opposed to Section~\ref{section: weighted}, we do not know how to compute the degree of the divisor $\Delta_1$. 
\end{rmk}

We calculate the classes $S(e_G(E'),y) \in H^*(G,\Z)$, where $y\in H_*(\PP(\mathcal{E}),\Z)$ and $E'=J(L)$. 

Let $T_1=H\cap D\subset SO_{5}(\Co)$ be the subgroup of diagonal matrices and set $T_2=\aut(\Oh_Y(-1)/Y)\cong \Co^\times$, $T=T_1\times T_2\subset G$. Then $T$ is a maximal torus of $G$. Let $P_1\subset SO_{5}(\Co)$ be the stabiliser of the point $[1:0:0:0:0]\in Y$, and set $P=P_1\times T_2$. Then $P$ is a parabolic subgroup of $G$ and $G/P\cong Y$. 

We now identify the rational cohomology of $BP_1, BP$ and $BG$ with subrings of $H^*(BT,\Q)$. Let $\varepsilon_1,\varepsilon_2 \in \mathfrak{X}(T_1)$ be the same elements as in Notation~\ref{notation: basis for diagonal}; in this subsection we denote the elements $\varepsilon_1 \in \mathfrak{X}(T_2)$ from Notation~\ref{notation: basis for diagonal} by $\zeta$ to avoid confusion. We have 
$$H^*(BT,\Q)\cong \Sym(\mathfrak{X}(T_1)\oplus \mathfrak{X}(T_2))\otimes \Q \cong \Q[\varepsilon_1,\varepsilon_2,\zeta].$$ 
We set 
$$u=\varepsilon_1, u_1=\varepsilon^2_2, \;\; s_1=\varepsilon^2_1 + \varepsilon^2_2, \;\; s_2 = (\varepsilon_1\varepsilon_2)^2.$$
We have then 
\begin{align*}
	&H^*(BP_1,\Q)\cong\Q[u,u_1], \; H^*(BG,\Q)\cong \Q[s_1, s_{2}, \zeta], H^*(BP,\Q)\cong \Q[u, u_1,\zeta],
\end{align*}
see~\cite[Section~4.3]{GK17}. We also identify the equivariant cohomology $H^*_G(\PP(\mathcal{E}),\Q)$ with
\begin{align*}
	H^*_G(\PP(\mathcal{E}),\Q) &\cong H^*_G(Y,\Q)[c]/(c^2+c_1^G(\mathcal{E})c+c_2^G(\mathcal{E}))\\
	&\cong H^*(BP,\Q)[c]/(c^2+c_1^G(\mathcal{E})c+c_2^G(\mathcal{E}))\\
	&\cong \Q[u, u_1, \zeta, c]/(c^2+c_1^G(\mathcal{E})c+c_2^G(\mathcal{E})), 
\end{align*}
where $c=c_1^G(\Oh_{\PP}(1)) \in H^2_G(\PP(\mathcal{E}),\Q)$ and $$c_1^G(\mathcal{E})= c_1^G(\Oh_Y\oplus \Oh_{Y}(-1)) = \zeta + \varepsilon_1, \; c_2^G(\mathcal{E}) = 0$$ are equivariant Chern classes of the vector bundle $\mathcal{E}$ over $Y$. With these identifications the map 
$$\beta^*\colon H^*(BG,\Q)\to H^*_G(Y,\Q) \cong H^*(BP,\Q)[c]/(c^2+c_1^G(\mathcal{E})c+c_2^G(\mathcal{E}))$$
is simply the inclusion, so $\beta^*(s_1)=u_1+u^2$ and $\beta^*(s_2)=u^2u_1$.

The first Chern class $c_1^G(L) \in H^2_G(\PP(\mathcal{E}),\Q)$ of the line bundle $L= \Oh_{\PP}(3)$ is $3c$. 
By \cite[Section~4.2.2]{GK17}, the weights of the $P_1$-representation that induces $\Omega_{Y}$ are $\varepsilon_1$, $\varepsilon_1-\varepsilon_2$, and $\varepsilon_1+\varepsilon_2$. Moreover, by the (analog of) Euler exact sequence~\eqref{equation: euler sequence}, we deduce that $$c_1^G(\Omega_{\PP(\mathcal{E})/Y})= -2c +c_1^G(\mathcal{E}^*)= -2c  - \varepsilon_1 - \zeta, $$
where $\Omega_{\PP(\mathcal{E})/Y}$ is the relative cotangent bundle of rank $1$.
So by the splitting principle and the exact sequences \eqref{equation: ses_projbundle} and \eqref{sesjetbundle}, we have 
\begin{align*}
	e_{G}(J(L)) &= 3c(\varepsilon_1 + 3c)(\varepsilon_1 -\varepsilon_2 + 3c)(\varepsilon_1+\varepsilon_2 + 3c)(-\varepsilon_1+c -\zeta) \in H^{10}_{G}(\PP(\mathcal{E}),\Q).
\end{align*}
\begin{lmm}\label{lemma: decomposition quadric}
There exists a decomposition 
\begin{equation*}e_G(J(L))=s_1 p_1+ s_2p_2+ \zeta q\in H^*_G(\PP(\mathcal{E}),\Q),
\end{equation*}
where
\begin{align*}
	p_1 &= 48u^2c \in H^{6}_G(\PP(\mathcal{E}),\Q), \\
	p_2 &= -60c  \in H^{2}_G(\PP(\mathcal{E}),\Q), \\
	q &= -30uu_1c+264u^3c + (\zeta) \in H^{8}_G(\PP(\mathcal{E}),\Q).
\end{align*}
\end{lmm}
\begin{proof}
This decomposition is found by using~Singular~\cite{Singular}.
\end{proof}
We note that $\alpha^*(\zeta)=0$, where $\alpha^*$ is the ring homomorphism
$$\alpha^*\colon H^*_G(\PP(\mathcal{E}),\Q) \to H^*(\PP(\mathcal{E}),\Q).$$
Then using Theorem \ref{maintheorem} and Lemma~\ref{lemma: decomposition quadric} we have
\begin{align}\label{lkquadric}
	O^*(\Lk(y))&=S(e_G(E'),y)\\
	&=\langle\alpha^*(p_1),y \rangle\bar\gamma(s_1) + \langle \alpha^*(p_2), y\rangle \bar\gamma(s_2)+ \langle \alpha^*(q),y\rangle\bar\gamma(\zeta) \nonumber
\end{align}
for every $y\in H_*(\PP(\mathcal{E}),\Z)$. The cohomology classes $\bar\gamma(s_1)$, $\bar\gamma(s_2)$, and $\bar\gamma(\zeta)$ were calculated in \cite[Example~4.1.13]{GK17} and \cite[Example~4.1.9]{GK17}, respectively. In particular, $\frac{1}{2}\bar\gamma(s_1)$, $\frac{1}{2}\bar\gamma(s_2)$, and $\bar\gamma(\zeta)$ are free multiplicative generators of $H^*(G,\Z)$. 

Let us describe the ring homomorphism $\alpha^*$. Set $d=c_1(\Oh_{\PP}(1))\in H^2(\PP(\mathcal{E}),\Z)$. Then $\alpha^*(c)=d$ and we have
$$H^*(\PP(\mathcal{E}),\Q) \cong H^*(Y,\Q)[d]/(d^2 +c_1(\mathcal{E})d+c_2(\mathcal{E})),$$
where $c_i(\mathcal{E})=\alpha^*(c^G_i(\mathcal{E}))$, $i=1,2$ are (non-equivariant) Chern classes of the vector bundle $\mathcal{E}$. Set $h=c_1(\Oh_Y(1))\in H^2(Y,\Z)$. As in~\cite[Section~4.3]{GK17}, we have 
$$\alpha^*(u) = -c_1(\Oh_{\PP^1}(1))=-h, \;\; \alpha^*(u_1) = -\alpha^*(u^2)=-h^2.$$
Therefore, $c_1(\mathcal{E})=-h$ and $c_2(\mathcal{E})=0$. Finally, by the integral Poincar\'{e} duality and the fact that $\langle h^3,[Y]\rangle =2$, we obtain that the set $\{1, h, \frac{1}{2}h^2, \frac{1}{2}h^3\}$ is a $\Z$-basis in $H^*(Y,\Z)$. By the Leray-Hirsch theorem, the set $$\{2^{-\lfloor \frac{i}{2} \rfloor} h^{i}d^{j}\}_{i=0,1,2,3; j=0,1}$$ is a $\Z$-basis in $H^*(\PP(\mathcal{E}),\Z)$.

Using the decomposition of Lemma~\ref{lemma: decomposition quadric}, we obtain the next proposition.
\begin{prop}\label{proposition: cokernels quadric}
Let $P^* \subset \fr{*}(\widetilde{G},\Z)$ denote the graded group of primitive elements. Then the map
\begin{equation}\label{yetanothers_quadric}
f_l\colon H_{10-2l}(\PP(\mathcal{E}),\Z)\rightarrow P^{2l-1}, \;\; y\mapsto S(e_{G}(J(L)),y)
\end{equation}
is 
\begin{enumerate}
\item  the multiplication by $-588$ for $l=1$;

\item is given by the matrix $\begin{pmatrix} 4\cdot 48 & 0 \end{pmatrix}$ for $l=2$;

\item is given by the matrix $\begin{pmatrix} 2\cdot 60 & 0 \end{pmatrix}$ for $l=4$
\end{enumerate}
for some $\Z$-bases of $H_*(\PP(\mathcal{E}),\Z)$ and $P^*$. Moreover, if $l\neq 1,2,4$, then the group $P^{2l-1}=0$. \qed
\end{prop}

Let $i_l$ be the order of the cokernel of the map $f_l$, $l\geq 1$. We calculate that $\prod i_l = 2^{11} \cdot 3^3 \cdot 5 \cdot 7^2$, so by Lemma~\ref{lemma: quadric affine} and Corollary~\ref{corollary: main corollary}, we conclude the next proposition.

\begin{prop}\label{propositon: bound quadric}
	For any regular section $s \in \reg{X, \Oh_X(3)} \cong \reg{\PP(\mathcal{E}), L}$ the order of the stabiliser $|G_{s}|$ divides $2^{11} \cdot 3^3 \cdot 5 \cdot 7^2$. \qed
\end{prop}

Let $\widetilde{G}$ be the fibre product $\widetilde{G} = G \times_{\aut(\PP(\mathcal{E}))} \aut_{\PP(\mathcal{E})}(L), $
see Notation~\ref{tildeg}. Since $L$ is an $G$-equivariant vector bundle, the group $\widetilde{G}$ is the split extension of $G$ by $\aut(L/\PP(\mathcal{E}))\cong \Co^\times$. As in Section~\ref{section: weighted}, we calculate the orbit map 
\begin{equation}\label{equation: orbit, first, quadric}
	O^*\colon H^1(\reg{\PP(\mathcal{E}), L},\Z) \to H^1(\widetilde{G},\Z)
\end{equation}
for the action by the extended group $\widetilde{G}$. Set $T_3=\aut(L/\PP(\mathcal{E})) \cong \Co^\times$ be the group of scalar multiplication. We denote the element $\varepsilon_1 \in \mathfrak{X}(T_3)$ from Notation~\ref{notation: basis for diagonal} by $\xi$ to avoid confusion. Then 
$$H^*(BT_3,\Z) \cong \Sym(\mathfrak{X}(T_3))\cong \Z[\xi],$$
and $\bar\gamma(\xi)$ is the generator of $H^1(T_3,\Z)$. By \cite[Lemma~3.2.15]{GK17}, we obtain
$$O^*(\Lk([\PP(\mathcal{E})])) = \langle \alpha^*(q), [\PP(\mathcal{E})] \rangle \cdot \bar\gamma(\zeta) + \langle c_4(J(L)), [\PP(\mathcal{E})] \rangle \cdot \bar\gamma(\xi).$$
Previously, we showed that $\langle \alpha^*(q), [\PP(\mathcal{E})] \rangle = -588$. By using the (analogs of) exact sequences~\eqref{equation: ses_projbundle} and~\eqref{equation: euler sequence}, one can show that $c_4(J(L))= 130h^3d \in H^8(\PP(\mathcal{E}),\Z)$, so 
\begin{equation}\label{equition: orbit, fundamental class, quadric}
O^*(\Lk([\PP(\mathcal{E})])) = -588 \bar\gamma(\zeta) + 260 \bar\gamma(\xi).
\end{equation}

The next proposition is the analog of Proposition~\ref{proposition: extended group, weighted}.

\begin{prop}\label{proposition: extended group, quadric}
There exists a homology class $b_1 \in H^1(\reg{\PP(\mathcal{E}),L},\Z)$ such that $$O^*(b_1) =\pm \bar\gamma(\xi) \in H^1(\widetilde{G},\Z).$$ Moreover, the order $j_1$ of the cokernel of map~\eqref{equation: orbit, first, quadric} divides $\tilde{j}_1=588=2^2\cdot 3 \cdot 7^2$.
\end{prop}

\begin{proof}
We construct the class $b_1$ as in the discussion before Proposition~\ref{proposition: extended group, weighted}. For the second part, let $\Lambda \subset H^1(\reg{\PP(\mathcal{E}), L},\Z)$ be the sublattice spanned by $\Lk([\PP(\mathcal{E})])$ and the class $b_1$. By the formula~\eqref{equition: orbit, fundamental class, quadric}, the linear map
\begin{equation}\label{equation: sublattice, quadric}
\Lambda \subset H^1(\reg{\PP(\mathcal{E}), L},\Z) \xrightarrow{O^*} H^1(\widetilde{G},\Z)
\end{equation}
is given by matrix 
$$\begin{pmatrix} -588 & 0 \\ 260 & \pm 1 \end{pmatrix}.$$ 
Therefore, the order $\tilde{j}_1$ of the cokernel of~\eqref{equation: sublattice, quadric} is $588$. Since $j_1$ must divide $\tilde{j}_1$, the assertion follows.
\end{proof}

\begin{rmk}\label{remark: components of discriminant quadric}
As in Remark~\ref{remark: components of discriminant}, we are not able to express the linking class $\Lk([\PP(\mathcal{E})]) \in H^1(\reg{X,\Oh_X(3)},\Z)$ in terms of $\Lk_{\Delta_0}$ and $\Lk_{\Delta_1}$, see Lemma~\ref{lemma: quadric affine}.
\end{rmk}

\begin{cor}\label{corollary: bound extended group, quadric}
For any regular section $$s \in \reg{\PP(\mathcal{E}), L} \cong \reg{X, \Oh_X(3)}$$ the order of the stabilisers $\widetilde{G}_{s}$ and $G_{Z(s)}$ divide $2^{11} \cdot 3^3 \cdot 5 \cdot 7^2$, where $G_{Z(s)}$ is the stabiliser of the zero locus $Z(s) \subset X$ under the effective $G$-action.
\end{cor}

\begin{proof}
We proceed as in Corollary~\ref{corollary: bound extended group, weighted}. By Lemma~\ref{lemma: quadric affine} and Theorem~\ref{thmquotslice}, it suffices to show that the map
$$O^*\colon H^*(\reg{\PP(\mathcal{E}),L},\Q) \to H^*(\widetilde{G},\Q) $$
is surjective and provide an integral class $a\in H^{12}(\reg{\PP(\mathcal{E}),L},\Z)$ such that $O^*(a) \in H^{12}(\widetilde{G},\Z)\cong \Z$ generates a subgroup of index $2^{11} \cdot 3^3 \cdot 5 \cdot 7^2$. By Proposition~\ref{proposition: cokernels quadric} and Proposition~\ref{proposition: extended group, quadric}, we can take $a$ to be the cup-product of $\Lk(b_4)$, $\Lk(b_2)$, $\Lk([\PP(\mathcal{E}])$, and $b_1$. Here, $b_4 \in H_2(\PP(\mathcal{E},\Z)$ (resp. $b_2 \in H_6(\PP(\mathcal{E},\Z)$) is a homology class such that $f_4(b_4)$ (resp. $f_2(b_2)$) generates a subgroup of index $120$ (resp. $192$). The last part follows from Proposition~\ref{Sum of line bundles} and Lemma~\ref{lemma: levi subgroup, quadric}.
\end{proof}

\begin{cor}\label{corollary: bound fano, quadric}
Let $\mathcal{X}$ be a smooth Fano threefold of Picard rank 1, index 1 and genus 4 such that the anticanonical embedding of $\mathcal{X}$ into $\PP^5$ is a complete intersection of a quadric of rank $5$ and a cubic hypersurface. Then $|\aut(\mathcal{X})|$ divides $2^{11} \cdot 3^3 \cdot 5 \cdot 7^2$.
\end{cor}

\begin{proof}
Let $\phi\colon \mathcal{X} \hookrightarrow \PP^5$ be the regular embedding defined by the anticanonical line bundle $\omega^{-1}_{\mathcal{X}}$. By~\cite[Proposition~4.1.12]{AG5}, the image of $\phi$ is a complete intersection in $\PP^5$ of a quadric hypersurface $Q$ and a cubic hypersurface $R$. 
As in Corollary~\ref{corollary: bound fano, quadric, simple}, the quadric $Q$ is uniquely determined by the Fano variety $\mathcal{X}$, and $\aut(\mathcal{X}) = \aut(Q)_{Q\cap R}$. By the assumption, $Q$ is of rank $5$; hence, by Sylvester's law of inertia, $Q$ is isomorphic to the singular quadric $X$ and $Q\cap R\cong Z(s)$ for some regular section $s\in \reg{X,\Oh_X(3)}\cong \reg{\PP(\mathcal{E}),L}$. 

Since $\mathcal{X}$ is a complete intersection of multidegree $(2,3)$ in $\PP^5$, the group $\aut(\mathcal{X})$ is finite, see~\cite[Th\'eor\`em~3.1]{Ben13}. Moreover, since $G$ is a Levi subgroup of $\aut(X)$ (see Lemma~\ref{lemma: levi subgroup, quadric}), we have $G_{Z(s)} = \aut(X)_{Z(s)}$. Finally, the assertion follows by Corollary~\ref{corollary: bound extended group, quadric}.
\end{proof}

\section{Quintic del Pezzo threefold}\label{section: sgm}
We recall the construction of the del Pezzo threefold of degree $5$ as a $SL_2(\Co)$-variety from~\cite[Section~7.1]{CS16}. Set $G=SL_2(\Co)$ and set $V_2$ to be the tautological $2$-dimensional $G$-representation. Let $V_5=\Sym^4(V_2)$ and $\Gr(2,V_5)$ be the Grassmann variety of $2$-planes in $V_5$. Consider the Pl\"{u}cker embedding
$$\Gr(2,V_5) \hookrightarrow \PP(\Lambda^2(V_5)) \cong \PP^9.$$
By the Clebsch--Gordon formula (or~\cite[Lemma~5.5.1]{CS16}), we have a $G$-equivariant splitting
$$\Lambda^2(V_5) \cong \Sym^2(V_2) \oplus \Sym^6(V_2); $$
so, $\PP(\Lambda^2(V_5))$ contains a $G$-invariant projective subspace $\PP(\Sym^6(V_2))$ of codimension $3$. We set
$$X=\Gr(2,V_5)\cap \PP(\Sym^6(V_2)) \subset \PP(\Lambda^2(V_5)). $$
Let us denote by $\Oh_X(1)$ the restriction of $\Oh_{\PP^9}(1)$ to $X$. 

By~\cite[Lemma~7.1.1]{CS16}, $X$ is a smooth algebraic $G$-variety. There is a short exact sequence 
\begin{equation}\label{equation: normal_quintic}
0 \to T_X \to (T_{\Gr(2,V_5)})|_X \to \mathcal{N} \to 0
\end{equation}
of $G$-equivariant vector bundles, where $T_X$ (resp. $T_{\Gr(2,V_5)}$) is the tangent bundle to $X$ (resp. to $\Gr(2,V_5)$) and $\mathcal{N}$ is the normal bundle of $X$ in $\Gr(2,V_5)$. Since~$X$ is a $G$-equivariant regular section of the equivariant vector bundle $\Oh_X(1) \otimes (\Lambda^2(V_5)/\Sym^6(V_2))$ over $\Gr(2,V_5)$, there is an \emph{equivariant} isomorphism
$$\mathcal{N} \cong \Oh_X(1) \otimes (\Lambda^2(V_5)/\Sym^6(V_2)) \cong \Oh_X(1) \otimes \Sym^2(V_2). $$
In particular, $\mathcal{N} \cong \Oh_X(1)^{\oplus 3}$ as a \emph{non-equivariant} vector bundle. 

By the adjunction formula and the Lefschetz theorem on hyperplane sections (see e.g.~\cite[Corollary~3.4.2]{AG5}), $X$ is a Fano threefold of index $2$ and anticanonical degree $40$ with the Picard group generated by $\Oh_X(1)$. In particular, 
$$\omega_X^{-1} \cong \Oh_X(2) = \Oh_X(1)^{\otimes 2},$$
where $\omega_X = \Lambda^3(\Omega_X)$ is the canonical line bundle. We record basic properties of $X$ in the next proposition.

\begin{prop}\label{proposition: basic properties of quintic}
$\quad$
\begin{enumerate}
\item Any smooth Fano threefold of Picard rank $1$, index $2$ and anticanonical degree $40$ is (algebraically) isomorphic to $X$.
\item The automorphism group $\aut(X)$ is isomorphic to the projective linear group $PSL_2(\Co)$ such that the constructed homomorphism $G\to \aut(X)$ identifies with the canonical surjection
$$SL_2(\Co) \twoheadrightarrow SL_2(\Co)/\{\pm I_2\} = PSL_2(\Co). $$
\item There is a ring isomorphism
$$H^*(X,\Z) \cong \Z[h,\lambda]/(h^2-5\lambda, \lambda^2), $$
where $h=c_1(\Oh_X(1)) \in H^2(X,\Z)$ and $\lambda \in H^4(X,\Z)$ is the generator of $H^4(X,\Z)$ such that $\langle h\lambda, [X] \rangle =1$.
\end{enumerate}
\end{prop}

\begin{proof}
We refer the reader to~\cite[Corollary~3.4.2]{AG5} for the first part and to~\cite[Proposition~7.1.10]{CS16} for the second one. By the exact sequence~\eqref{equation: normal_quintic}, we deduce that the topological Euler characteristic $\chi^{\mathrm{top}}(X) = 4$. Thus, by the Lefschetz theorem on hyperplane sections, $H^{\mathrm{odd}}(X,\Z)=0$. Since $\langle h^3, [X] \rangle =5$, we deduce the ring structure on $H^*(X,\Z)$ by the Poincar\'{e} duality.
\end{proof}

Next, we compute the $G$-equivariant cohomology ring $H^*_G(X,\Z)$ of $X$. Let $\widetilde{U}$ (resp. $\widetilde{Q}$) denote the tautological bundle (resp. the universal quotient bundle) over the Grassmannian $\Gr(2,V_5)$, $\rk(\widetilde{U})=2$, $\rk(\widetilde{Q})=3$. We set $U$ (resp. $Q$) to be the restriction of $\widetilde{U}$ (resp. $\widetilde{Q}$) to $X$. Note that both bundles $U$ and $Q$ have unique $G$-equivariant structures and there is a short exact sequence
\begin{equation}\label{equation: sub-quotient ses}
0 \to U \to V_5 \otimes \Oh_X \cong \Sym^4(V_2)\otimes \Oh_X \to Q \to 0
\end{equation}
of $G$-equivariant vector bundles over $X$.

\begin{lmm}\label{lemma: quintic_equivariant cohomology}
$\quad$
\begin{enumerate}
\item $c_1(U) = -h \in H^2(X,\Z)$ and $c_2(U) = 2\lambda \in H^4(X,\Z)$. In particular, $\lambda = (c_1(U))^2 - 2c_2(U)$.
\item Set $u=c_1^G(U) \in H^2_G(X,\Z)$ and $u_1 = (c^G_1(U))^2 - 2c^G_2(U) \in H^4_G(X,\Z)$. Then $H^*_G(X,\Z)$ is a free $H^*(BG,\Z)$-module spanned by $\{1,u,u_1,uu_1\}$. In particular, $H^*_G(X,\Z)$ is generated by $u$ and $u_1$ as an $H^*(BG,\Z)$-algebra.
\end{enumerate}
\end{lmm}

\begin{proof}
For the first part, it suffices to show that $\langle c_1(U)c_2(U), [X]\rangle = -2$, see Proposition~\ref{proposition: basic properties of quintic}. However, since $X$ is a linear section of the Grassmann variety $\Gr(2,V_5)$, we have
\begin{align*}
\langle c_1(U)c_2(U), [X]\rangle &= - \langle (c_1(\widetilde{U}))^4c_2(\widetilde{U}), [\Gr(2,V_5)]\rangle \\
&= -\frac{2}{5} \langle (c_1(\widetilde{U}))^6, [\Gr(2,V_5)]\rangle = -\frac{2}{5}\deg(\Gr(2,V_5))= -2.
\end{align*}

Finally, we note that $\alpha^*(u)=-h$ and $\alpha^*(u_1)=\lambda$. By the Leray--Hirsch theorem applied to the fibre sequence
$$X \xrightarrow{\alpha} X_{hG} \to BG, $$
this implies the last part of the lemma.
\end{proof}

We set $s_2=c_2^G(V_2) \in H^4(BG,\Z)$; $s_2$ is a multiplicative generator of the ring $H^*(BG,\Z)\cong \Z[s_2]$ and $\bar\gamma(s_2)$ is a multiplicative generator of the ring $H^*(G,\Z)\cong \Lambda_{\Z}[\bar\gamma(s_2)]$, see~\cite[Example~4.1.10]{GK17}. We will find the relations between $u$, $u_1$, and $s_2$ in the ring $H^*_G(X,\Z)$. By Proposition~\ref{proposition: basic properties of quintic}, we have
\begin{equation}\label{equation: first relation_quintic}
u^2=5u_1 + as_2 \in H^4_G(X,\Z)
\end{equation}
for some constant $a\in\Z$. By the exact sequence~\eqref{equation: sub-quotient ses}, we have
\begin{equation}\label{equation: main source_quintic}
c^G(U)c^G(Q)=c^G(\Sym^4(V_2)) \in H^*_G(X,\Z).
\end{equation}
By the splitting principle, we calculate
\begin{equation}\label{equation: total chern of trivial}
c^G(\Sym^4(V_2)) = 1 + 20s_2 + 64s_2^2.
\end{equation}
The equations~\eqref{equation: first relation_quintic}, \eqref{equation: main source_quintic}, and~\eqref{equation: total chern of trivial} yield the following identities
\begin{align}
&c^G_1(U)=u, \;\; c^G_2(U)=2u_1+\frac{a}{2}s_2, \label{equation: formulas for equavariant chern_eq1}\\
&c^G_1(Q)=-u, \;\; c_2^G(Q)=3u_1+\frac{40+a}{2}s_2, \;\; c_3^G(Q)=-uu_1-20us_2,\label{equation: formulas for equavariant chern_eq2} 
\end{align}
and the following relations
\begin{equation}\label{equation: prerelation_eq1}
c^G_2(U)c^G_2(Q)+c^G_1(U)c^G_3(Q)=64s^2_2,
\end{equation}
\begin{equation}\label{equation: prerelation_eq2}
c^G_2(U)c^G_3(Q)=0.
\end{equation}
By substituting the identities~\eqref{equation: formulas for equavariant chern_eq1} and~\eqref{equation: formulas for equavariant chern_eq2} into the relation~\eqref{equation: prerelation_eq1}, we obtain
\begin{equation}\label{equation: second prerelation}
u_1^2 + \left( \frac{3}{2}a -60\right)u_1s_2 + \left(\frac{1}{4}a^2 -10 a -64 \right) s_2^2 =0.
\end{equation}
Finally, the equations~\eqref{equation: prerelation_eq2} and~\eqref{equation: second prerelation} imply 
$$\left( \frac{5}{2} a - 160\right) u u_1 s_2 + \left( \frac{1}{2} a^2 - 30 a - 128 \right) u s_2^2 = 0.$$
By Lemma~\ref{lemma: quintic_equivariant cohomology}, the coefficients before $uu_1s_2$ and $us_2^2$ are zeros. Therefore, $a=64$.

\begin{prop}\label{proposition: equivariant cohomology_quintic}
There is an isomorphism
$$H^*_G(X,\Z) \cong \Z[s_2,u,u_1]/\left(u^2-5u_1-64s_2, u_1^2+36u_1s_2+320s_2^2\right) $$
of $H^*(BG,\Z)$-algebras, where $u=c^G_1(U)$ and $u_1=(c^G_1(U))^2 - 2c^G_2(U)$.
\end{prop}

\begin{proof}
We obtain the first (resp. the second) relation from~\eqref{equation: first relation_quintic} (resp. from~\eqref{equation: second prerelation}) by substituting $a=64$. Therefore, we have a surjective map
$$\Z[s_2,u,u_1]/\left(u^2-5u_1-64s_2, u_1^2+36u_1s_2+320s_2^2\right) \twoheadrightarrow H^*_G(X,\Z) $$
of $H^*(BG,\Z)$-algebras. By Lemma~\ref{lemma: quintic_equivariant cohomology}, this map is an isomorphism.
\end{proof}

\begin{prop}\label{proposition: chern classes of cotangent_quintic}
The $G$-equivariant Chern classes of the conormal bundle $\mathcal{N}^*$ and the cotangent bundle $\Omega_X$ are given by the following formulas
$$c^G_1(\mathcal{N}^*) = 3u, \;\; c^G_2(\mathcal{N}^*) = 15u_1 + 196s_2, \;\; c^G_3(\mathcal{N}^*) = 5uu_1+68us_2, $$
$$c^G_1(\Omega_X)=2u, \;\; c^G_2(\Omega_X)=12u_1 +196s_2, \;\; c^G_3(\Omega_X)=4uu_1 +72us_2.$$
\end{prop}

\begin{proof}
We first compute the equivariant Chern classes of the conormal bundle $$\mathcal{N}^*\cong \det(U)\otimes \Sym^2(V_2).$$
Since $c_1^G(\det U)=u$ and $c^G(\Sym^2(V_2)) = 1+4s_2$, we have
$$c^G_1(\mathcal N^*)=3u,\;\; c^G_2(\mathcal N^*)=3u^2+4s_2,\;\; c^G_3(\mathcal N^*)=u^3+4us_2.$$
By Proposition~\ref{proposition: equivariant cohomology_quintic}, $u^2=5u_1+64s_2$ and $u^3=5uu_1+64us_2$, so
$$c^G_2(\mathcal N^*)=15u_1+196s_2,\qquad c^G_3(\mathcal N^*)=5uu_1+68us_2 .$$
 
Finally, we compute the equivariant Chern classes of the cotangent bundle $\Omega_X$. By the splitting principle and~\eqref{equation: formulas for equavariant chern_eq1}, and~\eqref{equation: formulas for equavariant chern_eq2} with $a=64$, we have
$$c^G_1(Q^*\otimes U)=5u,\;\; c^G_2(Q^*\otimes U)=57u_1+776s_2,\;\; c^G_3(Q^*\otimes U)=75uu_1+1120us_2 .$$
By the equivariant short exact sequence~\eqref{equation: normal_quintic}, $$c^G(\Omega_X)=c^G(Q^*\otimes U)\,c^G(\mathcal N^*)^{-1}.$$ Since $\rk\Omega_X=3$, it suffices to compare components of degree $\leq 3$. Solving inductively and using Proposition~\ref{proposition: equivariant cohomology_quintic}, we obtain
\begin{align*}
c^G_1(\Omega_X)&=c^G_1(Q^*\otimes U)-c^G_1(\mathcal N^*)=2u,\\
c^G_2(\Omega_X)&=c^G_2(Q^*\otimes U)-c^G_1(\mathcal N^*)c^G_1(\Omega_X)-c^G_2(\mathcal N^*)\\
&=(57u_1+776s_2)-6u^2-(15u_1+196s_2)=12u_1+196s_2,\\
c^G_3(\Omega_X)&=c^G_3(Q^*\otimes U)-c^G_1(\mathcal N^*)c^G_2(\Omega_X)-c^G_2(\mathcal N^*)c^G_1(\Omega_X)-c^G_3(\mathcal N^*)\\
&=(75uu_1+1120us_2)-3u(12u_1+196s_2)\\
&-2u(15u_1+196s_2)-(5uu_1+68us_2)\\
&=4uu_1+72us_2. \qedhere
\end{align*}
\end{proof}

We take $L=\omega_X^{-1}=\Oh_X(2)$ to be the anticanonical line bundle; $L$~is a $G$-equivariant very ample line bundle. If $s\in \reg{X,L}$ is a regular section, we will find a restriction on the order $|PSL_2(\Co)_{Z(s)}|$, where $PSL_2(\Co)_{Z(s)}$ is the stabiliser group of the zero locus $Z(s) \subset X$ under the effective $PSL_2(\Co)$-action. 

Set $E'=J(L)$ to be the jet bundle of $L$, $\rk(E')=4$.

\begin{cor}\label{corollary: third chern of jet bundle}
$\langle c_3(E'), [X] \rangle = 84$.
\end{cor}

\begin{proof}
By the exact sequence~\eqref{sesjetbundle}, the splitting principle, and Proposition~\ref{proposition: chern classes of cotangent_quintic}, we have
\begin{align*}
c_3(E')&=c_1(L)c_2(\Omega_X\otimes L) +c_3(\Omega_X\otimes L) \\
&= c_1(L)(c_2(\Omega_X)+2c_1(L)c_1(\Omega_X) +3c_1(L)^2) \\
&+ c_3(\Omega_X) + c_2(\Omega_X)c_1(L)+c_1(\Omega_X)c_1(L)^2+ c_1(L)^3 \\
&= 2h(12\lambda+2\cdot 2h\cdot(-2h)+3\cdot(2h)^2) - 4h\lambda + 12\lambda \cdot 2h - 2h \cdot (2h)^2 + (2h)^3
\\
&=84h\lambda. \qedhere 
\end{align*}
\end{proof}

\begin{lmm}\label{lemma: regular section affine quintic}
The variety $\reg{X,L}$ is affine.
\end{lmm}

\begin{proof}
Since $L=\omega_{X}^{-1}$ is a very ample line bundle, it is enough to show that $\langle c_3(J(L)),[X]\rangle \neq 0$, see Proposition~\ref{proposition: discriminant is codim 1} and Corollary~\ref{corollary: third chern of jet bundle}.
\end{proof}

\begin{prop}\label{proposition: equivariant euler_quintic}
$e_G(E')=c^G_4(E')= - 1440 u_1 s_2-23040 s_2^2$.
\end{prop}

\begin{proof}
By the exact sequence~\eqref{sesjetbundle}, the splitting principle, and Proposition~\ref{proposition: chern classes of cotangent_quintic}, we have
\begin{align*}
	c_4(E')&=c_1(L)c_3(\Omega_X\otimes L) \\ 
	&= c_1(L)(c_3(\Omega_X)+c_1(L)c_2(\Omega_X) +c_1(L)^2c_1(\Omega_X)+c_1(L)^3) \\
&= 40 u^2 u_1+640 u^2s_2. 
\end{align*}
Finally, the relations in Proposition~\ref{proposition: equivariant cohomology_quintic} imply the assertion.
\end{proof}

\begin{prop}\label{propositon: bound quintic}
For any regular section $$s \in \reg{X, L} = \reg{X, \Oh_X(2)}$$ the order of the stabiliser $\widetilde{G}_{s}$ divides $2^{7} \cdot 3^3 \cdot 5 \cdot 7$ and the order of $PSL_2(\Co)_{Z(s)}$ divides $2^{6} \cdot 3^3 \cdot 5 \cdot 7$, where $PSL_{2}(\Co)_{Z(s)}$ is the stabiliser of the zero locus $Z(s) \subset X$ under the effective $PSL_{2}(\Co)$-action.
\end{prop}

\begin{proof}
We compute $S(e_G(E'),b) \in H^3(G,\Z)$, where $b \in H_4(X,\Z)$ is a generator. By Proposition~\ref{proposition: equivariant euler_quintic}, we have
\begin{align*}
S(e_G(E'),b) &= - 1440 \langle \alpha^*(u_1), b \rangle \bar\gamma(s_2) = - 1440 \langle \lambda, b \rangle \bar\gamma(s_2) = \pm 1440 \bar\gamma(s_2).
\end{align*}
By Theorem~\ref{maintheorem}, we deduce
\begin{align*}
O^*(\Lk(b)) = S(e_G(E'),b) &= \pm 1440 \bar\gamma(s_2) \in H^3(G,\Z). 
\end{align*}
By Corollary~\ref{corollary: main corollary} and Lemma~\ref{lemma: regular section affine quintic} and since $\bar\gamma(s_2)$ is a generator of $H^3(G,\Z)$, the order $|G_s|$ divides $1440=2^5 \cdot 3^2\cdot 5$ for every regular section $s \in \reg{X,L}$. By Corollary~\ref{corollary: main corollary}, the order $|\widetilde{G}_s|$ divides the previous number times 
$$\langle c_3(E'),[X]\rangle = 84 = 2^2 \cdot 3 \cdot 7.$$ Finally, we restrict the order of the stabiliser of the zero locus using Proposition~\ref{proposition: basic properties of quintic} and Proposition~\ref{Sum of line bundles}.
\end{proof}

\begin{cor}\label{corollary: bound fano special GM}
Let $\mathcal{X}$ be a smooth Fano threefold of Picard rank $1$, index $1$ and genus $6$. Suppose that $\mathcal{X}$ is a double cover of a quintic del Pezzo threefold branched in an anticanonical divisor. Then $|\aut(\mathcal{X})|$ divides $2^{7} \cdot 3^3 \cdot 5 \cdot 7$.
\end{cor}

\begin{proof}
	By~\cite[Appendix~A.2]{DK18}, there exists an $\aut(\mathcal{X})$-equivariant vector bundle $\mathcal{E}_2$ over $\mathcal{X}$ which is simple, globally generated, and $\dim \Gamma(\mathcal{X},\mathcal{E}_2)=5$. By~\cite[p.~33]{DK18}, the bundle $\mathcal{E}_2$ defines a regular map
\begin{equation}\label{equation: gushel map}
\overline{\phi}\colon \mathcal{X} \to \Gr(2, \Gamma(\mathcal{X},\mathcal{E}_2)) 
\end{equation}
such that the image $\mathcal{Y}=\overline{\phi}(\mathcal{X})$ is a transversal intersection of $\Gr(2, \Gamma(\mathcal{X},\mathcal{E}_2))$ with a projective subspace of codimension $3$, and $\mathcal{X}$ is a double cover of $\mathcal{Y}$ branched in a smooth anticanonical divisor $B\subset \mathcal{Y}$. Therefore, by Proposition~\ref{proposition: automorphisms of double cover}, there exists a short exact sequence
$$0 \to \Z/2 \to \aut(\mathcal{X}) \to \aut(\mathcal{Y})_B \to 1. $$
By Proposition~\ref{proposition: basic properties of quintic}, $\mathcal{Y}$ is isomorphic to $X$ such that $B\cong Z(s)$ for some regular section $s\in \reg{X,L}$. Finally, Proposition~\ref{propositon: bound quintic} implies the statement.
\end{proof}

\begin{rmk}\label{remark: ordinary gm fano}
By~\cite[Proposition~3.21(c)]{DK18}, \emph{any} smooth Fano threefold $\mathcal{X}$ of genus~$6$ has finite automorphism group $\aut(\mathcal{X})$. Moreover, again by~\cite[p.~33]{DK18}, if $\mathcal{X}$ is not a double cover of a quintic del Pezzo threefold, then the map~\eqref{equation: gushel map} is a closed embedding and its image is a complete intersection of the Grassmann variety $\Gr(2,5)$, a linear subspace of codimension~$2$, and a quadric. Therefore, if we set $X=\Gr(2,5)$, $G=PSL_5(\Co)$, and $E=\Oh_X(1)^{\oplus 2}\oplus \Oh_X(2)$, then $\im(\overline{\phi}) = Z(s)$ and $\aut(\mathcal{X})=G_{Z(s)}$ for some regular section $s\in \reg{X,E}$. We note that, by~\cite[Section~3.2]{GK17}, the map
$$p_s \colon \widetilde{G}_s \to G_{Z(s)}$$
is surjective, where $\widetilde{G}= G\times_{\aut(X)}\aut_X(E)$ is the extended automorphism group, see Notation~\ref{tildeg}. However, since $\dim H^1(\widetilde{G},\Q)=2$ and $\dim H^1(\reg{X,E},\Q) =1$ (see e.g.~\cite[Proposition~2.2.10]{GK17}), the map
$$O^*\colon H^*(\reg{X,E},\Q) \to H^*(\widetilde{G},\Q)$$
induced by the orbit map $O\colon \widetilde{G}\to \reg{X,E}$ is \emph{not} surjective.
\end{rmk}

\begin{rmk}\label{remark: prime orders genus 6}
By~\cite[Corollary~4.4]{DM22}, there are no smooth Fano threefolds of genus~$6$ admitting an automorphism of prime order $p\geq 13$, and there exists a unique (up to isomorphism) Fano threefold of genus $6$ admitting an automorphism of order $p=11$.
\end{rmk}

\begin{rmk}\label{remark: double epw sextic}
Let $\mathcal{X}$ be a smooth Fano threefold of genus $6$ such that $\mathcal{X}$ is a complete intersection of $\Gr(2,5)$ with a projective subspace of codimension $2$ and a quadric. Let ${\mathcal{Y}}_{\mathcal{X}}$ be the associated \emph{EPW-sextic}, see~\cite[Section~2]{DM22}. By Lemma~2.29 and Corollary~3.11 in~\cite{DK18}, we obtain $$\aut(\mathcal{X})\subset \aut(\mathcal{Y}_{\mathcal{X}}).$$ Let $\widetilde{\mathcal{Y}}_{\mathcal{X}}$ be the \emph{double} EPW-sextic, see~\cite[Appendix~A]{DM22}. Suppose that $\widetilde{\mathcal{Y}}_{\mathcal{X}}$ is smooth. By~\cite[Theorem~1.1(2)]{Ogr06}, $\widetilde{\mathcal{Y}}_{\mathcal{X}}$ is an irreducible symplectic variety which is a deformation of the
symmetric square of a K3-surface. Moreover, in this case, we get
$$\aut(\mathcal{X}) = \aut(\mathcal{Y}_{\mathcal{X}}) = \aut^s(\widetilde{\mathcal{Y}}_{\mathcal{X}}) $$
by~\cite[Proposition~A.2]{DM22}, where $\aut^s(\widetilde{\mathcal{Y}}_{\mathcal{X}})$ is the subgroup of \emph{symplectic} automorphisms. By~\cite[Theorem~1.1]{Mongardi13}, $\aut(\mathcal{X})$ is a subgroup of the Conway sporadic simple group $Co_1$, $$|Co_1|=  2^{21} \cdot 3^{9}\cdot  5^{4} \cdot 7^2 \cdot 11 \cdot 13 \cdot 23.$$ By Corollary~2.13, ibid., we obtain that the order $|\aut(\mathcal{X})|$ divides $$2^{21} \cdot 3^{9}\cdot  5^{4} \cdot 7^2 \cdot 11$$ provided that the double EPW-sextic $\widetilde{\mathcal{Y}}_{\mathcal{X}}$ is smooth. We note that the latter assumption excludes a closed subvariety of $\codim\geq 1$ in the moduli space of smooth Fano threefolds of genus $6$, see Statement~(3) in the introduction to~\cite{Ogr13}. We conjecture that the same restriction on the order of the automorphism group is true for \emph{any} smooth Fano threefold of genus~$6$.
\end{rmk}

\begin{rmk}\label{remark: lagrangian plane genus 6}
Fix a vector space $V_6\cong \Co^6$ of dimension $6$. Let $\LGr(10, \Lambda^3V_6)$ be the Grassmann variety of $10$-planes in the $20$-dimensional vector space $\Lambda^3V_6$ which are isotropic with respect to the natural skew-symmetric form
\begin{equation}\label{equation: skew-symmetric, genus 6}
\Lambda^3V_6 \otimes \Lambda^3V_6\to  \Lambda^6V_6 \cong \Co. 
\end{equation}
Let $\LGr(10, \Lambda^3V_6)_0$ denote the open subvariety of $\LGr(10, \Lambda^3V_6)$ consisting of those $10$-planes $A\subset\Lambda^3V_6$ such that $A$ does not contain decomposable 3-vectors (so there are no subspaces $W\subset V_6$ of dimension $3$ such that $\Lambda^3W \in A$).

In~\cite[Proposition~3.21]{DK18}, O.~Debarre and A.~Kuznetsov associated to each (ordinary) smooth Fano threefold $\mathcal{X}$ of genus $6$ a $10$-plane $A(\mathcal{X}) \in \LGr(10, \Lambda^3V_6)_0$ such that $\aut(\mathcal{X})$ is a subgroup of the stabiliser group $PGL(V_6)_{A(\mathcal{X})}$. We note that the group $PGL(V_6)_A$, $A\in \LGr(10, \Lambda^3V_6)_0$ is always finite, see~\cite[Table~1]{OGr16}. We explain a possible strategy to restrict $|PGL(V_6)_A|$ by using Theorem~\ref{maintheorem}.

Set $X=\Gr(3,V_6)$, $G=PGL(V_6)$, and $E=\Oh_X(1)^{\oplus 10}$. We identify the global sections $\Gamma(X,E)$ with the space of linear maps $\Hom(\Co^{10},\Lambda^3V_6)$. Note that $s\in \Gamma(X,E)$ is regular (or, equivalently, nowhere vanishing, see~\cite[Example~2.2.3]{GK17}) if and only if the corresponding map
$$s\colon \Co^{10} \to \Lambda^3V_6 $$
is injective and its image $\im(s)$ does not contain decomposable vectors. Furthermore, $\widetilde{G}_{s} = PGL(V_6)_{\im(s)}$, where $\widetilde{G}\cong GL_{10}(\Co)\times G$ as in Notation~\ref{tildeg}. However, $$\dim W_4H^3(\reg{X,E},\Q) =1, \;\; \text{and} \;\; \dim H^3(\widetilde{G},\Q)=2,$$
where $W_\bullet H^3(\reg{X,E},\Q)$ is the weight filtration, see~e.g.~\cite[Proposition~2.1.14]{GK17}. Therefore, the map
\begin{equation}\label{equation: orbit map, genus 6}
O^*\colon H^q(\reg{X,E},\Q) \to H^q(\widetilde{G},\Q)
\end{equation}
induced by the orbit map is \emph{not} surjective for $q=3$. Nevertheless, one calculates by Theorem~\ref{maintheorem} (or rather by~\cite[Corollary~2.1.31]{GK17}) that the map~\eqref{equation: orbit map, genus 6} is of rank~$1$ for $q=3$ and $P^q_{\Q} \subset \im(O^*)$ for $q\neq 3$, where $P^*_{\Q} \subset H^*(\widetilde{G},\Q)$ is the graded group of primitive elements.

Let $\langreg{X,E} \subset \reg{X,E}$ be the subset of sections $s\colon \Co^{10} \hookrightarrow \Lambda^3V_6$ such that $\im(s)$ is isotropic with respect to the skew-symmetric form~\eqref{equation: skew-symmetric, genus 6}. We conjecture that $\dim W_4H^3(\langreg{X,E},\Q) \geq 2$ and the map
\begin{equation}\label{equation: orbit map, lagran}
{O'}^*\colon H^3(\langreg{X,E},\Q) \to H^3(\widetilde{G},\Q)
\end{equation}
induced by the orbit map $O'\colon \widetilde{G} \to \langreg{X,E}$ \emph{is} surjective. Then the computation of the map~\eqref{equation: orbit map, lagran} with integral coefficients will give a restriction on the order $|PGL(V_6)_A|$, $A\in \LGr(10,\Lambda^3V_6)_0$.
\end{rmk}

\bibliographystyle{abbrv}
\bibliography{references}

\end{document}